\documentclass[reqno,12pt]{amsart}

\usepackage{graphicx,pdfsync,amssymb, latexsym,amsfonts,amsbsy,color,subcaption,esint,mathtools,enumerate,nicefrac,bm}
\usepackage{hyperref}
\usepackage{booktabs}
\usepackage{multirow}
\usepackage{appendix}
\hypersetup{
colorlinks=true,
linkcolor=blue,
filecolor=blue,
urlcolor=blue,
citecolor=blue,
}

\usepackage{float}
\usepackage{mathrsfs}
\restylefloat{table}

\usepackage{enumitem}
\setitemize{itemsep=0.4em}
\setenumerate{itemsep=0.5em}

\usepackage[utf8]{inputenc}
\usepackage[T1]{fontenc}
\usepackage[top=1.1in, bottom=1in, left=1in, right=1in]{geometry}

\DeclareMathOperator{\curl}{curl}

\newtheorem{theorem}{Theorem}[section]
\newtheorem{lemma}[theorem]{Lemma}
\newtheorem{proposition}[theorem]{Proposition}
\newtheorem{definition}[theorem]{Definition}
\newtheorem{corollary}[theorem]{Corollary}
\newtheorem{remark}[theorem]{Remark}

\newcommand{\thistheoremname}{}
\newtheorem{genericthm}[theorem]{\thistheoremname}

 \newtheorem*{genericthm*}{\thistheoremname}
\newenvironment{namedthm*}[1]
  {\renewcommand{\thistheoremname}{#1}%
   \begin{genericthm*}}
  {\end{genericthm*}}

\usepackage{times}
\usepackage{setspace}
\newcommand{\dd}{\mathop{}\!\mathrm{d}}
\let\del\partial

\newcommand{\RR}{\mathring R}
\newcommand{\tRR}{\mathring {\widetilde{R}}}
\newcommand{\R}{\mathbb R}

\newcommand{\ZZ}{\mathbb Z}
\newcommand{\NN}{\mathbb N}

\newcommand{\ootimes}{\mathbin{\mathring{\otimes}}}

\DeclareMathOperator{\tr}{Tr}
\newcommand{\uin}[0]{u^{\textup{in}}}

\newcommand{\wpq}[1][q+1]{w^{\textup{(p)}}_{#1}}
\newcommand{\wcq}[1][q+1]{w^{\textup{(c)}}_{#1}}

\newcommand{\twpq}[1][q+1]{\widetilde{w}^{\textup{(p)}}_{#1}}
\newcommand{\wtq}[1][q+1]{w^{\textup{(ns)}}_{#1}}
\newcommand{\wttq}[1][q+1]{w^{\textup{(t)}}_{#1}}
\newcommand{\wttqc}[1][q+1]{w^{\textup{(tc)}}_{#1}}
\newcommand{\twcq}[1][q+1]{\widetilde{w}^{\textup{(c)}}_{#1}}
\newcommand{\twtq}[1][q+1]{\widetilde{w}^{\textup{(ns)}}_{#1}}
\newcommand{\twttq}[1][q+1]{\widetilde{w}^{\textup{(t)}}_{#1}}
\newcommand{\twttqc}[1][q+1]{\widetilde{w}^{\textup{(tc)}}_{#1}}

\newcommand{\PP}[0]{\mathbb{P}_{\neq 0}}
\newcommand{\PH}[0]{\mathbb{P}_{H}}
\newcommand{\Ph}{\mathcal{P}_{ \neq 0}}

\newcommand{\wqloc}{w^\textup{loc}_{q+1}}

\newcommand{\TT}[0]{\mathsf{T}}
\newcommand{\TTT}[0]{\mathbb{T}}

\newcommand{\ii}{\textup{i}}

\newcommand{\ve}{v^{\epsilon}}

\newcommand{\Div}{\text{div}}

\newcommand{\uql}{u^{\text{loc}}_q}
\newcommand{\uqnl}{u^{\text{non-loc}}_q}
\newcommand{\tuql}{\widetilde{u}^{\text{loc}}_q}
\newcommand{\tuqnl}{\widetilde{u}^{\text{non-loc}}_q}
\newcommand{\tulql}{\widetilde{u}^{\text{loc}}_{\ell_q}}
\newcommand{\tulqnl}{\widetilde{u}^{\text{non-loc}}_{\ell_q}}
\newcommand{\uqql}{u^{\text{loc}}_{q+1}}
\newcommand{\uqqnl}{u^{\text{non-loc}}_{q+1}}
\newcommand{\tuqql}{\widetilde{u}^{\text{loc}}_{q+1}}
\newcommand{\tuqqnl}{\widetilde{u}^{\text{non-loc}}_{q+1}}

\newcommand{\uq}{u_q}
\newcommand{\ulql}{u^{\text{loc}}_{\ell_q}}
\newcommand{\ulqnl}{u^{\text{non-loc}}_{\ell_q}}

\newcommand{\ublqnl}{\bar{u}^{\text{non-loc}}_{\ell_q}}

\newcommand{\vq}{v_q}
\newcommand{\vlql}{v^{\text{loc}}_{\ell_q}}
\newcommand{\vlqnl}{v^{\text{non-loc}}_{\ell_q}}

\newcommand{\vdq}{v^{\textup{D}}_q}
\newcommand{\Rem}{R^{\text{rem}}_q}
\newcommand{\LinfB}[1][q+1]{\widetilde{L}^\infty_T B^{\frac{1}{2}}_{2,1}}
\newcommand{\LoB}[1][q+1]{{L}^1_t B^{1/2}_{2,1}}
\newcommand{\LoBt}[1][q+1]{\widetilde {L}^1_t B^{5/2}_{2,1}}
\newcommand{\spt}{\text{spt}}

\newcommand{\vql}{v^{\text{loc}}_q}
\newcommand{\vqnl}{v^{\text{non-loc}}_q}
\allowdisplaybreaks[4]

\numberwithin{equation}{section}

\begin{document}

\title{Non-uniqueness of weak solutions to the Navier-Stokes
equations in $\R^3$}

\author{Changxing Miao}
\address[Changxing Miao]{Institute  of Applied Physics and Computational Mathematics, Beijing, China.}

\email{miao\_changxing@iapcm.ac.cn}

\author{Yao Nie}

\address[Yao Nie]{School of Mathematical Sciences and LPMC, Nankai University, Tianjin, China.}

 \email{nieyao@nankai.edu.cn}

\author{Weikui Ye}

\address[Weikui Ye]{School of Mathematical Sciences, South China Normal University, Guangzhou,  China}

 \email{904817751@qq.com}


\date{\today}
\maketitle

\begin{abstract}
To our knowledge, the convex integration method has been widely applied to the study of  non-uniqueness of solutions to the Naiver-Stokes equations in the periodic region, but there are few works {on applying this method} to the corresponding problems in the whole space or other regions. In this paper, we  prove that weak solutions of the Navier-Stokes   equations  are not unique in the class of weak solutions with ﬁnite kinetic energy in the whole space, which extends the  non uniqueness result for the Navier-Stokes equations on torus $\TTT^3$ in  the  groundbreaking work (Buckmaster and Vicol, Ann. of Math., 189 (2019), pp.101-144) to $\R^3$. The critical ingredients of the proof include developing an iterative scheme in which the approximation solution is refined by decomposing it into local and non-local  parts. For the non-local part, we  introduce  the localized corrector {which} plays a crucial role in balancing  the compact support of the Reynolds stress error with the non-compact support of the solution. As applications of this argument, we first prove that  there exist infinitely many weak solutions that dissipate the kinetic energy in smooth bounded domain. Moreover, we {show} the instability of the Navier-Stokes equations  near Couette flow in $L^2(\R^3)$.
\end{abstract}


\section{Introduction}
In this paper, {we consider the incompressible  Navier-Stokes equations in $\R^3$}:
\begin{equation}
\left\{ \begin{alignedat}{-1}
&\del_t u-\Delta u+(u\cdot\nabla) u  +\nabla p   =  0,
 \\
&  \nabla \cdot u  = 0,
\end{alignedat}\right.  \label{NS}\tag{NS}
\end{equation}
 where $u: \R^3\times [0,T]\to\R^3$ denotes the velocity of the incompressible fluid, $p:\R^3\times [0,T]\to\R$ the pressure field. We are focused on the non-uniqueness of weak solutions to \eqref{NS}, where the weak solutions are defined as follows:
 \begin{definition}[Weak solution]\label{def-weak}We say $u\in C([0, T);L^2(\R^3))$ to be a weak
solution of  the Navier-Stokes equation \eqref{NS}  if the vector field  $u(\cdot, t)$ is  weakly divergence-free
for all $t\in [0, T)$ and for all divergence-free test functions $\phi\in C^\infty_0(\R^3\times(0, T))$,
     \begin{align}\nonumber
\int_0^T \int_{\R^3} (\del_t-\Delta)\,\phi u+\nabla \phi : u\otimes u \dd x \dd t=0,
\end{align}where $\nabla \phi : u\otimes u=\partial_i\phi_ju_iu_j$ under the Einstein summation convention.
 \end{definition}
 \subsection{Previous works and our main results}
In the pioneer paper  \cite{L}, Leray proved that for any initial  datum $u_0\in L^2(\R^3)$, there exists a  global weak solutions to the equations \eqref{NS} in $C_w([0, T]; L^2(\R^3)) \cap L^2([0, T]; \dot{H}^1(\R^3)) $ with the energy inequality
$$\|u(t)\|_{L^2}^{2}+2\int_{0}^{t}\|\nabla u(s)\|_{L^2}^{2}{\rm d s}\leq \|u_0\|_{L^2}^{2},\quad \forall \,t\in[0,T].$$
And this result was further developed by Hopf \cite{Hopf} in smooth bounded domain, with Dirichlet boundary conditions. This class of weak solutions is now called as \textit{Leray-Hopf weak solutions}. As discussed in \cite{BV-19, Hopf, JS,  Lady2}, the uniqueness for Leray-Hopf weak solutions of the equations \eqref{NS} remains an important open problem in the theory of the Navier-Stokes
 equations.

 In order to have a clearer understanding of how far we are from solving this challenging problem, many researchers are dedicated to finding sufficient conditions that guarantee the uniqueness of  Leray-Hopf weak solutions. The Serrin type criterions established by \cite{ESS, KS, Lady, Prodi, Serr} shows that the Leray-Hopf weak solution is unique in $L^p_tL^q_x$ with $\frac{2}{p}+\frac{3}{q}\le 1$. In fact, Fabes et al. in \cite{FJR} proved that this Serrin type condition except for $q=3$ also plays the role of a uniqueness criterion  for \textit{very weak solution} of the equations~\eqref{NS}.   The  uniqueness in the  endpoint case $C_tL^3_x$ for {very weak solution} was shown by Furioli et al.  in \cite{FLT}.

 In recent years,  there have been major progresses in the non-uniqueness problems for weak solutions, and various works
  toward non-uniqueness of Leray-Hopf weak solutions.  Jia and  \v{S}ver\'{a}k \cite{JS14, JS} showed the  non-uniqueness of Leray-Hopf weak solutions in $L^\infty_tL^{3,\infty}_x$ provided that a spectral condition holds.  Guillod and \v{S}ver\'{a}k \cite{GS} give convincing numerical evidence of this spectral condition. Recently, Albritton, Bru\'{e} and Colombo \cite{ABC} proved the non-uniqueness of the Leray-Hopf weak solutions for the forced 3D Navier-Stokes equations by adapting properly Vishik's
construction into the cross section of an axisymmetric vortex ring. In addition to approaches based on spectral analysis, many studies have explored the non-uniqueness of weak solutions by methods involving convex integration. In the  breakthrough work \cite{BV}, Buckmaster and Vicol obtained non-unique distributional solutions of the Navier-Stokes
equations in $C_tL^2(\TTT^3)$. They developed a new
convex integration scheme in Sobolev spaces using intermittent Beltrami flows which combined concentrations and oscillations. Furthermore, they jointly with Colombo \cite{BCV} presented the first example of a mild/weak solution to the Navier-Stokes equation on torus whose singular set of times  is  nonempty, and has Hausdorff dimension strictly less than 1.  Luo and Titi \cite{LT} proved the non-uniqueness of weak solutions in $C_tL^2(\TTT^3)$ for
 the Navier–Stokes equations with fractional hyperviscosity $(-\Delta)^{\theta}$, whenever the exponent $(-\Delta)^{\theta}$ is less than $5/4$. For the stochastic 3D Navier-Stokes equations, Hofmanov\'{a}, Zhu and Zhu \cite{zhu1,zhu2,zhu3} {established a series of works on the global existence and non-uniquenes.} Remarkably, Cheskidov and Luo \cite{1Cheskidov} proved the nonuniqueness of very weak solutions in the class $L^p_tL^q(\TTT^3)$ for $1\le p<2$ by exploiting the
temporal intermittency in the convex integrations scheme, and this result implies the sharpness of the Ladyzhenskaya-Prodi-Serrin criteria $\frac{2}{p}+\frac{3}{q}\le 1$ at the endpoint $(p,q)=(2,\infty)$. The approach involving convex integration also has been applied to various other models.  For instance,  the Euler equations   \cite{Buc, BDIS15, BDS, DS17, DRS, DS09, DS14, zhu4, Ise17, Ise18, Ise22} and references therein, the stationary Navier-Stokes equations \cite{Luo},  the transport equations
\cite{BCD, CL21, CL22, MoS, MS}, the MHD equations \cite{2Beekie, LZZ, MY}, and the Boussinesq equations \cite{MNY, TZ17, TZ18}.

The previous works (e.g. \cite{BCV, BV, 1Cheskidov}) on the non-uniqueness of weak solutions for the Navier-Stokes equations within the convex integration scheme are established for the periodic case, and the iterative scheme in these works cannot be directly applied to the Cauchy problem in $\R^3$, etc.

In this paper, we aim to develop a new convex integration scheme to construct non-unique  weak solutions of \eqref{NS} with finite kinetic energy in $\R^3$:
\begin{theorem}[Non-uniqueness of weak solutions]\label{t:main}
Let $T>0$, $e(t),\widetilde{e}(t): [0,T]\to (0,\infty)$ be two nonnegative smooth functions with $e(t)=\widetilde{e}(t)$ for $t\in[0, \tfrac{T}{2}]$. Then there exist  weak solutions $ u, \widetilde{u}  \in C([0,T];L^2(\R^3))$ in the sense of Definition \ref{def-weak} for \eqref{NS} with $u(0,x)=\widetilde{u}(0,x)$ and
$$\|u(t)\|^2_{L^2(\R^3)}=e(t), ~~~\|\widetilde{u}(t)\|^2_{L^2(\R^3)}=\widetilde{e}(t).$$
\end{theorem}
By choosing $e(t)$ and $\widetilde{e}(t) $ such that $e(t)\neq \widetilde{e}(t)$ for $ t> \tfrac{T}{2}$, Theorem  \ref{t:main} immediately shows
that there exist infinitely many weak solutions in $C_tL^2(\R^3)$ for \eqref{NS} with the same initial data. Particularly, weak solutions that dissipate the kinetic energy in $C_tL^2(\R^3)$  are not unique via choosing monotonically non-increasing functions
$e(t)$ and $\widetilde{e}(t)$. So we immediately have the following corollary.
\begin{corollary}\label{t:main0}
There exist infinitely many weak solutions $u\in C([0,T];L^2(\R^3))$ that dissipate the kinetic energy for the system  \eqref{NS}.
\end{corollary}
\subsection{Main ideas}  The proof of Theorem \ref{t:main} draws upon several  fundamental ideas  introduced by the work  \cite{BV} of   Buckmaster and Vicol, which was established in the periodic setting.  In constructing nonperiodic solutions in $\R^3$, we cannot proceed with the iterative scheme as in \cite{BV}, since we have to confront the issue that the Reynolds stress may not maintain the property of spatial compactly support, which did not arise in the previous work in the periodic setting. However, due to the limitations of current geometric lemma, it seems necessary for the Reynolds error to be with compact support  at each iteration step. This compels us to develop a new iterative scheme which enables us to control the support of the Reynolds stress  within the context of the whole space $\R^3$.

To ensure that the Reynolds tensor  has both compact support and divergence form at the next step, we cannot use the the inverse of the divergence operator due to the fact that a function with compact support, when acted upon by the inverse of the divergence operator,  may not retain compact support in the whole space. {Another difficulty comes from the fact that the absence of compact support in the approximate solution hinders the maintenance of the compact support property for the Reynolds tensor at the next step.} In this paper, we develop a new iterative scheme to ensure the compatibility between the Reynolds stress error with compact support, divergence-form, and the non-compact support property of the approximate solution  at each step.

To design a suitable iterative scheme, we divide the approximation solutions into local, and non-local parts, and introduce a \textit{localized corrector} for balancing the non-compact support and non-divergence form parts of the new tensor error. For the local and non-local parts, we construct different perturbations to proceed with the iteration.
\begin{enumerate}
    \item For given $u^{\text{loc}}_q$ with compact support,  the local part of the approximate solution $u_q$, we construct the {perturbation} $\wqloc$ with devergence-free and compact-support, based on the intermittent building blocks in  \cite{BV, MY} such that $u^{\text{loc}}_{q+1}=u^{\text{loc}}_{\ell_q}+\wqloc$. More precisely,  $\wqloc$ is composed by  four parts:

    \begin{enumerate}
     \item The principle perturbation $\wpq$, which is constructed by using the box flows as in ~\cite {MY};
     \item The incompressibility corrector $\wcq$, which is introduced to correct principal perturbation $\wpq$ to satisfy the incompressibility condition;
    \item The temporal corrector $\wttq$. We design it for cancelling  the extra errors with the traveling-wave property.
To ensure that  $\wttq$ has compact support, its construction is slightly different from that in \cite{BV, MY} and it is not divergence-free.
\item The incompressibility corrector $\wttqc$. It is given to correct  $\wttq$ to enforce divergence-free condition.
    \end{enumerate}

\item
For given $\uqnl$ without compact support, the remaining part of $u_q$, we introduce a \textit{localized corrector} $\wtq$ such that {$u^{\text{non-loc}}_{q+1}=u^{\text{non-loc}}_{\ell_q}+\wtq$}. As previously analyzed, without using the divergence inverse operator and the non-compact support property of the approximate solution {give rise to} some new tensor errors with non-divergence and non-compact support. Benefiting from the high frequency of flows, as shown in Lemma~\ref{tracefree}, one could expect that these  errors are small. This observation inspires us to treat these errors as  a forcing terms  of a incompressible Navier-Stokes equations (see \eqref{e:wt} below), and to define  the localized corrector $\wtq$  to be the solution of such equations. Subsequently, the localized corrector $\wtq$ cancels these errors with non-divergence or non-compact support. This leads to the new Reynolds tensor error  maintaining a compact support, allowing the iteration to proceed successfully.
\end{enumerate}

To our knowledge,  this is the first work of  constructing a non-periodic weak solution to the Navier-Stokes equations through convex integration  method. The new iterative scheme developed in this paper can be also applied for other related issues to the Navier-Stokes equations or other viscous models. 

\subsection{Applications to other problems}
As a matter of fact, the main idea of developing a new iterative scheme via refining approximation solutions is applicable not only to the whole space, but also to general domains. For instance, we present the non-uniqueness result of weak solutions  for the Navier-Stokes equations in smooth bounded domain. Furthermore, we also present a result of instability near the Couette flow.


We are in position to state our two applications.
\subsubsection{Application \uppercase\expandafter{\romannumeral1}: The non-uniqueness result in smooth bounded domain } We consider the incompressible Navier-Stokes equation in smooth  bounded domain $\mathcal{D}$:
 \begin{equation}
\left\{ \begin{alignedat}{-1}
&\del_t v-\Delta v+(v\cdot\nabla) v  +\nabla p   =  0,\quad{\rm in}\quad {\mathcal{D}}\times(0,T],
 \\
&  \nabla \cdot v  = 0, \quad\qquad\qquad \qquad \qquad\,\,\,\, \, {\rm in}\quad  {\mathcal{D}}\times(0,T],\\
& v=0,  \quad\qquad\qquad \qquad \qquad \quad \, \, \,\, \,  \, \, \, \,  {\rm on}\quad  {\partial\mathcal{D}}\times[0,T].
\end{alignedat}\right.  \label{DNS}
\end{equation}

\begin{theorem}\label{dns}There exist infinitely many weak solutions in $C([0,T];L^2(\mathcal{D}))$ that dissipate the kinetic energy for the \eqref{DNS}.
\end{theorem}
Here a vector field $v\in C([0,T];L^2(\mathcal{D}))$ is called a weak solution to  \eqref{DNS} if it solves  \eqref{DNS}  in the sense of distribution. {It seems to be the first result on  constructing}  non-unique weak solutions of the Navier-Stokes equations in a bounded non-periodic domain via  the convex integration method.
The proof is provided in Section \ref{proof of application}. The readers will see that the iterative scheme used in the proof differs slightly from the previous one in the proof of Theorem \ref{t:main}, {and thereby} this iterative scheme has a certain degree of flexibility and allows us to address different issues.

\subsubsection{Application \uppercase\expandafter{\romannumeral2}: The instability of the system \eqref{NS} near Couette flow }  Let $U=(x_2,0,0)\in\R^3$ be  the 3D  Couette flow. Employing the iterative scheme in this paper {enables us to obtain} the following instability result for the incompressible Navier-Stokes equations \eqref{NS} near the Couette flow  in $L^2(\R^3)$.

\begin{theorem}\label{Couette}
Given $0<\epsilon\ll 1$,  there exists a weak solution $u^{\epsilon}$ of the system \eqref{NS} such that
\[\|u^{\epsilon}(0,\cdot)-U(\cdot)\|_{L^2(\R^3)}\leq \epsilon,\]
meanwhile, for every  $3\epsilon^{1/2}\le t\le 5\epsilon^{1/2}$,
\[\|u^{\epsilon}(t,\cdot)-U(\cdot)\|_{L^2(\R^3)}\geq \epsilon^{-1/2}.\]
\end{theorem}
Hydrodynamic stability at high Reynolds number has been a very active ﬁeld in the ﬂuid mechanics. There is a series of work on the stability threshold problem, e.g. \cite{BGM17, BGM20, BGM22, CWZ, DM23, WZ} for the 3D Couette ﬂow. As a simple application of the convex integration method, we present an instability result for the \eqref{NS}, which corresponds to the ill-posedness of the perturbation  system near the 3D Couette ﬂow in $L^2(\R^3)$.

\subsection{Organization of the paper} In Section \ref{Induction}, we give our induction scheme and the two iterative proposition, Proposition
\ref{iteration} and  Proposition \ref{p:main-prop2} below, which immediately show Theorem~\ref{t:main}. Section \ref{proof-1} and  Section \ref{proof-2} present the detailed proofs of  Proposition
\ref{iteration} and  Proposition \ref{p:main-prop2}, respectively.
In Section \ref{proof of application}, we utilize the iterative scheme in this paper to prove Theorems \ref{dns} and Theorems \ref{Couette}. Appendix contains some technical tools such as the geometric Lemma, an improved Hölder inequality, the definitions of mollifiers and Lerner-Chemin space $\widetilde L^{\infty}_tB^{1/2}_{2,1}$, and the inverse divergence iteration step.

\noindent {\bf{Notations}}\, For a $\TTT^3$-periodic function $f$, we denote
\begin{align*}
\mathbb{P}_{=0} f:=\frac{1}{|\TTT^3|}\int_{\TTT^3} f(x)\dd x,\quad\PP f=f-\mathbb{P}_{=0} f\quad \text{and}\quad u\ootimes v:=u\otimes v-\frac{1}{3}\tr(u\otimes v).
\end{align*}
In the following, the notation $x\lesssim y$ means $x\le Cy$ for a universal constant that may change from line to line. We use the symbol $\lesssim_N$ to express that the constant in the inequality depends on the parameter $N$.  {Without ambiguity, we will denote $L^m([0,T];Y(\R^3))$ and $L^m([0,T];L^m(\R^3))$ by $L^m_t Y$ and $L^m_{t,x}$, respectively.} We also denote the Lebesgue space, Sobolev space, Hilbert space and Besov space in $\R^3$ by $L^p,~W^{s,p},~H^s$ and $B^{s}_{p,r}$, respectively.

\section{Induction scheme}\label{Induction}
In this section, we give the induction scheme and the  two  iterative  propositions, that is Proposition \ref{iteration} and Proposition \ref{p:main-prop2} below, which enable us to  prove Theorem \ref{t:main}.
\subsection{Parameters}\label{para}First of all, we introduce several parameters.
Let
\begin{align}\label{b-beta}
b=2^{15}, \quad\beta=b^{-4},  \quad 0<\alpha\leq b^{-7}.
\end{align}
Suppose that $a\in \NN$ is a large number depending on $b,\beta,\alpha$ and the initial data $\uin$. We define
\begin{align}\label{def-lq}
    \lambda_q :=   a^{b^q},  \quad \delta_q :=  \lambda_q^{-2\beta}, \quad \ell_q:=\lambda^{-60}_q,\quad q\ge 0,
\end{align}
and
\begin{align}\label{omega}
\Omega_{q}:=\big[-\tfrac{1}{2} +\lambda_q^{-\alpha},~
\tfrac{1}{2} -\lambda_q^{-\alpha}  \big]^3   ,~~~~q\geq1.
\end{align}
For given $e(t)$ and $\widetilde{e}(t)$ in Theorem \ref{t:main}, there exists an appropriate small constant $c$ such that $c\delta_2<e(t),\widetilde{e}(t)<\frac{3}{c}\delta_2$ for $t\in [0, T]$. Without loss of generality, we set $c=1$.

\subsection{Iterative procedure }\label{sec-ite}As is usual in convex integration schemes, we consider the following relaxation of the Navier-Stokes equations
 \begin{equation}\label{NSR}\tag{NSR}
\left\{ \begin{alignedat}{-1}
   & \del_t u_q-\Delta u_q+\Div (u_q\otimes u_q) +\nabla p_q   =\Div\RR_q,
 \\
  &\nabla \cdot u_q = 0, \\
\end{alignedat}\right.
\end{equation}
where the \emph{Reynolds stress} $\RR_q$  is a symmetric trace-free $3\times3$  matrix.

To employ induction, we suppose that the solution $(u_q, p_q, \RR_q)$ of the equations \eqref{NSR} on $(0, T]\times \R^3$ satisfies the following conditions:
\begin{align}
&\uq=\uql+\uqnl,
    \label{uq-tigh}\\
&\|\uql\|_{L^{\infty}_tL^2}\leq\sum_{j=0}^{q-1}C_0\delta^{1/2}_{j+1},\qquad\|\uqnl\|_{\widetilde L^{\infty}_tB^{1/2}_{2,1} } \le \sum_{j=0}^{q-1}\lambda^{-1}_{j}, \label{e:vq-C0}\\
&\|(\uql,~\uqnl)\|_{L^{\infty}_{t}H^5} \leq \lambda^{10}_q, \label{e:vq-H5}
    \\
    &\| \RR_q \|_{L^{\infty}_{t}L^1 }  \le \delta_{q+1}\lambda_q^{-4 \alpha}, \qquad\| \RR_q \|_{L^{\infty}_{t}W^{5,1} }  \le  \lambda_q^{10},
    \label{e:RR_q-C0}
       \\
     &{\spt_{x}} \uql, {\spt_{x}} \RR_q\subseteq \Omega_q   ,
    \label{e:RR_q-tigh}\\
    &\delta_{q+1}\leq e(t)-\int_{\mathbb{R}^3}|u_q|^2\dd x \leq 3\delta_{q+1}, \quad 0\le t\le T.\label{et}
\end{align}
Here $C_0$ is a universal large enough number. The following proposition shows that there exists a solution  $(u_{q+1}, p_{q+1}, \RR_{q+1})$ of the equations~\eqref{NSR}
satisfying the above inductive conditions \eqref{uq-tigh}--\eqref{et} with $q$ replaced by $q+1$, which guarantees the iteration proceeds successfully.

{\begin{proposition}\label{iteration}
Assume that $(u_q,p_q,\RR_q )$ solves
\eqref{NSR} and satisfies  \eqref{uq-tigh}--\eqref{et},
then there exists a solution $ (u_{q+1},  p_{q+1}, \RR_{q+1} )$, satisfying \eqref{uq-tigh}--\eqref{et} with $q$ replaced by $q+1$, and such that
\begin{align}
        \|u_{q+1} - u_q\|_{L^\infty_tL^2} &\leq  C_0\delta_{q+1}^{1/2}.
        \label{e:velocity-diff}
\end{align}
\end{proposition}}
Moreover, we will show the following proposition.
\begin{proposition}\label{p:main-prop2}Let $T>0$ and $e(t)=\widetilde{e}(t)$ for $t\in[0, \tfrac{T}{4}+\lambda^{-1}_1  ]$. Suppose that  $(u_q ,p_q,\RR_q)$ solves
\eqref{NSR} and satisfies  \eqref{uq-tigh}--\eqref{et},    $(\widetilde{u}_q, \widetilde{p}_q, \widetilde{\RR}_q)$  solves
\eqref{NSR} and satisfies  \eqref{uq-tigh}--\eqref{et} with $e(t)$ replaced by $\widetilde{e}(t)$. Then   if
\begin{align*}
 \uql=\tuql, \quad\uqnl=\tuqnl, \quad \RR_{q } ={\tRR}_{q },\quad\text{on}\quad [0, \tfrac{T}{4}+\lambda^{-1}_q],
\end{align*}
then we have
\begin{align*}
 \uqql=\tuqql, \quad\uqqnl=\tuqqnl, \quad \RR_{q+1 } ={\tRR}_{q+1},\quad\text{on}\quad [0, \tfrac{T}{4}+\lambda^{-1}_{q+1}].
\end{align*}
\end{proposition}
\subsection{Proof of Theorem \ref{t:main}}\label{P-to-T}For given $e(t)$ and $\widetilde{e}(t)$, let $(u_1, p_1, \RR_1)=(0,0,0)$. One easily verifies that $(u_1, p_1, \RR_1)$ satisfies \eqref{NSR} and \eqref{uq-tigh}--\eqref{et}, as well as \eqref{et} with $e(t)$ replaced by  $\widetilde{e}(t)$. Utilizing Proposition \ref{iteration} inductively, we obtain a sequence of solutions $\{(u_q, p_q,\RR_q)\}$ and $\{(\widetilde{u}_q, \widetilde{p}_q, \widetilde{\RR}_q)\}$ of the equations \eqref{NSR} respectively satisfying  estimates \eqref{uq-tigh}--\eqref{e:velocity-diff}, and \eqref{uq-tigh}--\eqref{e:velocity-diff} with $e(t)$ replaced by  $\widetilde{e}(t)$. By the definition of $\delta_q$, one can easily deduce that $\sum_{i=2}^{\infty}\delta^{1/2}_{i}$ converges to a finite number. This fact combined with \eqref{e:velocity-diff} implies that $\{u_q\}$ and $\{\widetilde{u}_q\}$ are two  {Cauchy sequences in $C_tL^2( \R^3) $}. Since  $\|(\RR_q, {\tRR}_q)\|_{L^\infty_t L^1}\rightarrow 0$ as $q\to\infty$,  the limit functions $ u$ and  $ \widetilde{u}$  are weak solutions of the Navier-Stokes equation  \eqref{NS} and satisfies
    \begin{align}
    &\int_{\R^3}| u(t,\cdot)|^2\dd x=e(t),~~~~t\in[0,T], \label{u-e}\\
    &\int_{\R^3}|\widetilde{u}(t,\cdot)|^2\dd x=\widetilde{e}(t),~~~~t\in[0,T].\label{tu-te}
    \end{align}
Since $e(t)=\widetilde{e}(t)$  for $t\in[0, \tfrac{T}{2}]$,   Proposition \ref{p:main-prop2} implies that $u(0,x)=\widetilde{u}(0,x)$. This relation combined with  \eqref{u-e} and  \eqref{tu-te} shows Theorem \ref{t:main}.

\section{Proof of Proposition \ref{iteration}}\label{proof-1} In this section, we are devoted to the proof of Proposition \ref{iteration}. More specifically, we construct  $(u_{q+1}, p_{q+1}, \RR_{q+1})$ in Proposition \ref{iteration} by the following two steps:
\begin{enumerate}
  \item [$\bullet$]Step 1: Mollification: $(\uq, p_q, \RR_q)\mapsto (u_{\ell_q},  p_{\ell_q}, \RR_{\ell_q})$. We define $(u_{\ell_q},  p_{\ell_q}, \RR_{\ell_q})$ by mollifying  $(\uq, p_q, \RR_q)$ so that  we  {obtain higher regularity estimates for the perturbation.}
 \item [$\bullet$]Step 2: Perturbation: $ (u_{\ell_q},  p_{\ell_q}, \RR_{\ell_q})\mapsto (u_{q+1}, p_{q+1}, \RR_{q+1})$. {By making use of the box flows introduced in \cite{MY}, the temporal corrector inspired by \cite{BV}  and the localized corrector , we construct the perturbation $w_{q+1}$. Then we define $u_{q+1}$ by adding $w_{q+1}$ on~$u_{\ell_q}$.}
\end{enumerate}
\subsection{Mollification}
We define the functions $(u_{\ell_q},  p_{\ell_q}, \RR_{\ell_q})$ by the spatial  mollifier $\psi_{\ell_q}$ and the time mollifier $\varphi_{\ell_q}$ in Definition \ref{e:defn-mollifier-t} as follows: For $(x,t)\in \R^3\times [0, T-\ell_q]$,
\begin{align}
   & u_{\ell_q} (x,t):=\int_t^{t+\ell_q}(\uq * \psi_{\ell_q})(x,s)\varphi_{\ell_q}(t-s)\dd s, \label{def-ulq}\\
   &\ulql(x,t):=\int_t^{t+\ell_q}(\uql * \psi_{\ell_q})(x,s)\varphi_{\ell_q}(t-s)\dd s, \label{def-ulql}\\
      &\ulqnl(x,t):=\int_t^{t+\ell_q}(\uqnl * \psi_{\ell_q})(x,s)\varphi_{\ell_q}(t-s)\dd s, \label{def-ulqnl}\\
   &p_{\ell_q}(x,t):=\int_t^{t+\ell_q}(p_q * \psi_{\ell_q})(x,s)\varphi_{\ell_q}(t-s)\dd s,\label{def-plq}\\
 &   \RR_{\ell_q}(x,t) := \int_t^{t+\ell_q}(\RR_q * \psi_{\ell_q})(x,s)\varphi_{\ell_q}(t-s)\dd s,  \label{def-Rlq}\\
 &\Rem :=  u_{\ell_q} \otimes u_{\ell_q} - \int_t^{t+\ell_q}((u_q \otimes u_q) * \psi_{\ell_q})(x,s)\varphi_{\ell_q}(t-s)\dd s.\label{def-Rem}
\end{align}
 One easily verifies that  $(u_{\ell_q},  p_{\ell_q}, \RR_{\ell_q}, \Rem)$ solves the Cauchy problem
\begin{equation}
\left\{ \begin{alignedat}{-1}
&\del_t u_{\ell_q}-\Delta u_{\ell_q}+\Div (u_{\ell_q}\otimes u_{\ell_q})  +\nabla p_{\ell_q}   =  \Div \RR_{\ell_q} +\Div \Rem,
\\
 & \nabla \cdot u_{\ell_q} = 0,
  \\
  &  u_{\ell_q} |_{t=0}=   \int_0^{\ell_q}(u_q * \psi_{\ell_q})(x,s)\varphi_{\ell_q}(-s)\dd s
\end{alignedat}\right.
 \label{e:mollified-euler}
\end{equation}
{such that
\begin{align}
&\spt_x \ulql, \,\spt_x \RR_{\ell_q} \subseteq \Omega_{q}+[-\lambda^{-1}_q, \lambda^{-1}_q]^3.\label{supp-vlq}
\end{align}
Moreover,  $(u_{\ell_q}, \RR_{\ell_q}, \Rem)$ satisfies the following estimates.
\begin{proposition}[Estimates  for $(u_{\ell_q}, \RR_{\ell_q}, \Rem)$]\label{p:estimates-for-mollified}For any integers $L,N\ge 0$, we have
\begin{align}
&\| \ulql-\uql\|_{L^\infty_tH^4} \lesssim \lambda^{10}_q \ell_q,\quad \| \ulqnl-\uqnl\|_{L^\infty_tH^4} \lesssim \lambda^{10}_q \ell_q,\label{e:v_ell-vq}
\\
&\|\partial^L_t u_{\ell_q} \|_{L^\infty_tH^{N+5}} \lesssim  \lambda^{10}_{q} \ell_q^{-N-L},  \label{e:v_ell-CN+1}
\\
&{\|\RR_{\ell_q}\|_{L^\infty_t L^1} \le  \delta_{q+1} \lambda^{-4\alpha}_q}, && \label{e:R_ell}
\\
&\|\partial^L_t\RR_{\ell_q}\|_{L^\infty_tW^{N+5,1}} \lesssim   \lambda^{10}_{q} \ell_q^{-N-L},\label{e:R_ell-W} \\
&\|\Rem\|_{L^\infty_tW^{N+4,1}} \lesssim  \lambda^{20}_{q} \ell_q^{-N+1},\label{e:R_rem}\\
 & \frac{\delta_{q+1}}{2}\leq e(t) - \int_{\mathbb R^3}  |u_{\ell_q}|^2 \dd x\leq 4\delta_{q+1}. \label{e:ell}
\end{align}
\end{proposition}}

{\subsection{Perturbation}
 The principle perturbation is mainly composed of two parts: amplitudes and the intermittent building blocks.\\
 \noindent\emph{Amplitudes}\quad We define amplitudes firstly. Let $\chi:[0,\infty)\to [1,\infty)$ be a smooth function satisfying
\begin{equation}\label{def-chi}
\chi(z)=\left\{ \begin{alignedat}{-1}
&1, \quad 0\le z\le 1,\\
&z, \quad z\ge 2,
\end{alignedat}\right.
\end{equation}
with $z\le 2\chi(z)\le 4z$ for $z\in (1,2)$. We define that
\[\chi_q:=\chi\Big(\Big\langle\tfrac{\RR_{\ell_q}}{\delta_{q+1}\lambda^{-2\alpha}_q}\Big\rangle\Big)
,\quad \text{where}\,\,\big\langle\cdot \big\rangle:=\sqrt{1+|\cdot|^2}. \]
Next, the space cutoffs $\{\eta_q(x)\}_{q\ge 1}\in C^{\infty}_c(\R^3)$ are defined by
\begin{equation}\label{eta}
\spt \,\,\eta_{q}=\Omega_{q+1} , \quad\quad   \eta_q|_{\Omega_{q}+ [ -\lambda^{-1}_q ,\lambda^{-1}_q ]^3}\equiv 1,\quad~~\quad|D^N\eta_q|\lesssim \lambda^{N}_q\,\,\text{for}\,\,N\ge 0,
\end{equation}
where $\Omega_{q+1}$ is defined in \eqref{omega}.  Thanks to \eqref{supp-vlq}, one can easily deduce that
\begin{align}\label{R=Rq}
\eta_{q}\RR_{\ell_{q}}=\RR_{ \ell_{q}}.
\end{align}
We claim that
\begin{align}\label{eta_qchi_q}
  \frac{1}{2}\le \int_{\R^3}\eta^2_q\chi_q\dd x
   \le 5.
\end{align}
Indeed, we deduce from the definitions of $\eta_q$ and $\chi_q$ that
\begin{align*}
  \int_{\R^3}\eta^2_q\chi_q\dd x=\int_{\Omega_{q+1}}\eta_q^2(x)\chi\Big(\Big\langle\tfrac{\RR_{\ell_q}}{\delta_{q+1}\lambda^{-2\alpha}_q}\Big\rangle\Big)\dd x.
\end{align*}
On one hand, from \eqref{e:R_ell}, one has
\begin{align}
   \int_{\Omega_{q+1}}\eta_q^2(x)\chi\Big(\Big\langle\tfrac{\RR_{\ell_q}}{\delta_{q+1}\lambda^{-2\alpha}_q}\Big\rangle\Big)\dd x
   \le&\int_{\Omega_{q+1}}4\big(1+\big|\tfrac{\RR_{\ell_q}}{\delta_{q+1}\lambda^{-2\alpha}_q}\big|\big)\dd x\le 5.
\end{align}
On the other hand,
\begin{align*}
   &\int_{\Omega_{q+1}}\eta_q^2(x)\chi\Big(\Big\langle\tfrac{\RR_{\ell_q}}{\delta_{q+1}\lambda^{-4\alpha}_q}\Big\rangle\Big)\dd x
   \ge\int_{\Omega_{q}}1\dd x\ge\frac{1}{2}.
\end{align*}
We define
 \begin{align}\label{energy2}
\rho_q(t) =   \frac{1}{3\int_{\mathbb{R}^3}\eta^2_q\chi_{q}\dd x}\Big(e(t) - \int_{\mathbb R^3}  |u_{\ell_q}|^2 \dd x - 2{\delta_{q+2}} \Big).
\end{align}
Then \eqref{e:ell} and \eqref{eta_qchi_q} together show that
\begin{align}\label{est-rhoq}
   \frac{1}{60}\delta_{q+1}\le\rho_q(t)\le 3\delta_{q+1}
\end{align}
Moreover, by \eqref{e:v_ell-CN+1}, \eqref{e:R_ell-W} and the definition of $\rho_q(t)$ in \eqref{energy2}, we see
 \begin{align}
\|\partial_t\rho_q(t)\|_{L^\infty_{t,x}}\lesssim &\delta_{q+1}\|\partial_t\chi_q\|_{L^\infty_{t,x}}+\|u_{\ell_q}\|_{L^\infty_t L^2}\|\partial_t u_{\ell_q}\|_{L^\infty_t L^2}\nonumber\\
\lesssim &\lambda^{20}_{q}\ell^{-1}_q.\label{pt-rho-q}
\end{align}

We are in position to give the  amplitudes $a_{(k,q)}(t,x)$ by
\begin{equation}\label{def-akq}
    a_{(k,q)}(t,x)=\eta_{q}
 a_k\Big({\rm Id}-\frac{\RR_{\ell_{q}}}{\chi_{q}\rho_q}\Big)
 (\chi_q \rho_q)^{1/2},
\end{equation}
 where $a_k$ stems from Lemma \ref{first S}. 

\subsubsection{Estimates for  amplitudes $a_{(k,q)}$} In order to estimate  $a_{(k,q)}$, we give some estimates for $\chi_q\rho_q$ firstly. Note that  the smooth function $\chi \ge1$ in \eqref{def-chi}, one can expect to obtain higher-order derivative estimates for $(\chi_q\rho_q)^{1/2}$ and  $(\chi_q\rho_q)^{-1}$, which are shown by Proposition \ref{cr1/2} as follows.
{\begin{proposition}[Estimates for $(\chi_q\rho_q)^{1/2}$ and $(\chi_q\rho_q)^{-1}$]\label{cr1/2}For any $N\ge 2$, we have
\begin{align}
&\|(\chi_q\rho_q)^{1/2}\|_{L^\infty_tH^N}\le    C_N\delta^{-N+1/2}_{q+1}\lambda^{10N+2N\alpha}_q\ell^{-N}_q, \label{chi-H4}\\
& \|(\chi_q\rho_q)^{-1}\|_{L^\infty_tH^N}
\le C_N\delta^{-N-1}_{q+1}\lambda^{10N+2N\alpha}_q\ell^{-N}_q,   \label{est-cr-1}\\
&\|\partial_t(\chi_q\rho_q)^{1/2}\|_{L^\infty_tH^N}\le  C_N\delta^{-N-1/2}_{q+1}\lambda^{10N+10+2(N+1)\alpha}_q\ell^{-2N}_q,\label{t-chi-H4}\\
&\|\partial_t(\chi_q\rho_q)^{-1}\|_{L^\infty_tH^N}
\le C_N\delta^{-N-2}_{q+1}\lambda^{10N+10+2(N+1)\alpha}_q\ell^{-2N}_q. \label{est-t-cr-1}
\end{align}
\end{proposition}
\begin{proof}
Since $\chi(x)\ge1$, $\rho_q\sim \delta_{q+1}$  and
    \begin{align*}
    \chi_q\rho_q=\chi\Big(\Big\langle\tfrac{\RR_{\ell_q}}{\delta_{q+1}\lambda^{-2\alpha}_q}\Big\rangle\Big)\rho_q,
    \end{align*}
we infer from \eqref{e:R_ell-W} that
\begin{align*}
\|(\chi_q\rho_q)^{1/2}\|_{L^\infty_tH^N}=&\|\chi^{1/2}_q\|_{L^\infty_t H^N}\|\rho_q^{1/2}(t)\|_{L^\infty}\\
\le&C_N\delta^{1/2}_{q+1}\Big(1+\Big{\|}\frac{\RR_{\ell_q}}{\delta_{q+1}\lambda^{-2\alpha}_q}\Big{\|}_{L^\infty_{t,x}}\Big)^{N-1}\Big{\|}\frac{\RR_{\ell_q}}{\delta_{q+1}\lambda^{-2\alpha}_q}\Big{\|}_{L^\infty_tH^N}\\
\le& C_N\delta^{-N+1/2}_{q+1}\lambda^{10N+2N\alpha}_q\ell^{-N}_q
\end{align*}
and
\begin{align}
 \|(\chi_q\rho_q)^{-1}\|_{L^\infty_tH^N}\le& C_N\delta^{-1}_{q+1}\Big(1+\Big{\|}\frac{\RR_{\ell_q}}{\delta_{q+1}\lambda^{-2\alpha}_q}\Big{\|}_{L^\infty_{t,x}}\Big)^{N-1}\Big{\|}\frac{\RR_{\ell_q}}{\delta_{q+1}\lambda^{-2\alpha}_q}\Big{\|}_{L^\infty_tH^N}\nonumber\\
\le& C_N\delta^{-N-1}_{q+1}\lambda^{10N+2N\alpha}_q\ell^{-N}_q. \nonumber
\end{align}
With the help of \eqref{e:R_ell-W} and \eqref{pt-rho-q}, one deduces  that
\begin{align*}
&\|\partial_t(\chi_q\rho_q)^{1/2}\|_{L^\infty_tH^N}\\
\le&\|\partial_t\chi_q^{1/2}\|_{L^\infty_tH^N}\|\rho_q^{1/2}(t)\|_{L^\infty}+
\|\chi_q^{1/2}\|_{L^\infty_tH^N}\|\partial_t\rho_q^{1/2}(t)\|_{L^\infty}\\
\le& C_N\delta^{1/2}_{q+1}\Big{\|}\frac{\partial_t\RR_{\ell_q}}{\delta_{q+1}\lambda^{-2\alpha}_q}\Big{\|}_{L^\infty_tH^N}\Big(1+\Big{\|}\frac{\RR_{\ell_q}}{\delta_{q+1}\lambda^{-2\alpha}_q}\Big{\|}_{L^\infty_{t,x}}\Big)^{N-1}\Big{\|}\frac{\RR_{\ell_q}}{\delta_{q+1}\lambda^{-2\alpha}_q}\Big{\|}_{L^\infty_tH^N}\\
&+ C_N\delta^{-1/2}_{q+1}\lambda^{20}_q\ell^{-1}_q\Big(1+\Big{\|}\frac{\RR_{\ell_q}}{\delta_{q+1}\lambda^{-2\alpha}_q}\Big{\|}_{L^\infty_{t,x}}\Big)^{N-1}\Big{\|}\frac{\RR_{\ell_q}}{\delta_{q+1}\lambda^{-2\alpha}_q}\Big{\|}_{L^\infty_tH^N}\\
\le& C_N\delta^{-N-1/2}_{q+1}\lambda^{10N+10+2(N+1)\alpha}_q\ell^{-2N}_q,
\end{align*}
and
\begin{align*}
&\|\partial_t(\chi_q\rho_q)^{-1}\|_{L^\infty_tH^N}\\
\le&\|\partial_t\chi_q^{-1}\|_{L^\infty_tH^N}\|\rho_q^{-1}(t)\|_{L^\infty}+
\|\chi_q^{-1}\|_{L^\infty_tH^N}\|\partial_t\rho_q^{-1}(t)\|_{L^\infty}\\
\le& C_N\delta^{-1}_{q+1}\Big{\|}\frac{\partial_t\RR_{\ell_q}}{\delta_{q+1}\lambda^{-2\alpha}_q}\Big{\|}_{L^\infty_tH^N}\Big(1+\Big{\|}\frac{\RR_{\ell_q}}{\delta_{q+1}\lambda^{-2\alpha}_q}\Big{\|}_{L^\infty_{t,x}}\Big)^{N-1}\Big{\|}\frac{\RR_{\ell_q}}{\delta_{q+1}\lambda^{-2\alpha}_q}\Big{\|}_{L^\infty_tH^N}\\
&+ C_N\delta^{-1}_{q+1}\lambda^{20}_q\ell^{-1}_q\Big(1+\Big{\|}\frac{\RR_{\ell_q}}{\delta_{q+1}\lambda^{-2\alpha}_q}\Big{\|}_{L^\infty_{t,x}}\Big)^{N-1}\Big{\|}\frac{\RR_{\ell_q}}{\delta_{q+1}\lambda^{-2\alpha}_q}\Big{\|}_{L^\infty_tH^N}\\
\le& C_N\delta^{-N-2}_{q+1}\lambda^{10N+10+2(N+1)\alpha}_q\ell^{-2N}_q.
\end{align*}
Hence we complete the proof of Proposition \ref{cr1/2}.
\end{proof}}
Based on Proposition \ref{cr1/2}, we are in position to estimate ${a}_{(k,q)}$.
\begin{proposition}[Estimates for ${a}_{(k,q)}$]\label{est-ak}{For $N\ge 2$, we have
\begin{align}
&\|{a}_{(k,q)}\|_{L^\infty_t H^N}\le C_N\ell^{-2N}_q, \label{akq-HN} \\
&\|\partial_t {a}_{(k,q)}\|_{L^\infty_t H^N}\le  C_N\ell^{-6N}_q.\label{t-akq-HN}
\end{align}}
\end{proposition}
\begin{proof}{
Combining with \eqref{e:R_ell-W} and \eqref{est-cr-1}, we obtain that
\begin{align}
\Big\|\frac{\RR_{\ell_q}}{\chi_q \rho_q}\Big\|_{L^\infty_tH^N}\le& \|\RR_{\ell_q}\|_{L^\infty_{t,x}} \|(\chi_q\rho_q)^{-1}\|_{L^\infty_t H^N}+\|\RR_{\ell_q}\|_{L^\infty_t H^N} \|(\chi_q\rho_q)^{-1}\|_{L^\infty_{t,x}}\nonumber\\
\le&C_N\delta^{-N-1}_{q+1}\lambda^{10N+10+2N\alpha}_q\ell^{-N}_q+C\delta^{-1}_{q+1}\lambda^{10}_q\ell^{-N}_q\nonumber\\
\le&C_N \delta^{-N-1}_{q+1}\lambda^{10N+10+2N\alpha}_q\ell^{-N}_q.\label{R-cr-HN}
\end{align}
This inequality together with \eqref{e:R_ell-W} implies that
\begin{align}
 \Big{\|} a_{k}\Big({\rm Id}-\frac{\RR_{\ell_q}}{\chi_q\rho_q}\Big)\Big{\|}_{L^\infty_t H^N}\le &C_N\Big(1+\Big\|\frac{\RR_{\ell_q}}{\chi_q\rho_q}\Big\|_{L^\infty_{t,x}}\Big)^{N-1}\Big\|\frac{\RR_{\ell_q}}{\chi_q\rho_q}\Big\|_{L^\infty_tH^N}\nonumber\\
\le&C_N \delta^{-2N}_{q+1}\lambda^{20N+2N\alpha}_q\ell^{-N}_q.\label{ak-HN}
\end{align}
Since ${a}_{(k,q)}=\eta_q a_{k}\big({\rm Id}-\frac{\RR_{\ell_q}}{\chi_q\rho_q}\big)(\chi_q\rho_q)^{1/2}$, we have by \eqref{chi-H4}, \eqref{R-cr-HN} and \eqref{ak-HN} that
\begin{align*}
\|{a}_{(k,q)}\|_{L^\infty_t H^N}\le& \Big{\|} a_{k}\Big({\rm Id}-\frac{\RR_{\ell_q}}{\chi_q\rho_q}\Big)\Big{\|}_{L^\infty_t H^N}\|(\chi_q\rho_q)^{1/2}\|_{L^\infty_{t,x}}+C\|(\chi_q\rho_q)^{1/2}\|_{L^\infty_t H^N}\\
\le&C_N\delta^{-2N}_{q+1}\lambda^{20N+2N\alpha}_q\ell^{-N}_q+C_N\delta^{-N+1/2}_{q+1}\lambda^{10N+2N\alpha}_q\ell^{-N}_q\\
\le&C_N\delta^{-2N}_{q+1}\lambda^{20N+2N\alpha}_q\ell^{-N}_q.
\end{align*}
Owning to
\begin{align}\label{Con1}
\alpha<\frac{1}{4},  \,\,\delta^{-2}_{q+1}<\lambda^{\frac{1}{100}}_q, \,\,\ell_q=\lambda^{-60}_q,
\end{align}
 we obtain \eqref{akq-HN}.}

By \eqref{e:R_ell-W}, \eqref{est-cr-1} and \eqref{est-t-cr-1}, we have that
\begin{align}
\Big\|\partial_t\Big(\frac{\RR_{\ell_q}}{\chi_q \rho_q}\Big)\Big\|_{L^\infty_tH^N}\le& \|\partial_t\RR_{\ell_q}\|_{L^\infty_t H^N} \|(\chi_q\rho_q)^{-1}\|_{L^\infty_t H^N}+\|\RR_{\ell_q}\|_{L^\infty_t H^N} \|\partial_t(\chi_q\rho_q)^{-1}\|_{L^\infty_t H^N}\nonumber\\
\le&C_N  \delta^{-N-1}_{q+1}\lambda^{10N+10+2N\alpha}_q\ell^{-2N}_q+C_N \delta^{-N-2}_{q+1}\lambda^{10N+20+2(N+1)\alpha}_q\ell^{-3N}_q\nonumber\\
\le&C_N \delta^{-N-2}_{q+1}\lambda^{10N+20+2(N+1)\alpha}_q\ell^{-3N}_q.\label{t-R-cr-HN}
\end{align}
Therefore, one infers from  \eqref{e:R_ell-W}, \eqref{R-cr-HN} and \eqref{t-R-cr-HN} that
\begin{align}
\Big\|\partial_ta_{k}\Big({\rm Id}-\frac{\RR_{\ell_q}}{\chi_q \rho_q}\Big)\Big\|_{L^\infty_tH^N}\le&\Big\|(a'_k)\Big({\rm Id}-\frac{\RR_{\ell_q}}{\chi_q \rho_q}\Big)\Big\|_{L^\infty_tH^N}\Big\|\partial_t\Big(\frac{\RR_{\ell_q}}{\chi_q\rho_q}\Big)\Big\|_{L^\infty_tH^N}\nonumber\\
\lesssim& \Big(1+\Big\|\Big(\frac{\RR_{\ell_q}}{\chi_q\rho_q}\Big)\Big\|_{L^\infty_{t,x}}\Big)^{N-1}\Big\|\Big(\frac{\RR_{\ell_q}}{\chi_q\rho_q}\Big)\Big\|_{L^\infty_tH^N}\Big\|\partial_t\Big(\frac{\RR_{\ell_q}}{\chi_q \rho_q}\Big)\Big\|_{L^\infty_tH^N}\nonumber\\
\le&
C_N \delta^{-3N-2}_{q+1}\lambda^{30N+20+(10N+6)\alpha}_q\ell^{-4N}_q
\label{t-ak-HN}
\end{align}
Collecting \eqref{chi-H4}, \eqref{t-chi-H4}, \eqref{ak-HN} and \eqref{t-ak-HN} together shows that
\begin{align*}
\|\partial_t {a}_{(k,q)}\|_{L^\infty_t H^N}\le
&\|\eta_q\|_{ H^N}\Big\|\partial_t a_{k}\Big({\rm Id}-\frac{\RR_{\ell_q}}{\chi_q\rho_q}\Big)\Big\|_{L^\infty_t H^N}\|(\chi_q\rho_q)^{1/2}\|_{L^\infty_t H^N}\\
&+\|\eta_q\|_{ H^N}\Big\| a_{k}\Big({\rm Id}-\frac{\RR_{\ell_q}}{\chi_q\rho_q}\Big)\Big\|_{L^\infty_t H^N}\|\partial_t(\chi_q\rho_q)^{1/2}\|_{L^\infty_t H^N}\\
\le& C_N  \delta^{-4N-2}_{q+1}\lambda^{41N+20+(12N+6)\alpha}_q\ell^{-5N}_q,
\end{align*}
where we use the fact that $\|\eta_q\|_{H^N}\lesssim \lambda^N_{q}$. Thanks to \eqref{Con1}, we prove \eqref{t-akq-HN}. Thus, we complete the proof of Proposition \ref{est-ak}.
\end{proof}
\subsubsection{Construction the perturbation}
\noindent\\
\emph{Building blocks}\quad Assume that $\Phi:\mathbb{R}\rightarrow\mathbb{R}$ is a smooth cutoff function supported on the interval $(0, \lambda^{-1}_1]$ and $\int_{\R} \Phi''(x) \dd x=0$. We set $\phi=\frac{\dd^2}{\dd x^2}\Phi$ and normalize it in such a way that
\begin{align}\label{normalize}
\int_{\mathbb{R}}\phi^2\dd x=1.
\end{align}
For any small positive parameter $r$ such that $r^{-1}\in\ZZ$, we define
\begin{align}\label{def-phir}
\phi_{r}(x) :=r^{-\frac{1}{2}}\phi (r^{-1}x),\qquad \Phi_{r}(x) :=r^{-\frac{1}{2}}\Phi (r^{-1}x).
\end{align}
 We periodize $\phi_r$ and $\Phi_r$ so that  the resulting functions are periodic functions defined on $\mathbb{R}/\mathbb{Z}=:\mathbb{T}$, and we still denote the $\mathbb{T}$-periodic functions by $\phi_r$ and $\Phi_r$. Next, we choose a large number  $N_{\Lambda}$ such that
$N_{\Lambda}{k}\in \mathbb{Z}^3$ for all $k\in\Lambda$, where $\Lambda\subset\mathbb{S}^2\cap\mathbb{Q}^3$ is denoted in Remark \ref{Lambda}.

Let $r=\lambda^{-3/4}_{q+1}$ and $\mu=\lambda^{5/4}_{q+1}$, we define
\begin{align*}
&\phi_{k,r}(x): =\phi_{r}( N_{\Lambda}k\cdot x),\,\,\phi_{\bar{k},r, \mu}(x,t):=\phi_{r}( N_{\Lambda}(\bar{k}\cdot x+\mu t)),\,\,\phi_{\bar{\bar{k}},r}(x):=\phi_{r}( N_{\Lambda}\bar{\bar{k}}\cdot  x).
  \end{align*}
Let $\sigma=\lambda^{1/16}_{q+1}$, we define
\begin{align*}
 &\phi_{k,r,\sigma}(x):= \phi_{k,r}(\sigma x)=N^{-1}_{\Lambda}r^{2}\sigma^{-1}\Div{\big((\Phi_{r})'( N_{\Lambda}k\cdot \sigma x)k\big)},
 \end{align*}
 where
  \begin{align*}
\Phi’_{r}( N_{\Lambda}k\cdot \sigma x):=r(\Phi_{r})'(N_{\Lambda}k\cdot \sigma x).
   \end{align*}
   Therefore, we have
  \begin{align*}
 &\phi_{k,r,\sigma}(x)=N^{-1}_{\Lambda}r\sigma^{-1}\Div{\big(\Phi'_{r}(N_{\Lambda}k\cdot \sigma x)k\big)}.
 \end{align*}
Let
 \begin{align*}
 &\phi_{\bar{k},r,\sigma, \mu}(x,t):=\phi_{\bar{k},r, \mu}(\sigma x,\sigma t)={N^{-1}_{\Lambda}r\sigma^{-1}}\Div{\big(\Phi'_{r}(\sigma N_{\Lambda}(\bar{k}\cdot x+\mu t))\bar{k}}\big),\\
 &\phi_{\bar{\bar{k}},r,\sigma}(x):=\phi_{\bar{\bar{k}},r}(\sigma x)=N^{-1}_{\Lambda}r\sigma^{-1}\Div{\big(\Phi'_{r}( N_{\Lambda}\bar{\bar{k}}\cdot \sigma x)\bar{\bar{k}}\big)}.
\end{align*}
Finally, we define a set of functions $\{\phi_{k,\bar{k},\bar{\bar{k}}}\}_{k\in\Lambda}$ and $\{\phi_{k,\bar{k},\bar{\bar{k}},\sigma}\}_{k\in\Lambda}:\TTT^3\times \R\to\R$ by
\begin{align*}
    \phi_{k,\bar{k},\bar{\bar{k}}}(x,t)=\phi_{k,r}(x)\phi_{\bar{k},r, \mu}(x,t)\phi_{\bar{\bar{k}},r}(x)
\end{align*}
and
\begin{align*}
\phi_{k,\bar{k},\bar{\bar{k}},\sigma}(x,t)
&:=\phi_{k,r,\sigma}(x-x_k)\phi_{\bar{k},r,\sigma, \mu}(x-x_k,t)\phi_{\bar{\bar{k}},r,\sigma}(x-x_k),\\
&= \phi_{k,\bar{k},\bar{\bar{k}}}(\sigma(x-x_k),\sigma t)
\end{align*}
where  $\{x_k\}_{k\in\Lambda}$ are shifts which guarantee that
\begin{align}\label{disjoint}
\phi_{k,\bar{k},\bar{\bar{k}},\sigma}\cdot\phi_{k',\bar{k}',\bar{\bar{k}}',\sigma}=0,
~~~k\neq k',~~k,k'\in\Lambda.
\end{align}
Readers can refer to \cite{BV-19} for this equality. Moreover, we deduce from \eqref{normalize} that
\begin{align}\label{phi-1}
    \|\phi_{k,\bar{k},\bar{\bar{k}},\sigma}\|_{L^2(\TTT^3)}=1.
\end{align}
From the definitions of $\phi_{k,\bar{k},\bar{\bar{k}}}$ and $\phi_{k,\bar{k},\bar{\bar{k}},\sigma}$, one easily concludes by $k\perp \bar{k}\perp \bar{\bar{k}}$ that
\begin{align}
 \Div(\phi^2_{k,\bar{k},\bar{\bar{k}}} \bar{k}\otimes \bar{k})&=\frac{1}{\mu}\partial_t\phi^2_{k,\bar{k},\bar{\bar{k}}}\bar{k}, \label{div-pt-phi}\\
    \Div(\phi^2_{k,\bar{k},\bar{\bar{k}},\sigma} \bar{k}\otimes \bar{k})&=\frac{1}{\mu}\partial_t\phi^2_{k,\bar{k},\bar{\bar{k}},\sigma}\bar{k}, \label{div-pt}
\end{align}
By \cite[Lemma 5.2]{2Beekie}, one immediately infers the following estimate:
 \begin{proposition}\label{guji1}
For $p\in[1,\infty]$, $m,n\in\mathbb{N}$, we have the following estimates
\begin{align*}
   &\big\|{D}^n_t{D}^m_x\phi_{k,\bar{k},\bar{\bar{k}},\sigma}\big\|_{L^\infty_t L^p(\TTT^3)}\lesssim N^{m+n}_{\Lambda}(r^{-1}\sigma)^{m+n}\mu^n r^{3(\frac{1}{p}-\frac{1}{2})}=N^{m+n}_{\Lambda}\lambda^{\frac{13}{16}m+\frac{33}{16}n-\frac{9}{4}(\frac{1}{p}-\frac{1}{2})}_{q+1}.
\end{align*}
\end{proposition} Before constructing the perturbation, we introduce $\psi,\Psi\in C^{\infty}(\mathbb{T})$ such that $\psi=\Psi''$ and $\int_{\TTT}\Psi'\dd x=0$. Then we define
\begin{align*}
    \psi_k=\psi(N_{\Lambda}\lambda_{q+1}k\cdot x), \qquad \Psi_k=\Psi(N_{\Lambda}\lambda_{q+1}k\cdot x), \quad k\in\Lambda.
\end{align*}
We assume that
{\begin{align}\label{psi-1}
    \|\psi_k\|^2_{L^2(\TTT^3)}=1.
\end{align}}
Note that $|k|=1$, one deduces that
\begin{align}\label{psik}
\psi_{k} \bar{k}  =N^{-1}_{\Lambda}\lambda^{-1}_{q+1}\Div\big({\Psi'_k}k\otimes\bar{k}\big),
\end{align}
where
\begin{align*}
    \Psi'_k:=\Psi'(N_{\Lambda}\lambda_{q+1}k\cdot x).
\end{align*}
We  construct the principal perturbation $ \wpq$ as follows:
\begin{equation}\label{def-wpq-1}
 \begin{aligned}
 \wpq&:=\sum_{k\in\Lambda}a_{(k,q)} \phi_{k,\bar{k},\bar{\bar{k}},\sigma}\psi_k\bar{k}.
 \end{aligned}
 \end{equation}
Here $\phi_{k,\bar{k},\bar{\bar{k}},\sigma}\psi_k\bar{k}$ is the so-called box flow introduced in \cite{MY}.  Since $\wpq$ is not divergence-free,  we need to construct the incompressibility corrector of  $\wpq$  to ensure that the perturbation is divergence-free. Note that
\begin{align*}
 \Div\big({\Psi'_k}\bar{k}\otimes k\big)=0, \quad \forall k\perp \bar{k},
\end{align*}
We rewrite $\wpq$  by \eqref{psik}
\begin{align}
 \wpq
 =&\sum_{k\in\Lambda}N^{-1}_{\Lambda}\lambda^{-1}_{q+1}a_{(k,q)}\phi_{k,\bar{k},\bar{\bar{k}},\sigma}{\Div\big({\Psi'_k}(k\otimes \bar{k}-\bar{k}\otimes k)\big)}. \label{def-wpq}
\end{align}
Based on this equality, we define $\wcq$ by
\begin{align}\label{def-wcq}
\wcq=&\sum_{k\in\Lambda}N^{-1}_{\Lambda}\lambda^{-1}_{q+1}{\Psi'_k}{\Div\big(a_{(k,q)} \phi_{k,\bar{k},\bar{\bar{k}},\sigma} (k\otimes \bar{k}-\bar{k}\otimes k)\big)}.
\end{align}
One easily deduces that
\begin{align*}
&\Div(\wpq+\wcq)
=\sum_{k\in\Lambda}N^{-1}_{\Lambda}\lambda^{-1}_{q+1}\Div\Div\big(a_{(k,q)} \phi_{k,\bar{k},\bar{\bar{k}},\sigma}\Psi'_k(k\otimes \bar{k}-\bar{k}\otimes k)\big)
=0,
\end{align*}
where we have used the fact that $\Div\Div M=0$ for any anti-symmetric matrix $M$.

Next, we define $\wttq$ by
\begin{equation}\label{wttq}
 \begin{aligned} \wttq= -\frac{1}{\mu}\sum_{k\in\Lambda}a^2_{(k,q)}
 \PH\PP\big(\phi^2_{k,\bar{k},\bar{\bar{k}},\sigma} \bar{k} \big)=:-\frac{1}{\mu}\sum_{k\in\Lambda}a^2_{(k,q)}
 \Ph\big(\phi^2_{k,\bar{k},\bar{\bar{k}},\sigma} \bar{k} \big),
 \end{aligned}
 \end{equation}
where $\PH$ is the Helmholtz  projector onto divergence-free vector fields. Since $\Div \Ph\big(\phi^2_{k,\bar{k},\bar{\bar{k}},\sigma} \bar{k} \big)=0$, we have
\[-\Delta \Ph\big(\phi^2_{k,\bar{k},\bar{\bar{k}},\sigma} \bar{k} \big)=\curl\curl \Ph\big(\phi^2_{k,\bar{k},\bar{\bar{k}},\sigma} \bar{k} \big).\]
Based on this equality, we have
\begin{align*}
   \Ph\big(\phi^2_{k,\bar{k},\bar{\bar{k}},\sigma} \bar{k} \big)&=(-\Delta)^{-1}\curl\curl \Ph\big(\phi^2_{k,\bar{k},\bar{\bar{k}},\sigma} \bar{k} \big)\\
   &=\curl \Big(\sigma^{-1}\big((-\Delta)^{-1}\curl \Ph (\phi^2_{k,\bar{k},\bar{\bar{k}}}\bar{k})\big)(\sigma (x-x_k), \sigma t)\Big)\\
   &=:\curl(\sigma^{-1} G(\sigma (x-x_k), \sigma t)).
\end{align*}
Then we define
 \begin{equation}\label{wttqc}
 \begin{aligned}
 \wttqc:= \frac{-1}{\mu}\sum_{k\in\Lambda}\nabla{a}^2_{(k,q)}\times \sigma^{-1}G(\sigma (x-x_k), \sigma t) .
 \end{aligned}
 \end{equation}
 We immediately verify that
\begin{align}
   \Div (\wttq+  \wttqc)=\Div\curl\Big(-\frac{1}{\mu}\sum_{k\in\Lambda}a_{(k,q)}\sigma^{-1}G(\sigma (x-x_k), \sigma t)\Big)=0.\label{wttq+wttqc}
\end{align}
Finally, we define $\wtq$ by solving the following Cauchy problem:
 \begin{equation}
\left\{ \begin{alignedat}{-1}
&\del_t \wtq-\Delta \wtq+\Div (\wtq\ootimes \wtq)+\Div(\wtq\ootimes\ulqnl)\\
&\qquad\qquad\qquad\qquad\quad+\Div(\ulqnl\ootimes \wtq)  +\nabla p^{(\textup{ns})}_{q+1}  = F_{q+1},
\\
 & \Div \wtq = 0,
  \\
  & \wtq |_{t=0}=  0 ,
\end{alignedat}\right.
 \label{e:wt}\tag{LNS}
\end{equation}
{where the force term $F_{q+1}$ is determined by $\RR^{rem}_q$, and some terms extracted from  $\Div(\wpq\otimes \wpq)$ and $\partial_t (\wpq+\wcq)$,  as  in \eqref{def-F_{q+1}}. To demonstrate the derivation of $F_{q+1}$ clearly,  we have to decompose  $\Div(\wpq\otimes \wpq)$ and $\partial_t (\wpq+\wcq)$ . }

\begin{proposition}[The decomposition of $\Div(\wpq\ootimes\wpq)$]\label{def-F2}Let
$\wpq$ be defined in \eqref{def-wpq-1}, then there exist a trace-free matrix $\RR^{(1)}_{(q+1)}$, pressure $P^{(1)}_{(q+1)}$ and the error term $F^{(1)}_{q+1}$ such that
\begin{align}
\Div(\wpq\ootimes \wpq)+\Div\RR_{\ell_q}=\Div \RR^{(1)}_{(q+1)}-\partial_t\wttq+F^{(1)}_{q+1}+\nabla P^{(1)}_{q+1}\label{decom-wpwp},
\end{align}
satisfying
\begin{align}
\|\RR^{(1)}_{q+1}\|_{L^\infty_t L^1}\lesssim {\ell^{-8}_q}\lambda^{-\frac{3}{16}}_{q+1}, \quad \|\RR^{(1)}_{q+1}\|_{L^\infty_t W^{5,1}}\lesssim {\ell^{-8}_q}\lambda^{\frac{77}{16}}_{q+1},\quad \spt_x \RR^{(1)}_{q+1}\subseteq\Omega_{q+1},\label{est-Rq1}
\end{align}
and
\begin{align}
    \|F^{(1)}_{q+1}\|_{\widetilde{L}^\infty_t B^{-\frac{3}{2}}_{2,1}}\lesssim\ell^{-22}_q \lambda^{-\frac{1}{32}}_{q+1},\quad \|F^{(1)}_{q+1}\|_{{L}^\infty_t H^4}\lesssim \ell^{-10}_q\lambda^{\frac{35}{8}}_{q+1}.\label{est-Fq1}
\end{align}
\end{proposition}
\begin{proof} By \eqref{phi-1} and \eqref{psi-1}, we have
\begin{align*}
\mathbb{P}_{=0}(\phi^2_{k, \bar{k}, \bar{\bar{k}}, \sigma})=1, \quad \mathbb{P}_{=0}(\psi^2_{k})=1.
\end{align*}
According to \eqref{def-wpq-1}, we obtain that
\begin{align*}
&\Div(\wpq\otimes \wpq)+\Div\RR_{\ell_q}\\
=&\sum_{k\in\Lambda}\Div\big(a^2_{(k,q)} \phi^2_{k, \bar{k}, \bar{\bar{k}}, \sigma} \psi^2_k\bar{k}\otimes \bar{k}\big)+\Div\RR_{\ell_q}\\
=&\sum_{k\in\Lambda} \Div\big(a^2_{(k,q)} \bar{k}\otimes \bar{k}\big)+\Div\RR_{\ell_q}+\sum_{k\in\Lambda} \Div\big(a^2_{(k,q)} \phi^2_{k,\bar{k},\bar{\bar{k}},\sigma} \mathbb{P}_{\neq 0}(\psi^2_k)\bar{k}\otimes \bar{k}\big)\\
&+\sum_{k\in\Lambda}\Div\big(a^2_{(k,q)} \mathbb{P}_{\neq 0}( \phi^2_{k, \bar{k}, \bar{\bar{k}}, \sigma} ) \bar{k}\otimes \bar{k}\big).
\end{align*}
By  the definition of $a_{(k,q)}$ in \eqref{def-akq} and Lemma \ref{first S}, we have
\begin{align*}
&\sum_{k\in\Lambda} \Div\big(a^2_{(k,q)} \bar{k}\otimes \bar{k}\big)+\Div\RR_{\ell_q}=\Div\big( \eta^2_q(\chi_q \rho_q){\rm Id }-\eta^2_q\RR_{\ell_q}\big)+\Div\RR_{\ell_q}
=\nabla\big(\eta^2_q \chi_q \rho_q\big),
\end{align*}
where we have used \eqref{R=Rq}. Thanks to $k\perp\bar{k}$, one obtains that
\begin{align*}
\sum_{k\in\Lambda} \Div\big(a^2_{(k,q)} \phi^2_{k,\bar{k},\bar{\bar{k}},\sigma} \mathbb{P}_{\neq 0}(\psi^2_k)\bar{k}\otimes \bar{k}\big)=\sum_{k\in\Lambda} \Div\big(a^2_{(k,q)}  \phi^2_{k,\bar{k},\bar{\bar{k}},\sigma}\bar{k}\otimes \bar{k} \big)\mathbb{P}_{\neq 0}(\psi^2_k).
\end{align*}
By making use of
\begin{align*}
&\sum_{k\in\Lambda}\Div\big(a^2_{(k,q)} \mathbb{P}_{\neq 0}( \phi^2_{k, \bar{k}, \bar{\bar{k}}, \sigma} ) \bar{k}\otimes \bar{k}\big)\\
=& \sum_{k\in\Lambda}\Div\big(a^2_{(k,q)}  \bar{k}\otimes \bar{k}\big)   \mathbb{P}_{\neq 0}( \phi^2_{k, \bar{k}, \bar{\bar{k}}, \sigma} )+ \sum_{k\in\Lambda}a^2_{(k,q)} \Div\big(  \mathbb{P}_{\neq 0}( \phi^2_{k, \bar{k}, \bar{\bar{k}}, \sigma} )\bar{k}\otimes \bar{k}\big),
\end{align*}
we have
\begin{align}
&\Div(\wpq\otimes \wpq)+\Div\RR_{\ell_q}\notag\\
=&\nabla\big(\eta^2_q \chi_q \rho_q\big)+\sum_{k\in\Lambda}\Div\big(a^2_{(k,q)}  \bar{k}\otimes \bar{k}\big)   \mathbb{P}_{\neq 0}( \phi^2_{k, \bar{k}, \bar{\bar{k}}, \sigma} )\notag\\
&+ \sum_{k\in\Lambda}a^2_{(k,q)} \Div\big(  \mathbb{P}_{\neq 0}( \phi^2_{k, \bar{k}, \bar{\bar{k}}, \sigma} )\bar{k}\otimes \bar{k}\big)+\sum_{k\in\Lambda} \Div\big(a^2_{(k,q)}  \phi^2_{k,\bar{k},\bar{\bar{k}},\sigma}\bar{k}\otimes \bar{k} \big)\mathbb{P}_{\neq 0}(\psi^2_k)\notag\\
=:&\nabla\big(\eta^2_q \chi_q \rho_q\big)+\sum_{k\in\Lambda}\Div\big(a^2_{(k,q)}  \bar{k}\otimes \bar{k}\big)   \mathbb{P}_{\neq 0}( \phi^2_{k, \bar{k}, \bar{\bar{k}}, \sigma} )+{\rm I+II}.\label{wpwp}
\end{align}
As for ${\rm I}$, one infers that
\begin{align*}
   {\rm I}=&\sum_{k \in\Lambda }a^2_{(k,q)} \Big(\mathbb{P}_{H}+\frac{\nabla\Div}{\Delta}\Big)\Div(     \mathbb{P}_{\neq 0}( \phi^2_{k,\bar{k},\bar{\bar{k}},\sigma}) \bar{k}\otimes\bar{k} )   \notag\\
   =&\frac{1}{\mu}\sum_{k \in\Lambda }a^2_{(k,q)}\partial_t\Ph  ( \phi^2_{k,\bar{k},\bar{\bar{k}},\sigma} \bar{k} )+\sum_{k \in\Lambda }a^2_{(k,q)}\frac{\nabla\Div\Div}{\Delta}(   \phi^2_{k,\bar{k},\bar{\bar{k}},\sigma} \bar{k}\otimes\bar{k} ) \notag,
\end{align*}
where we have used  \eqref{div-pt}. By the definition of $\wttq$, we have
\begin{align*}
&\frac{1}{\mu}\sum_{k \in\Lambda }a^2_{(k,q)}\partial_t\Ph  ( \phi^2_{k,\bar{k},\bar{\bar{k}},\sigma} \bar{k} )\\
=&-\partial_t\wttq-\frac{1}{\mu}\sum_{k \in\Lambda }\partial_ta^2_{(k,q)}\Ph  ( \phi^2_{k,\bar{k},\bar{\bar{k}},\sigma} \bar{k} )\\
=&-\partial_t\wttq-\frac{1}{\mu}\sum_{k \in\Lambda }\partial_ta^2_{(k,q)}  \Big({\rm  Id}-\frac{\nabla\Div}{\Delta}\Big)\PP ( \phi^2_{k,\bar{k},\bar{\bar{k}},\sigma} \bar{k} )\\
=&-\partial_t\wttq-\frac{1}{\mu}\sum_{k \in\Lambda }\partial_ta^2_{(k,q)} \PP ( \phi^2_{k,\bar{k},\bar{\bar{k}},\sigma} \bar{k} )\\
&+\frac{1}{\mu}\sum_{k \in\Lambda }\nabla\Big(\partial_ta^2_{(k,q)}\frac{\Div}{\Delta}\PP ( \phi^2_{k,\bar{k},\bar{\bar{k}},\sigma} \bar{k} ) \Big)-\frac{1}{\mu}\sum_{k \in\Lambda }\nabla\big(\partial_ta^2_{(k,q)}\big)\frac{\Div}{\Delta}\PP ( \phi^2_{k,\bar{k},\bar{\bar{k}},\sigma} \bar{k} ) .
\end{align*}
Moreover, we rewrite
\begin{align*}
&\sum_{k \in\Lambda }a^2_{(k,q)}\frac{\nabla\Div\Div}{\Delta}(   \phi^2_{k,\bar{k},\bar{\bar{k}},\sigma} \bar{k}\otimes\bar{k} )\\
=&\sum_{k \in\Lambda }\nabla\big(a^2_{(k,q)} \frac{\Div\Div}{\Delta}(   \phi^2_{k,\bar{k},\bar{\bar{k}},\sigma} \bar{k}\otimes\bar{k} )\big)-\sum_{k \in\Lambda }\nabla a^2_{(k,q)}\frac{\Div\Div}{\Delta}(   \phi^2_{k,\bar{k},\bar{\bar{k}},\sigma} \bar{k}\otimes\bar{k} ).
\end{align*}
Hence, we have
\begin{align}
&\sum_{k \in\Lambda }a^2_{(k,q)} \Div(   \mathbb{P}_{\neq 0}( \phi^2_{k,\bar{k},\bar{\bar{k}},\sigma})\bar{k}\otimes \bar{k} )  \notag\\
&-\partial_t\wttq-\frac{1}{\mu}\sum_{k \in\Lambda }\partial_ta^2_{(k,q)} \PP ( \phi^2_{k,\bar{k},\bar{\bar{k}},\sigma} \bar{k} )-\frac{1}{\mu}\sum_{k \in\Lambda }\nabla\big(\partial_ta^2_{(k,q)}\big)\frac{\Div}{\Delta}\PP ( \phi^2_{k,\bar{k},\bar{\bar{k}},\sigma} \bar{k} )  \notag\\
&+\nabla\Big(\sum_{k \in\Lambda }\frac{1}{\mu}\partial_ta^2_{(k,q)}\frac{\Div}{\Delta}\PP ( \phi^2_{k,\bar{k},\bar{\bar{k}},\sigma} \bar{k}) +\sum_{k \in\Lambda }a^2_{(k,q)} \frac{\Div\Div}{\Delta}(   \phi^2_{k,\bar{k},\bar{\bar{k}},\sigma} \bar{k}\otimes\bar{k} ) \Big) \notag\\
&-\sum_{k \in\Lambda }\nabla a^2_{(k,q)}\frac{\Div\Div}{\Delta}(   \phi^2_{k,\bar{k},\bar{\bar{k}},\sigma} \bar{k}\otimes\bar{k} ).\label{aPphi}
\end{align}
Plugging \eqref{aPphi} into \eqref{wpwp} yields that
\begin{align}
&\Div(\wpq\otimes \wpq)+\Div\RR_{\ell_q}\notag\\
=&\sum_{k\in\Lambda} \Div\big(a^2_{(k,q)} \phi^2_{k,\bar{k},\bar{\bar{k}},\sigma} \bar{k}\otimes \bar{k} \big)\mathbb{P}_{\neq 0}(\psi^2_k)+\sum_{k\in\Lambda}\Div\big(a^2_{(k,q)}  \bar{k}\otimes \bar{k}\big)   \mathbb{P}_{\neq 0}( \phi^2_{k, \bar{k}, \bar{\bar{k}}, \sigma} )\notag\\
&-\partial_t\wttq-\frac{1}{\mu}\sum_{k \in\Lambda }\partial_ta^2_{(k,q)} \PP ( \phi^2_{k,\bar{k},\bar{\bar{k}},\sigma} \bar{k} )-\frac{1}{\mu}\sum_{k \in\Lambda }\nabla\big(\partial_ta^2_{(k,q)}\big)\frac{\Div}{\Delta}\PP ( \phi^2_{k,\bar{k},\bar{\bar{k}},\sigma} \bar{k} )  \notag\\
&+\nabla\Big(\eta^2_q\chi_q\rho_q+\sum_{k \in\Lambda }\frac{1}{\mu}\partial_ta^2_{(k,q)}\frac{\Div}{\Delta}\PP ( \phi^2_{k,\bar{k},\bar{\bar{k}},\sigma} \bar{k}) +\sum_{k \in\Lambda }a^2_{(k,q)} \frac{\Div\Div}{\Delta}(   \phi^2_{k,\bar{k},\bar{\bar{k}},\sigma} \bar{k}\otimes\bar{k} ) \Big) \notag\\
&-\sum_{k \in\Lambda }\nabla a^2_{(k,q)}\frac{\Div\Div}{\Delta}(   \phi^2_{k,\bar{k},\bar{\bar{k}},\sigma} \bar{k}\otimes\bar{k} ).\notag
\end{align}
With regard to ${\rm II}$, we have
\begin{align*}
    &{\rm II}
    =\sum_{k\in\Lambda}\Div\big(a^2_{(k,q)} \phi^2_{k,\bar{k},\bar{\bar{k}},\sigma} \bar{k}\otimes \bar{k} \big)\sum_{l\neq 0}c_{l}e^{\ii2\pi N_{\Lambda}\lambda_{q+1}lk\cdot x}.
\end{align*}
Making use of Lemma \ref{tracefree} with $G=\Div\big(a^2_{(k,q)} \phi^2_{k,\bar{k},\bar{\bar{k}},\sigma} \bar{k}\otimes \bar{k} \big)$ and $\rho^{(n)}=(2\pi)^{-n+1}(-1)^ne^{\ii 2\pi k\cdot x}$,   one concludes  that there exist a trace-free matrix $\RR_{(N)}$, pressure $P_{(N)}$ and the error term $E_{(N)}$ such that
\begin{align}\label{decom-(a-phi)}
{\rm II}=\Div \RR_{(N)}+\nabla P_{(N)}+E_{(N)},
\end{align}
where
\begin{align*}
    &\RR^{ij}_{(N)}=\sum_{k\in\Lambda}\sum_{l\neq 0}b_{l}\sum_{n=1}^{N} (N_{\Lambda}l\lambda_{q+1})^{-n}A^{(n)}_{ij\alpha_n}(G)\big(\frac{\partial^{\alpha_n}}{{\Delta}^n}e^{\ii k\cdot x}\big)\circ{(N_{\Lambda}l\lambda_{q+1} x)},\\
    &P_{(N)}=\sum_{k\in\Lambda}\sum_{l\neq 0}b_{l}\sum_{n=1}^{N} (N_{\Lambda}l\lambda_{q+1})^{-n}B^{(n)}_{\alpha_n}(G)\big(\frac{\partial^{\alpha_n}}{{\Delta}^n}e^{\ii k\cdot x}\big)\circ{(N_{\Lambda}l\lambda_{q+1} x)}\\
    &E^i_{(N)}=\sum_{k\in\Lambda}\sum_{l\neq 0}b_{l}(N_{\Lambda}l\lambda_{q+1})^{-N}C^{(N)}_{i\alpha_N}(G) \big(\frac{\partial^{\alpha_N}}{{\Delta}^N}e^{\ii k\cdot x}\big)\circ{(N_{\Lambda}l\lambda_{q+1} x)}.
\end{align*}
Moreover, we infer from Lemma \ref{tracefree} that
\[\spt_x \RR_{(N)}=\spt_x a^2_{(k,q)}\subseteq\spt \,\eta_{q}=\Omega_{q+1}.\]
Particularly, we set {{$N=20$}} to obtain that
\begin{align}
&\Div(\wpq\ootimes \wpq)+\Div\RR_{\ell_q}=\Div(\wpq\otimes \wpq)+\Div\RR_{\ell_q}-\nabla \Big(\frac{1}{3}\wpq\cdot \wpq {\rm Id}\Big)\notag\\
=&\Div \RR_{(N)}-\partial_t\wttq+\Big(E_{(N)}+\sum_{k\in\Lambda}\Div\big(a^2_{(k,q)}  \bar{k}\otimes \bar{k}\big)   \mathbb{P}_{\neq 0}( \phi^2_{k, \bar{k}, \bar{\bar{k}}, \sigma} )-\frac{1}{\mu}\sum_{k \in\Lambda }\partial_ta^2_{(k,q)} \PP ( \phi^2_{k,\bar{k},\bar{\bar{k}},\sigma} \bar{k} )\notag\\
&-\frac{1}{\mu}\sum_{k \in\Lambda }\nabla\big(\partial_ta^2_{(k,q)}\big)\frac{\Div}{\Delta}\PP ( \phi^2_{k,\bar{k},\bar{\bar{k}},\sigma} \bar{k} ) -\sum_{k \in\Lambda }\nabla a^2_{(k,q)}\frac{\Div\Div}{\Delta}(   \phi^2_{k,\bar{k},\bar{\bar{k}},\sigma} \bar{k}\otimes\bar{k} )\Big)\notag\\
&+\nabla\Big(\eta^2_q\chi_q\rho_q+\sum_{k \in\Lambda }\frac{1}{\mu}\partial_ta^2_{(k,q)}\frac{\Div}{\Delta}\PP ( \phi^2_{k,\bar{k},\bar{\bar{k}},\sigma} \bar{k}) +\sum_{k \in\Lambda }a^2_{(k,q)} \frac{\Div\Div}{\Delta}(   \phi^2_{k,\bar{k},\bar{\bar{k}}} \bar{k}\otimes\bar{k} )\notag\\
&+P_{(N)}-\frac{1}{3}\wpq\cdot \wpq {\rm Id}\Big) \notag\\
=:&\Div \RR^{(1)}_{(q+1)}-\partial_t\wttq+F^{(1)}_{q+1}+\nabla P^{(1)}_{q+1},\label{eq-wpwp}
\end{align}
which proves \eqref{decom-wpwp} and $\spt_x \RR^{(1)}_{q+1}\subseteq\Omega_{q+1}$. Here
\begin{align*}
F^{(1)}_{q+1}=& E_{(N)}+\sum_{k\in\Lambda}\Div\big(a^2_{(k,q)}  \bar{k}\otimes \bar{k}\big)   \mathbb{P}_{\neq 0}( \phi^2_{k, \bar{k}, \bar{\bar{k}}, \sigma} )-\frac{1}{\mu}\sum_{k \in\Lambda }\partial_ta^2_{(k,q)} \PP ( \phi^2_{k,\bar{k},\bar{\bar{k}},\sigma} \bar{k} )\notag\\
&-\frac{1}{\mu}\sum_{k \in\Lambda }\nabla\big(\partial_ta^2_{(k,q)}\big)\frac{\Div}{\Delta}\PP ( \phi^2_{k,\bar{k},\bar{\bar{k}},\sigma} \bar{k} ) -\sum_{k \in\Lambda }\nabla a^2_{(k,q)}\frac{\Div\Div}{\Delta}(   \phi^2_{k,\bar{k},\bar{\bar{k}},\sigma} \bar{k}\otimes\bar{k} ).
\end{align*}

Now we estimate $\RR^{(1)}_{q+1}$. Since $A^{(n)}_{ij\alpha_n}$ is composed by the $n-1$th derivative of $\Div\big(a^2_{(k,q)} \phi^2_{k,\bar{k},\bar{\bar{k}},\sigma} \bar{k}\otimes \bar{k} \big)$, Proposition~\ref{guji1} and Proposition \ref{est-ak} implies that
\begin{align*}
    \|\RR^{(1)}_{q+1}\|_{L^\infty_t L^1}\lesssim& \sum_{n=1}^{10}\lambda^{-n}_{q+1}\|A^{(n)}_{ij\alpha_n}\|_{L^\infty_t L^1}\lesssim\sum_{n=1}^{{20}}\lambda^{-n}_{q+1}\|a^2_{(k,q)} \phi^2_{k,\bar{k},\bar{\bar{k}},\sigma}\|_{L^\infty_t W^{n,1}}\\
    \lesssim&\sum_{n=1}^{20}\lambda^{-n}_{q+1}
    (\|a^2_{(k,q)}\|_{L^\infty_{t,x}} \|\phi^2_{k,\bar{k},\bar{\bar{k}},\sigma}\|_{L^\infty_t W^{n,1}(\TTT^3)}+\|a^2_{(k,q)}\|_{L^\infty_t W^{n,1} } \|\phi^2_{k,\bar{k},\bar{\bar{k}},\sigma}\|_{L^\infty(\TTT^3)})\\
    \lesssim&\sum_{n=1}^{{20}}\lambda^{-\frac{3}{16}n}_{q+1}{\ell^{-8}_q}\lesssim\lambda^{-\frac{3}{16}}_{q+1}{\ell^{-8}_q},
\end{align*}
and
\begin{align*}
    \|\RR^{(1)}_{q+1}\|_{L^\infty_t W^{5,1}}\lesssim& \sum_{n=1}^{{20}}\lambda^{-n}_{q+1}(\|a^2_{(k,q)} \phi^2_{k,\bar{k},\bar{\bar{k}},\sigma}\|_{L^\infty_t W^{5+n,1}}+\lambda^5_{q+1}\|a^2_{(k,q)} \phi^2_{k,\bar{k},\bar{\bar{k}},\sigma}\|_{L^\infty_t W^{n,1}})\\
    \lesssim&\sum_{n=1}^{10}{\ell^{-8}_q}\lambda^{-\frac{3}{16}n+5}_{q+1}\lesssim\lambda^{\frac{77}{16}}_{q+1}{\ell^{-8}_q},
\end{align*}
where we use the fact that $\spt_x a_{(k,q)}\subseteq \Omega_q\subseteq [-\frac{1}{2},\frac{1}{2}]^3$ and $$\|a_{(k,q)}\cdot f\|_{L^{p}(\mathbb{R}^3)}\leq\|a_{(k,q)}\|_{L^p_1(\mathbb{R}^3)}\| f\|_{L^{p_2}(\Omega_q)},\quad\frac{1}{p_1}+\frac{1}{{p_2}}=\frac{1}{p}.$$
 We will apply this inequality multiple times in the future chapters.  Hence we derive \eqref{est-Rq1}. Now we aim to estimate $F^{(1)}_{q+1}$. Note that $C^{(N)}_{i\alpha}$ is made up of the $n$th derivative of $\Div\big(a^2_{(k,q)} \phi^2_{k,\bar{k},\bar{\bar{k}},\sigma} \bar{k}\otimes \bar{k} \big)$, we have
\begin{align}
    \|E_{(N)}\|_{\widetilde{L}^\infty_t B^{-\frac{3}{2}}_{2,1}}\lesssim& \|E_{(N)}\|_{{L}^\infty_t L^{\frac{3}{2}}}\lesssim\lambda^{-N}_{q+1}\|a^2_{(k,q)} \phi^2_{k,\bar{k},\bar{\bar{k}},\sigma}\|_{L^\infty W^{N+1,\frac{3}{2}}}\nonumber\\
    \lesssim&\ell^{-8}_q\lambda^{-N}_{q+1}\lambda^{\frac{13}{16}(N+1)+\frac{3}{4}}_{q+1}\lesssim\ell^{-8}_q\lambda^{-\frac{35}{16}}_{q+1}.\label{F1-1-B}
\end{align}
and
\begin{align}
    \|E_{(N)}\|_{{L}^\infty_t H^4}\lesssim& \lambda^{-N}_{q+1}\|a^2_{(k,q)} \phi^2_{k,\bar{k},\bar{\bar{k}},\sigma}\|_{L^\infty_t H^5}+\lambda^{-N+5}_{q+1}\|a^2_{(k,q)} \phi^2_{k,\bar{k},\bar{\bar{k}},\sigma}\|_{L^\infty_t H^1}\nonumber\\
    \lesssim&\ell^{-8}_q\lambda^{-N}_{q+1}\lambda^{\frac{83}{16}}_{q+1}+\ell^{-8}_q\lambda^{-N+\frac{111}{16}}_{q+1}\lesssim \ell^{-8}_q\lambda^{-10}_{q+1}.\label{F1-1-W}
\end{align}
In terms of $\sum_{k\in\Lambda}\Div\big(a^2_{(k,q)}  \bar{k}\otimes \bar{k}\big)   \mathbb{P}_{\neq 0}( \phi^2_{k, \bar{k}, \bar{\bar{k}}, \sigma} )$, we rewrite it as
\begin{align*}
&\sum_{k\in\Lambda}\Div\big(a^2_{(k,q)}  \bar{k}\otimes \bar{k}\big)   \mathbb{P}_{\neq 0}( \phi^2_{k, \bar{k}, \bar{\bar{k}}, \sigma} )
=\sum_{k\in\Lambda}\Div\big(a^2_{(k,q)}  \bar{k}\otimes \bar{k}\big)\frac{\Div \nabla}{\Delta}   ( \phi^2_{k, \bar{k}, \bar{\bar{k}}, \sigma} )\\
=&\sum_{k\in\Lambda}\Div\big(a^2_{(k,q)}  \bar{k}\otimes \bar{k}\big)\Div\Big(\sigma^{-1}\big(\frac{ \nabla}{\Delta}  \phi^2_{k, \bar{k}, \bar{\bar{k}}}\big)(\sigma x) \Big)\\
=&\sum_{k\in\Lambda}\Div\Big(\sigma^{-1}\big(\frac{ \nabla}{\Delta}  \phi^2_{k, \bar{k}, \bar{\bar{k}}}\big)(\sigma x) \otimes \Div\big(a^2_{(k,q)}  \bar{k}\otimes \bar{k}\big)\Big)\\
&-\sum_{k\in\Lambda}\sigma^{-1}\big(\frac{ \nabla}{\Delta}  \phi^2_{k, \bar{k}, \bar{\bar{k}}}\big)(\sigma x)\cdot\nabla \Div\big(a^2_{(k,q)}  \bar{k}\otimes \bar{k}\big).
\end{align*}
Since $spt_xa_{(k,q)}\subset \Omega_q\subset [-\frac{1}{2},\frac{1}{2}]^3$, we have
\begin{align*}
&\big\|\sum_{k\in\Lambda}\Div\big(a^2_{(k,q)}  \bar{k}\otimes \bar{k}\big)   \mathbb{P}_{\neq 0}( \phi^2_{k, \bar{k}, \bar{\bar{k}}, \sigma} )\big\|_{\widetilde{L}^\infty_t B^{-\frac{3}{2}}_{2,1}}\\
\le&C\sum_{k\in\Lambda}\big\|\sigma^{-1}\big(\frac{ \nabla}{\Delta}  \phi^2_{k, \bar{k}, \bar{\bar{k}}}\big)(\sigma x) \otimes \Div\big(a^2_{(k,q)}  \bar{k}\otimes \bar{k}\big)\big\|_{\widetilde{L}^\infty_t B^{-\frac{1}{2}}_{2,1}}\\
&+C\sum_{k\in\Lambda}\big\|\sigma^{-1}\big(\frac{ \nabla}{\Delta}  \phi^2_{k, \bar{k}, \bar{\bar{k}}}\big)(\sigma x)\cdot\nabla \Div\big(a^2_{(k,q)}  \bar{k}\otimes \bar{k}\big)\big\|_{\widetilde{L}^\infty_t B^{-\frac{3}{2}}_{2,1}}\\
\le&C\sigma^{-1}\sum_{k\in\Lambda}\big\|\big(\frac{ \nabla}{\Delta}  \phi^2_{k, \bar{k}, \bar{\bar{k}}}\big)(\sigma x) \otimes \Div\big(a^2_{(k,q)}  \bar{k}\otimes \bar{k}\big)\big\|_{{L}^\infty_t L^{\frac{3(1+\varepsilon)}{2-\varepsilon}}}\\
&+C\sigma^{-1}\sum_{k\in\Lambda}\big\|\big(\frac{ \nabla}{\Delta}  \phi^2_{k, \bar{k}, \bar{\bar{k}}}\big)(\sigma x)\cdot\nabla \Div\big(a^2_{(k,q)}  \bar{k}\otimes \bar{k}\big)\big\|_{{L}^\infty_t L^{\frac{3}{2}}}\\
\le&C\sigma^{-1}\sum_{k\in\Lambda}\big\|\frac{ \nabla}{\Delta}  \phi^2_{k, \bar{k}, \bar{\bar{k}}}\big\|_{{L}^\infty_t W^{1,1+\varepsilon}(\mathbb{T}^3)} \| \Div\big(a^2_{(k,q)}  \bar{k}\otimes \bar{k}\big)\big\|_{{L}^\infty_{t,x}}\\
&+C\sigma^{-1}\sum_{k\in\Lambda}\big\|\frac{ \nabla}{\Delta}  \phi^2_{k, \bar{k}, \bar{\bar{k}}}\big\|_{{L}^\infty_t W^{1,1}(\mathbb{T}^3)} \| \nabla\Div\big(a^2_{(k,q)}  \bar{k}\otimes \bar{k}\big)\big\|_{{L}^\infty_{t,x}},
\end{align*}
where $\varepsilon$ is some small constant and we use the fact that
$$\|f\cdot g\|_{L^p(\mathbb{R}^3)}\leq\|f \|_{L^p(\Omega)}\|g\|_{L^p(\Omega)},~~~~\text{if}~~\text{spt}_x~g\subset \Omega.$$
Using Proposition \ref{guji1}, we conclude that
\begin{align}
&\big\|\sum_{k\in\Lambda}\Div\big(a^2_{(k,q)}  \bar{k}\otimes \bar{k}\big)   \mathbb{P}_{\neq 0}( \phi^2_{k, \bar{k}, \bar{\bar{k}}, \sigma} )\big\|_{\widetilde{L}^\infty_t B^{-\frac{3}{2}}_{2,1}}
\le C\ell^{-12}_q\sigma^{-1}\lambda^{\frac{9\varepsilon}{4(1+\varepsilon)}}_{q+1}\le C\ell^{-12}_q\lambda^{-\frac{1}{32}}_{q+1},\label{F1-2-B}
\end{align}
where $\varepsilon<\frac{1}{71}$. Following this method, one immediately infers that
\begin{align}
&\big\|\sum_{k \in\Lambda }\nabla a^2_{(k,q)}\frac{\Div\Div}{\Delta}(   \phi^2_{k,\bar{k},\bar{\bar{k}},\sigma} \bar{k}\otimes\bar{k} )\big\|_{\widetilde{L}^\infty_t B^{-\frac{3}{2}}_{2,1}}
\le C\ell^{-12}_q\sigma^{-1}\lambda^{\frac{9\varepsilon}{4(1+\varepsilon)}}_{q+1}\le C\ell^{-12}_q\lambda^{-\frac{1}{32}}_{q+1},\label{F1-3-B}
\end{align}
By Proposition \ref{guji1} and ${L}^\infty_t L^{\frac{9}{8}}(\R^3)\hookrightarrow \widetilde{L}^\infty_t B^{-\frac{3}{2}}_{2,1}(\R^3)$, we have
\begin{align}
    &\frac{1}{\mu}\big\|\sum_{k \in\Lambda }\partial_ta^2_{(k,q)} \PP ( \phi^2_{k,\bar{k},\bar{\bar{k}},\sigma} \bar{k} )\big\|_{\widetilde{L}^\infty_t B^{-\frac{3}{2}}_{2,1}}+
    \frac{1}{\mu}\big\|\sum_{k \in\Lambda }\nabla\big(\partial_ta^2_{(k,q)}\big)\frac{\Div}{\Delta}\PP ( \phi^2_{k,\bar{k},\bar{\bar{k}},\sigma} \bar{k} )\big\|_{\widetilde{L}^\infty_t B^{-\frac{3}{2}}_{2,1}}\nonumber\\
    \le&C{\mu}^{-1}\sum_{k \in\Lambda }\big(\big\|\partial_ta^2_{(k,q)} \PP ( \phi^2_{k,\bar{k},\bar{\bar{k}},\sigma} \bar{k} )\big\|_{{L}^\infty_t L^{\frac{9}{8}}}+\big\|\nabla\big(\partial_ta^2_{(k,q)}\big)\frac{\Div}{\Delta}\PP ( \phi^2_{k,\bar{k},\bar{\bar{k}},\sigma} \bar{k} )\big\|_{{L}^\infty_t L^{\frac{9}{8}}}\big)\nonumber\\
    \le&C{\mu}^{-1}\sum_{k \in\Lambda }(\|\partial_ta^2_{(k,q)}\|_{L^\infty_{t,x}}+\|\nabla\partial_ta^2_{(k,q)}\|_{L^\infty_{t,x}})\|\PP ( \phi^2_{k,\bar{k},\bar{\bar{k}},\sigma} \bar{k} )\big\|_{{L}^\infty_t L^{\frac{9}{8}}}\nonumber\\
    \le&C{\mu}^{-1}\ell^{-22}_q\lambda^{\frac{1}{4}}_{q+1}\le C\ell^{-22}_q \lambda^{-1}_{q+1}.\label{F1-4-B}
\end{align}
With the help of Proposition \ref{guji1} and Proposition \ref{est-ak}, we have
\begin{align}
&\big\|\sum_{k\in\Lambda}\Div\big(a^2_{(k,q)}  \bar{k}\otimes \bar{k}\big)   \mathbb{P}_{\neq 0}( \phi^2_{k, \bar{k}, \bar{\bar{k}}, \sigma} )\big\|_{{L}^\infty_t H^4}+\big\|\sum_{k \in\Lambda }\nabla a^2_{(k,q)}\frac{\Div\Div}{\Delta}(   \phi^2_{k,\bar{k},\bar{\bar{k}},\sigma} \bar{k}\otimes\bar{k} )\big\|_{{L}^\infty_t H^4}\nonumber\\
&+\frac{1}{\mu}\big\|\sum_{k \in\Lambda }\partial_ta^2_{(k,q)} \PP ( \phi^2_{k,\bar{k},\bar{\bar{k}},\sigma} \bar{k} )\big\|_{{L}^\infty_t H^4}+
    \frac{1}{\mu}\big\|\sum_{k \in\Lambda }\nabla\big(\partial_ta^2_{(k,q)}\big)\frac{\Div}{\Delta}\PP ( \phi^2_{k,\bar{k},\bar{\bar{k}},\sigma} \bar{k} )\big\|_{{L}^\infty_t H^4}\nonumber\\
\lesssim&\|a^2_{(k,q)}  \|_{L^\infty_t W^{1,\infty}}(\|\phi^2_{k, \bar{k}, \bar{\bar{k}}, \sigma}\|_{L^\infty_t H^4}+1)+\|a^2_{(k,q)}  \|_{L^\infty_t W^{5,\infty}}(\|\phi^2_{k, \bar{k}, \bar{\bar{k}}, \sigma}\|_{L^\infty_t L^2}+1)\nonumber\\
&+\frac{1}{\mu}\|\partial_ta^2_{(k,q)}  \|_{L^\infty_t W^{1,\infty}}(\|\phi^2_{k, \bar{k}, \bar{\bar{k}}, \sigma}\|_{L^\infty_t H^4}+1)+\frac{1}{\mu}\|\partial_ta^2_{(k,q)}  \|_{L^\infty_t W^{5,\infty}}(\|\phi^2_{k, \bar{k}, \bar{\bar{k}}, \sigma}\|_{L^\infty_t L^2}+1)\nonumber\\
\lesssim& \ell^{-10}_q\lambda^{\frac{35}{8}}_{q+1}.\label{F1-2-W}
\end{align}
Therefore, collecting \eqref{F1-1-B} and \eqref{F1-2-B}--\eqref{F1-4-B} together yields that
\begin{align*}
    \|F^{(1)}_{q+1}\|_{\widetilde{L}^\infty_t B^{-\frac{3}{2}}_{2,1}}\le C\ell^{-22}_q \lambda^{-\frac{1}{32}}_{q+1}.
\end{align*}
By \eqref{F1-1-W} and \eqref{F1-2-W}, we conclude that
\begin{align*}
    \|F^{(1)}_{q+1}\|_{{L}^\infty_t H^4}\lesssim \ell^{-10}_q\lambda^{\frac{35}{8}}_{q+1}.
\end{align*}
Therefore, we show \eqref{est-Fq1} and complete the proof of Proposition \ref{def-F2}.
\end{proof}

\begin{proposition}[The decomposition of $\partial_t (\wpq+\wcq)$]\label{F1}Let
$\wpq$ and $\wcq$ be defined in \eqref{def-wpq} and \eqref{def-wcq} respectively, then there exist a trace-free matrix $\RR^{(2)}_{(q+1)}$, pressure $P^{(2)}_{(q+1)}$ and the error term $F^{(2)}_{q+1}$ such that
\begin{align}
    \partial_t (\wpq+\wcq)=\Div \RR^{(2)}_{q+1}+\nabla P^{(2)}_{(q+1)}+F^{(2)}_{q+1}\label{decom-wp+wc}
\end{align}
satisfying
\begin{align}
\|\RR^{(2)}_{q+1}\|_{L^\infty_t L^1}
\lesssim\ell^{-12}_q\lambda^{-\frac{1}{16}}_{q+1},\quad \|\RR^{(2)}_{q+1}\|_{L^\infty_t W^{5,1}}\lesssim \ell^{-4}_q\lambda^{\frac{79}{16}}_{q+1},\quad \spt_x \RR^{(2)}_{q+1}\subseteq\Omega_{q+1}. \label{est-Rq2}
\end{align}
and
\begin{align}
\|F^{(2)}_{q+1}\|_{\widetilde{L}^\infty_t B^{-\frac{3}{2}}_{2,1}} \lesssim  \ell^{-12}_q\lambda^{-\frac{9}{4}}_{q+1}, \quad  \|F^{(2)}_{q+1}\|_{{L}^\infty_t H^4} \le \ell^{-4}_q\lambda^{4}_{q+1},\label{est-Fq2}
\end{align}
\end{proposition}
\begin{proof}Thanks to \eqref{def-wpq} and \eqref{def-wcq}, we have
\begin{align}\label{wpq+wcq}
\wpq+\wcq=\sum_{k\in\Lambda}N^{-1}_{\Lambda}\lambda^{-1}_{q+1}{\Div\big( a_{(k,q)} \phi_{k,\bar{k},\bar{\bar{k}},\sigma}\Psi'_k(k\otimes \bar{k}-\bar{k}\otimes k)\big)}.
    \end{align}
and
\begin{align}
\partial_t(\wpq+\wcq)=&\sum_{k\in\Lambda}N^{-1}_{\Lambda}\lambda^{-1}_{q+1}{\Div\big(\partial_t(a_{(k,q)} \phi_{k,\bar{k},\bar{\bar{k}},\sigma})\Psi'_k(k\otimes \bar{k}-\bar{k}\otimes k)\big)}\nonumber\\
=&\sum_{k\in\Lambda}N^{-1}_{\Lambda}\lambda^{-1}_{q+1}{\Div\big(\partial_t(a_{(k,q)} \phi_{k,\bar{k},\bar{\bar{k}},\sigma})\Psi'_k(k\otimes \bar{k}+\bar{k}\otimes k)\big)}\nonumber\\
&-2\sum_{k\in\Lambda}N^{-1}_{\Lambda}\lambda^{-1}_{q+1}{\Div\big(\partial_t(a_{(k,q)} \phi_{k,\bar{k},\bar{\bar{k}},\sigma})\Psi'_k\bar{k}\otimes k\big)}.\label{ptwpq+wcq}
\end{align}
For the second term in the right side of \eqref{ptwpq+wcq}, we rewrite it as
\begin{align*}
    &-2\sum_{k\in\Lambda}N^{-1}_{\Lambda}\lambda^{-1}_{q+1}{\Div\big(\partial_t(a_{(k,q)} \phi_{k,\bar{k},\bar{\bar{k}},\sigma})\Psi'_k\bar{k}\otimes k\big)}\\
   =& -2\sum_{k\in\Lambda}N^{-1}_{\Lambda}\lambda^{-1}_{q+1}{\Div\big(\partial_t(a_{(k,q)} \phi_{k,\bar{k},\bar{\bar{k}},\sigma})\bar{k}\otimes k\big)}\Psi'_k\\
   =&-2\sum_{k\in\Lambda}N^{-1}_{\Lambda}\lambda^{-1}_{q+1}{\Div\big(\partial_t(a_{(k,q)} \phi_{k,\bar{k},\bar{\bar{k}},\sigma})\bar{k}\otimes k\big)}\sum_{l\neq 0}c_{l}e^{\ii 2\pi N_{\Lambda}\lambda_{q+1}lk\cdot x}.
\end{align*}
With the aid of Proposition \ref{tracefree}, there exist a  traceless symmetric stress $\mathring{R}_{(M)}$, pressure $P_{(M)}$ and error term $E_{(M)}$ such that
\begin{align}\label{line-eq}
-2\sum_{k\in\Lambda}N^{-1}_{\Lambda}\lambda^{-1}_{q+1}{\Div\big(\partial_t(a_{(k,q)} \phi_{k,\bar{k},\bar{\bar{k}},\sigma})\Psi'_k\bar{k}\otimes k\big)}=\Div \RR_{(M)}+\nabla P_{(M)}+E_{(M)},
\end{align}
where $\spt_x \RR_{(M)}=\spt_x a_{(k,q)}\subseteq\spt \eta_q=\Omega_{q+1}$,
\begin{align}
    &\RR^{ij}_{(M)}=-2\sum_{k\in\Lambda}\sum_{l\neq 0}\sum_{n=1}^{M} c_{l}l^{-n}(N_{\Lambda}\lambda_{q+1})^{-(n+1)}A^{(n)}_{ij\alpha_n}(G)\big(\frac{\partial^{\alpha_n}}{{\Delta}^n}e^{\ii k\cdot x}\big)\circ{(N_{\Lambda}l\lambda_{q+1} x)},\notag\\
     &P_{(M)}=-2\sum_{k\in\Lambda}\sum_{l\neq 0}\sum_{n=1}^{M}c_{l}l^{-n}(N_{\Lambda}\lambda_{q+1})^{-(M+1)}B^{(n)}_{\alpha_n}(G)\big(\frac{\partial^{\alpha_n}}{{\Delta}^n}e^{\ii k\cdot x}\big)\circ{(N_{\Lambda}l\lambda_{q+1} x)},\notag\\
    &E^i_{(M)}=-2\sum_{k\in\Lambda}\sum_{l\neq 0}c_{l}l^{-M}(N_{\Lambda}\lambda_{q+1})^{-(M+1)}C^{(M)}_{i\alpha}(G) \big(\frac{\partial^{\alpha}}{{\Delta}^n}e^{\ii k\cdot x}\big)\circ{(N_{\Lambda}l\lambda_{q+1} x)},\label{def-EM}\\
    &G=\Div\big(\partial_t(a_{(k,q)} \phi_{k,\bar{k},\bar{\bar{k}},\sigma})\bar{k}\otimes k\big).\notag
\end{align}
Plugging \eqref{line-eq} into \eqref{ptwpq+wcq} and  setting $M=20$, we obtain that
\begin{align*}
&\partial_t(\wpq+\wcq)\\
=&\Div\Big(\big(\sum_{k\in\Lambda}N^{-1}_{\Lambda}\lambda^{-1}_{q+1}{\partial_t(a_{(k,q)} \phi_{k,\bar{k},\bar{\bar{k}},\sigma})\Psi'_k(k\otimes \bar{k}+\bar{k}\otimes k)\big)+\RR_{(M)}\Big)}+\nabla P_{(M)}+E_{(M)}\\
=:&\Div \RR^{(2)}_{q+1}+\nabla P^{(2)}_{(q+1)}+F^{(2)}_{q+1}.
\end{align*}
Hence we  have the decomposition \eqref{decom-wp+wc} and $\spt_x \RR^{(2)}_{q+1}\subseteq\Omega_{q+1}$. Now we estimate $\RR^{(2)}_{q+1}$. By Proposition \ref{guji1} and Proposition \ref{est-ak}, one infers that
\begin{align*}
\|\RR^{(2)}_{q+1}\|_{L^\infty_t L^1}\le &\sum_{k\in\Lambda}\big\|N^{-1}_{\Lambda}\lambda^{-1}_{q+1}\partial_t(a_{(k,q)} \phi_{k,\bar{k},\bar{\bar{k}},\sigma})\Psi'_k\big\|_{L^\infty_t L^1}+\big\|\RR_{(M)}\big\|_{L^\infty_t L^1}\\
\lesssim&\sum_{k\in\Lambda}\lambda^{-1}_{q+1}\|\partial_t(a_{(k,q)} \phi_{k,\bar{k},\bar{\bar{k}},\sigma})\|_{L^\infty_t L^1}+\sum_{k\in\Lambda}\sum_{n=1}^{20}\lambda^{-(n+1)}_{{q+1}}\|\partial_t(a_{(k,q)} \phi_{k,\bar{k},\bar{\bar{k}},\sigma})\|_{L^\infty_t W^{n,1}}\\
\lesssim&\ell^{-12}_q\lambda^{-1}_{q+1}\lambda^{\frac{33}{16}-\frac{9}{8}}_{q+1}+\sum_{k\in\Lambda}\sum_{n=1}^{20}\ell^{-12}_q\lambda^{-(n+1)}_{{q+1}}\lambda^{\frac{13}{16}n+\frac{33}{16}-\frac{9}{8}}_{q+1}\\
\lesssim&\ell^{-12}_q\lambda^{-\frac{1}{16}}_{q+1},
\end{align*}
and
\begin{align*}
\|\RR^{(2)}_{q+1}\|_{L^\infty_t W^{5,1}}\le &\sum_{k\in\Lambda}\big\|N^{-1}_{\Lambda}\lambda^{-1}_{q+1}\partial_t(a_{(k,q)} \phi_{k,\bar{k},\bar{\bar{k}},\sigma})\Psi'_k\big\|_{L^\infty_t W^{5,1}}+\big\|\RR_{(M)}\big\|_{L^\infty_t W^{5,1}}\\
\lesssim&\sum_{k\in\Lambda}\lambda^{-1}_{q+1}\|\partial_t(a_{(k,q)} \phi_{k,\bar{k},\bar{\bar{k}},\sigma})\|_{L^\infty_t W^{5,1}}+\sum_{k\in\Lambda}\lambda^{4}_{q+1}\|\partial_t(a_{(k,q)} \phi_{k,\bar{k},\bar{\bar{k}},\sigma})\|_{L^\infty_t L^1}\\
&+\sum_{k\in\Lambda}\sum_{n=1}^{20}\lambda^{-(n+1)}_{{q+1}}(\|\partial_t(a_{(k,q)} \phi_{k,\bar{k},\bar{\bar{k}},\sigma})\|_{L^\infty_t W^{n+5,1}}+\lambda^5_{q+1}\|\partial_t(a_{(k,q)} \phi_{k,\bar{k},\bar{\bar{k}},\sigma})\|_{L^\infty_t W^{n,1}})\\
\lesssim&\ell^{-4}_q\lambda^{-1}_{q+1}\lambda^{\frac{65}{16}+\frac{33}{16}-\frac 98}_{q+1}+\ell^{-4}_q\lambda^4_{q+1}\lambda^{\frac{33}{16}-\frac 98}_{q+1}+\sum_{n=1}^{20}\ell^{-4}_q\lambda^{-(n+1)}_{{q+1}}\lambda^{5}_{q+1}\lambda^{\frac{13}{16}n+\frac{33}{16}-\frac 98}_{q+1}\\
\lesssim&\ell^{-4}_q\lambda^{\frac{79}{16}}_{q+1}.
\end{align*}
Therefore, we prove \eqref{est-Rq2}. With regard to $F^{(2)}_{q+1}$, by  Proposition \ref{est-ak} and Proposition \ref{guji1}, for $M=20$, we deduce from \eqref{def-EM} that
\begin{align*}
\|F^{(2)}_{q+1}\|_{\widetilde{L}^\infty_t B^{-\frac{3}{2}}_{2,1}} \lesssim &\|F^{(2)}_{q+1}\|_{{L}^\infty_t L^{\frac32}} \lesssim\lambda^{-(M+1)}_{q+1}\|\partial_t(a_{(k,q)} \phi_{k,\bar{k},\bar{\bar{k}},\sigma})\|_{L^\infty W^{M+1,\frac32}}\\
\lesssim &\ell^{-12}_q\lambda^{-(M+1)}_{q+1}\lambda^{\frac{13}{16}(M+1)+\frac{27}{16}}_{q+1}\lesssim \ell^{-12}_q\lambda^{-\frac{3}{16}(M+1)+\frac{27}{16}}_{q+1}\le \ell^{-12}_q\lambda^{-\frac{9}{4}}_{q+1},
\end{align*}
and
\begin{align*}
\|F^{(2)}_{q+1}\|_{{L}^\infty H^4} \lesssim &\lambda^{-(M+1)}_{q+1}(\|\partial_t(a_{(k,q)} \phi_{k,\bar{k},\bar{\bar{k}},\sigma})\|_{L^\infty H^{M+6}}+\lambda^5_{q+1}\|\partial_t(a_{(k,q)} \phi_{k,\bar{k},\bar{\bar{k}},\sigma})\|_{L^\infty H^{M+1}})\\
\lesssim &\ell^{-4}_q\lambda^{\frac{113}{16}-\frac{3}{16}(M+1)}_{q+1}\le \ell^{-4}_q\lambda^{4}_{q+1},
\end{align*}
which show \eqref{est-Fq2}. Hence, we complete the proof of Proposition \ref{F1}.
\end{proof}

Now we give the localized corrector $\wtq$ in more details. Let $\wtq$ be the solution of the equations \eqref{e:wt}, where  $F_{q+1}$ is given by
\begin{align}
    F_{q+1}:=&-F^{(1)}_{q+1}-F^{(2)}_{q+1}-\Div \RR^{rem}_q-\partial_t\wttqc. \label{def-F_{q+1}}
\end{align}
Then we define the perturbation $w_{q+1}$ by
\[w_{q+1}=\wpq+\wcq+\wttq+\wttqc+\wtq.\]
\subsubsection{Estimates for the perturbation} In this section, we aim to estimate $\wpq$. Firstly, we shall use the estimate of $F_{q+1}$ to obtain the bound of $\wtq$.
 \begin{proposition}[Estimates for $F_{q+1}$]\label{est-Fq}Let $F_{q+1}$ be defined in \eqref{def-F_{q+1}}. Then we have
 \begin{align*}
     \|F_{q+1}\|_{L^\infty_t H^4}\le \lambda^{5}_{q+1}, \qquad  \|F_{q+1}\big\|_{\widetilde L^{\infty}_tB^{-3/2}_{2,1} }\lesssim \lambda^{-40}_{q}.
 \end{align*}
   \end{proposition}
   \begin{proof}From the definition of $\wttqc$ in \eqref{wttqc} and \eqref{div-pt-phi}, one infers that
 \begin{align}
 \partial_t\wttqc=& -\frac{1}{\mu}\sum_{k\in\Lambda}\nabla{a}^2_{(k,q)}\times  (\partial_tG)(\sigma (x-x_k), \sigma t)-\frac{1}{\mu}\sum_{k\in\Lambda} \partial_t\nabla{a}^2_{(k,q)}\times \sigma^{-1}G(\sigma (x-x_k), \sigma t)\notag\\
 =&\sum_{k\in\Lambda}\nabla{a}^2_{(k,q)}\times\Big(\frac{\curl \Ph \Div}{\Delta}(\phi^2_{k,\bar{k},\bar{\bar{k}}}\bar{k}\otimes \bar{k})\Big)(\sigma (x-x_k), \sigma t)\notag\\
 &-\frac{1}{\mu}\sum_{k\in\Lambda} \partial_t\nabla{a}^2_{(k,q)}\times \sigma^{-1}G(\sigma (x-x_k), \sigma t).\label{pt-wtc}
\end{align}
Utilizing Proposition \ref{est-ak} and Proposition \ref{guji1}, we see
\begin{align*}
&\|\partial_t\wttqc\|_{L^\infty_t H^4}\\
\le& \Big\|\sum_{k\in\Lambda}\nabla{a}^2_{(k,q)}\times\Big(\frac{\curl \Ph \Div}{\Delta}(\phi^2_{k,\bar{k},\bar{\bar{k}}}\bar{k}\otimes \bar{k})\Big)(\sigma (x-x_k), \sigma t)\Big\|_{L^\infty_t H^4}\\
&+\Big\|\frac{1}{\mu}\sum_{k\in\Lambda} \partial_t\nabla{a}^2_{(k,q)}\times \sigma^{-1}G(\sigma (x-x_k), \sigma t)\Big\|_{L^\infty_t H^4}\\
\lesssim&\sum_{k\in\Lambda} (\|\nabla{a}^2_{(k,q)}\|_{L^\infty_t W^{4,\infty}}\|\phi^2_{k,\bar{k},\bar{\bar{k}}}\|_{L^\infty_t L^2(\TTT^3)}+\|\nabla{a}^2_{(k,q)}\|_{L^\infty_t L^{\infty}}\|\phi^2_{k,\bar{k},\bar{\bar{k}}}\|_{L^\infty_t H^4(\TTT^3)})\\
&+\sum_{k\in\Lambda}\frac{1}{\mu\sigma}(\| \partial_t\nabla{a}^2_{(k,q)}\|_{L^\infty_t L^\infty}\|\phi^2_{k,\bar{k},\bar{\bar{k}}}\|_{L^\infty_t H^3(\TTT^3)}+\| \partial_t\nabla{a}^2_{(k,q)}\|_{L^\infty_t W^{4,\infty}}\|\phi^2_{k,\bar{k},\bar{\bar{k}}}\|_{L^\infty_t L^2(\TTT^3)})\\
\lesssim& \ell^{-18}_q\lambda^{\frac{35}{8}}_{q+1}.
\end{align*}
By \eqref{e:R_rem}, one has
\begin{align*}
\|\Div\Rem\|_{L^\infty_t H^4}\lesssim\lambda^{20}_q\ell^{-2}_q.
\end{align*}
The above two inequalities together with \eqref{est-Fq1} and \eqref{est-Fq2}  yields that
\begin{align}
\|F_{q+1}\|_{L^\infty_t H^4}
&\lesssim\ell^{-4}_q\lambda^{4}_{q+1}+\ell^{-18}_q\lambda^{\frac{35}{8}}_{q+1}+\lambda^{20}_q\ell^{-2}_q \le\lambda^{5}_{q+1}.\notag
\end{align}

To estimate $\|F_{q+1}\|_{{\widetilde{L}^\infty_t B^{-\frac 32}_{2,1}}}$, we estimate $\|\partial_t\wttqc\|_{\widetilde{L}^\infty_t B^{-\frac 32}_{2,1}}$ firstly.  By $L^\infty_t L^{\frac 32}(\R^3)\hookrightarrow \widetilde{L}^\infty B^{-\frac 32}_{2,1}(\R^3)$, we infer from \eqref{pt-wtc} that
\begin{align}
&\Big\|\frac{1}{\mu}\sum_{k\in\Lambda} \partial_t\nabla{a}^2_{(k,q)}\times \sigma^{-1}G(\sigma (x-x_k), \sigma t)\Big\|_{\widetilde{L}^\infty_t B^{-\frac 32}_{2,1}}\notag\\
\lesssim& \sum_{k\in\Lambda}\mu^{-1}\sigma^{-1}\| \partial_t\nabla{a}^2_{(k,q)}\|_{L^\infty_{t,x}}\|\phi^2_{k,\bar{k},\bar{\bar{k}}}\|_{L^\infty_t L^{\frac 32}(\TTT^3)}\lesssim \ell^{-12}_q\lambda^{-\frac{9}{16}}_{q+1},\label{est-ptwttqc-1}
\end{align}
Moreover, owning to
\begin{align*}
&\big(\frac{\curl \Ph \Div}{\Delta}(\phi^2_{k,\bar{k},\bar{\bar{k}}}\bar{k}\otimes \bar{k})\big)(\sigma (x-x_k), \sigma t)\\
=&\sum_{j=1}^3\big(\frac{\partial_j\partial_j}{\Delta}\frac{\curl \Ph \Div}{\Delta}(\phi^2_{k,\bar{k},\bar{\bar{k}}}\bar{k}\otimes \bar{k})\big)(\sigma (x-x_k), \sigma t)\\
=&\sum_{j=1}^3\sigma^{-1}\partial_j\Big(\big(\frac{ \partial_j }{\Delta}\frac{\curl \Ph \Div}{\Delta}(\phi^2_{k,\bar{k},\bar{\bar{k}}}\bar{k}\otimes \bar{k})\big)(\sigma (x-x_k), \sigma t)\Big),
\end{align*}
 one has
\begin{align*}
&\sum_{k\in\Lambda}\nabla{a}^2_{(k,q)}\times\Big(\frac{\curl \Ph \Div}{\Delta}(\phi^2_{k,\bar{k},\bar{\bar{k}}}\bar{k}\otimes \bar{k})\Big)(\sigma (x-x_k), \sigma t)\\
=&\sum_{k\in\Lambda}\sum_{j=1}^3\sigma^{-1}\partial_j  \Big(\nabla{a}^2_{(k,q)}\times \big(\frac{\partial_j   }{\Delta}\frac{\curl \Ph \Div}{\Delta}(\phi^2_{k,\bar{k},\bar{\bar{k}}}\bar{k}\otimes \bar{k})\big)(\sigma (x-x_k), \sigma t)\Big)\\
&-\sum_{k\in\Lambda}\sum_{j=1}^3\sigma^{-1}\partial_j \nabla{a}^2_{(k,q)}\times \big(\frac{\partial_j   }{\Delta}\frac{\curl \Ph \Div}{\Delta}(\phi^2_{k,\bar{k},\bar{\bar{k}}}\bar{k}\otimes \bar{k})\big)(\sigma (x-x_k), \sigma t).
\end{align*}
Hence, by $L^\infty_t L^{\frac {3(1+\varepsilon)}{2-\varepsilon}}(\R^3)\hookrightarrow \widetilde{L}^\infty_t B^{-\frac 12}_{2,1}(\R^3)$ for any $0<\varepsilon<1$, we obtain that
\begin{align}
&\Big\|\sum_{k\in\Lambda}\sum_{j=1}^3\sigma^{-1}\partial_j  \Big(\nabla{a}^2_{(k,q)}\times \big(\frac{\partial_j   }{\Delta}\frac{\curl \Ph \Div}{\Delta}(\phi^2_{k,\bar{k},\bar{\bar{k}}}\bar{k}\otimes \bar{k})\big)(\sigma (x-x_k), \sigma t)\Big)\Big\|_{\widetilde{L}^\infty_t B^{-\frac 32}_{2,1}}\notag\\
\lesssim&\sum_{k\in\Lambda}\sum_{j=1}^3\sigma^{-1}\big\|  \nabla{a}^2_{(k,q)}\times \big(\frac{\partial_j   }{\Delta}\frac{\curl \Ph \Div}{\Delta}(\phi^2_{k,\bar{k},\bar{\bar{k}}}\bar{k}\otimes \bar{k})\big)(\sigma (x-x_k), \sigma t)\big\|_{L^\infty_t L^{\frac {3(1+\varepsilon)}{2-\varepsilon}}}\notag\\
\lesssim&\sum_{k\in\Lambda}\sum_{j=1}^3g\sigma^{-1}\big\|  \nabla{a}^2_{(k,q)}\|_{L^\infty_{t,x}}\Big\|\frac{\partial_j   }{\Delta}\frac{\curl \Ph \Div}{\Delta}(\phi^2_{k,\bar{k},\bar{\bar{k}}}\bar{k}\otimes \bar{k})\Big\|_{L^\infty_t L^{1+\varepsilon}(\TTT^3)}.\notag
\end{align}
Using $L^\infty_t W^{1, 1+\varepsilon}(\TTT^3)\hookrightarrow L^\infty_t L^{\frac {3(1+\varepsilon)}{2-\varepsilon}}(\TTT^3)$, we have
\begin{align}
&\Big\|\sum_{k\in\Lambda}\sum_{j=1}^3\sigma^{-1}\partial_j  \Big(\nabla{a}^2_{(k,q)}\times \big(\frac{\partial_j   }{\Delta}\frac{\curl \Ph \Div}{\Delta}(\phi^2_{k,\bar{k},\bar{\bar{k}}}\bar{k}\otimes \bar{k})\big)(\sigma (x-x_k), \sigma t)\Big)\Big\|_{\widetilde{L}^\infty_t B^{-\frac 32}_{2,1}}\notag\\
\lesssim&\sum_{k\in\Lambda}\sigma^{-1}\big\|  \nabla{a}^2_{(k,q)}\|_{L^\infty_{t,x}}\|\phi^2_{k,\bar{k},\bar{\bar{k}}}\|_{L^\infty_t L^{1+\varepsilon}(\TTT^3)} \lesssim \ell^{-10}_q\lambda^{\frac{9\varepsilon}{4(1+\varepsilon)}-\frac{1}{16}}_{q+1} \lesssim \ell^{-10}_q\lambda^{-\frac{1}{32}}_{q+1},\label{est-ptwttqc-2}
\end{align}
where the last inequality holds for choosing $\varepsilon=\frac{1}{71}$. Using $L^\infty_t L^{1+\varepsilon}(\R^3)\hookrightarrow \widetilde{L}^\infty_t B^{-\frac 12}_{2,1}(\R^3)$ for $\varepsilon=\frac{1}{71}$, one deduces that
\begin{align}
&\Big\|\sum_{k\in\Lambda}\sum_{j=1}^3\sigma^{-1}\partial_j \nabla{a}^2_{(k,q)}\times \big(\frac{\partial_j   }{\Delta}\frac{\curl \Ph \Div}{\Delta}(\phi^2_{k,\bar{k},\bar{\bar{k}}}\bar{k}\otimes \bar{k})\big)(\sigma (x-x_k), \sigma t)\Big\|_{\widetilde{L}^\infty_t B^{-\frac 32}_{2,1}}\notag\\
\lesssim&\Big\|\sum_{k\in\Lambda}\sum_{j=1}^3\sigma^{-1}\partial_j \nabla{a}^2_{(k,q)}\times \big(\frac{\partial_j   }{\Delta}\frac{\curl \Ph \Div}{\Delta}(\phi^2_{k,\bar{k},\bar{\bar{k}}}\bar{k}\otimes \bar{k})\big)(\sigma (x-x_k), \sigma t)\Big\|_{_{L^\infty_t L^{1+\varepsilon}} }\notag\\
\lesssim&\sum_{k\in\Lambda}\sum_{j=1}^3\sigma^{-1}\big\| \partial_j  \nabla{a}^2_{(k,q)}\|_{L^\infty_{t,x}}\|\phi^2_{k,\bar{k},\bar{\bar{k}}}\|_{L^\infty_t L^{1+\varepsilon}(\TTT^3)}  \lesssim \ell^{-12}_q\lambda^{\frac{9\varepsilon}{4(1+\varepsilon)}-\frac{1}{16}}_{q+1} \lesssim \ell^{-12}_q\lambda^{-\frac{1}{32}}_{q+1}.\label{est-ptwttqc-3}
\end{align}
Collecting \eqref{est-ptwttqc-1}--\eqref{est-ptwttqc-3} together implies that
\begin{align*}
\|\partial_t\wttqc\|_{\widetilde{L}^\infty_t B^{-\frac 32}_{2,1}}\lesssim\ell^{-12}_q\lambda^{-\frac{1}{32}}_{q+1}.
\end{align*}
Moreover, combining this inequality with \eqref{e:R_rem}, \eqref{est-Fq1} and \eqref{est-Fq2}, we conclude that
\begin{align*}
\|F_{q+1}\big\|_{\widetilde L^{\infty}_tB^{-3/2}_{2,1} }\le &\|F^{(1)}_{q+1}\big\|_{\widetilde L^{\infty}_tB^{-3/2}_{2,1} }+\|F^{(2)}_{q+1}\big\|_{\widetilde L^{\infty}_tB^{-3/2}_{2,1} }+\|\Rem\|_{ L^{\infty}_tL^2}+\|\partial_t\wttqc\|_{\widetilde{L}^\infty B^{-\frac 32}_{2,1}}\\
\lesssim &\ell^{-12}_q\lambda^{-\frac{3}{16}}_{q+1}+\ell^{-22}_q \lambda^{-\frac{1}{32}}_{q+1}+\lambda^{20}_q\ell_q+\ell^{-12}_q\lambda^{-\frac{1}{32}}_{q+1}
\lesssim \lambda^{-40}_{q}.
\end{align*}
Therefore we prove Proposition \ref{est-Fq}.
   \end{proof}
Based on the estimate of $F_{q+1}$, we are in position to estimate $\wtq$.
   \begin{proposition}[Estimates for $\wtq$]\label{wtq-H3}Let $\wtq$ be the solution of the equations \eqref{e:wt}, then we have
\begin{align}
&\|\wtq\|_{\widetilde L^{\infty}_tB^{1/2}_{2,1} }\le \lambda^{-20}_{q},\quad
\|\wtq\|_{L^{\infty}_t H^5}\lesssim\lambda^{5}_{q+1} . \label{estimate-wt}
\end{align}
   \end{proposition}
   \begin{proof}
       With the aid of Duhamel formula, we write the equations \eqref{e:wt} in the integral form
\begin{align}
\wtq(x,t)=&\int_0^t e^{(t-s)\Delta}\mathcal{P}\Div (\wtq\otimes \wtq+\uqnl\ootimes \wtq+ \wtq\otimes\uqnl)(s) \dd s\nonumber\\
&-\int_0^t e^{(t-s)\Delta}\mathcal{P} F_{q+1}(s)\dd s,\label{inte-wtq}
\end{align}
where $\mathcal{P}$ is the Leray projector onto divergence-free victor fields.

Thanks to \eqref{e:vq-C0}, we have
\begin{align*}
&\Big\|\int_0^t e^{(t-s)\Delta}\mathcal{P}\Div ( \uqnl\ootimes \wtq+ \wtq\ootimes\uqnl )\dd s\Big\|_{\widetilde L^{\infty}_tB^{1/2}_{2,1}}\\
\lesssim&(\|\uqnl\ootimes \wtq\|_{\widetilde L^{\infty}_tB^{-1/2}_{2,1}}+\|\wtq\ootimes \uqnl\|_{\widetilde L^{\infty}_tB^{-1/2}_{2,1}})\\
\lesssim&\|\uqnl\|_{\widetilde{L}^{\infty}_tB^{1/2}_{2,1}}\|\wtq\|_{\widetilde L^{\infty}_tB^{1/2}_{2,1}}\\
\le&C\lambda^{-1}_0\|\wtq\|_{\widetilde L^{\infty}_tB^{1/2}_{2,1} }.
\end{align*}
By Lemma \ref{heat} and Proposition \ref{est-Fq}, one gets
\begin{align*}
&\Big\|\int_0^t e^{(t-s)\Delta}\mathcal{P} F_{q+1}(s)\dd s\Big\|_{\widetilde L^{\infty}_tB^{1/2}_{2,1} }
\lesssim\big\|F_{q+1}\big\|_{\widetilde L^{\infty}_tB^{-3/2}_{2,1} }\lesssim\lambda^{-40}_q.
\end{align*}
Taking $\lambda_0$ large enough such that $C\lambda^{-1}_0\le \frac{1}{2}$, for large enough $a$, we collect the above two estimates together to obtain
\begin{align*}
\|\wtq\|_{\widetilde L^{\infty}_tB^{1/2}_{2,1}}\lesssim\|\wtq\|^2_{ \widetilde L^{\infty}_tB^{1/2}_{2,1}}+\lambda^{-40}_{q}.
\end{align*}
By  the Banach fixed point theorem and the the continuity method, this estimate implies that as long as $a$ is large enough, the equation \eqref{inte-wtq} admits a unique  mild solution $\wtq$ on $[0,T]$ with
\begin{align}
 \|\wtq\|_{ \widetilde L^{\infty}_tB^{1/2}_{2,1}}\lesssim \lambda^{-40}_{q}\le \lambda^{-20}_{q}.\label{es：wtq}
\end{align}
Moreover, combining Proposition \ref{est-Fq} with \eqref{es：wtq}, note that $L^\infty_t H^4(\R^3)\hookrightarrow \widetilde L^{\infty}_t B^3_{2,2}(\R^3)$, we have
\begin{align*}
 \|\wtq\|_{ \widetilde L^{\infty}_t B^5_{2,2} }
\lesssim&\|\wtq\|_{\widetilde L^{\infty}_tB^{-1}_{\infty,\infty}}\|\wtq\|_{\widetilde L^{\infty}_t B^5_{2,2}}+\|\uqnl\|_{L^\infty_t L^\infty}\|\wtq\|_{\widetilde L^{\infty}_t B^4_{2,2}}\\
 &+\|\uqnl\|_{\widetilde L^{\infty}_t B^4_{2,2}}\|\wtq\|_{L^\infty_t L^\infty}+\|F_{q+1}\|_{\widetilde L^{\infty}_t B^3_{2,2}} \\
 \lesssim&\|\wtq\|_{\widetilde L^{\infty}_tB^{1/2}_{2,1}}\|\wtq\|_{\widetilde L^{\infty}_t B^5_{2,2}}+\|\uqnl\|_{L^\infty_t L^\infty}\|\wtq\|^{7/9}_{\widetilde L^{\infty}_t B^5_{2,2}}\|\wtq\|^{2/9}_{\widetilde L^{\infty}_t B^{1/2}_{2,1}}\\
& +\|\uqnl\|_{\widetilde L^{\infty}_t B^4_{2,2}}\|\wtq\|^{2/9}_{\widetilde L^{\infty}_t B^5_{2,2}}\|\wtq\|^{7/9}_{\widetilde L^{\infty}_t B^{1/2}_{2,1}}+\|F_{q+1}\|_{ L^{\infty}_t H^4}\\
\le&C(\lambda^{-20}_{q}\|\wtq\|_{\widetilde L^{\infty}_t B^5_{2,2}}+\lambda^{25}_q+\lambda^{5}_{q+1})+\frac{1}{2}\|\wtq\|_{ \widetilde L^{\infty}_t B^5_{2,2} }.
\end{align*}
This inequality shows that
\begin{align*}
 \|\wtq\|_{ L^{\infty}_t H^5 } \le \|\wtq\|_{\widetilde L^{\infty}_t B^5_{2,2} }\lesssim\lambda^{5}_{q+1}.
\end{align*}
Hence we complete the proof of Proposition \ref{wtq-H3}.
   \end{proof}
Next, we  estimate the remaining four parts of $w^{{\rm(loc)}}_{q+1}$, namely $\wpq$, $\wcq$, $\wttq$ and $\wttqc$. Then combining the above estimate of $\wtq$, we immediately  obtain the estimates for $w_{q+1}$.

   \begin{proposition}[Estimates for $w_{q+1}$]\label{estimate-wq+1}There exists a  constant {$C_0$} such that
\begin{align}
&{\lambda_{q+1}}\|\wpq\|_{L^{\infty}_tL^1}+\|\wpq\|_{ L^{\infty}_tL^2}+{\lambda^{-6}_{q+1}}\|\wpq\|_{L^{\infty}_tH^5}\le \frac{1}{4}C_0\delta^{1/2}_{q+1},\label{estimate-wp}\\
&{\lambda_{q+1}}\|\wcq\|_{L^{\infty}_tL^1}+\|\wcq\|_{ L^{\infty}_tL^2}+{\lambda^{-5}_{q+1}}\|\wcq\|_{L^{\infty}_tH^5}
\le \lambda^{-\frac {1}{8}}_{q+1},\label{estimate-wc}\\
&\|\wttq\|_{ L^{\infty}_tL^1}+\|\wttq\|_{ L^{\infty}_tL^2}+{\lambda^{-3}_{q+1}}\|\wttq\|_{L^{\infty}_tH^5}
\le \lambda^{-\frac {1}{16}}_{q+1},\label{estimate-wttq}\\
&\|\wttqc\|_{ L^{\infty}_tL^1}+\|\wttqc\|_{ L^{\infty}_tL^2}+{\lambda^{-3}_{q+1}}\|\wttqc\|_{L^{\infty}_tH^5}
\le \lambda^{-\frac {1}{16}}_{q+1},\label{estimate-wttqc}\\
&\|w_{q+1}\|_{ L^{\infty}_tL^2}+{\lambda^{-6}_{q+1}}\|w_{q+1}\|_{L^{\infty}_tH^5}\le \frac{1}{2}C_0\delta^{1/2}_{q+1}.\label{estimate-w}
\end{align}
   \end{proposition}
\begin{proof}By the definition of $\wpq$ in \eqref{def-wpq-1} and Lemma \ref{Holder}, we have
  \begin{align*}
\|\wpq\|_{L^\infty_t L^2}
\lesssim& \sum_{k\in\Lambda}\|a_{(k,q)}\|_{L^\infty_t L^2}\|\phi_{k,\bar{k},\bar{\bar{k}},\sigma}  \|_{L^2(\TTT^3)}+\sum_{k\in\Lambda}\lambda^{-\frac{1}{16}}_{q+1}\|a_{(k,q)}\|_{L^\infty_t W^{1,\infty}}\|\phi_{k,\bar{k},\bar{\bar{k}},\sigma}\|_{L^2_t(\TTT^3)}.
\end{align*}
From the definition of $a_{(k,q)}$, \eqref{est-rhoq} and \eqref{akq-HN}, one deduces that
 \begin{align*}
\|a_{(k,q)}\|_{L^\infty_t L^2}\lesssim\delta^{1/2}_{q+1}.
 \end{align*}
This together with Proposition \ref{est-ak} shows that  there exists a universal constant $C_0$ such that
 \begin{align*}
\|\wpq\|_{L^\infty_t L^2}\le& C\delta^{1/2}_{q+1}+C\ell^{-6}_q\lambda^{-\frac{1}{16}}_{q+1}
\le \frac{1}{8}C_0\delta^{1/2}_{q+1}.
 \end{align*}
We infer from Proposition \ref{est-ak} and Proposition \ref{guji1} that
\begin{align*}
\|\wpq\|_{L^\infty_t L^1}\le& \sum_{k\in\Lambda}\|a_{(k,q)}\|_{L^\infty_{t,x}}\|\phi_{k,\bar{k},\bar{\bar{k}},\sigma}\psi_k\|_{L^\infty L^1(\TTT^3)}
\lesssim\ell^{-4}_q\lambda^{-\frac 98}_{q+1},
\end{align*}
and
\begin{align*}
\|\wpq\|_{L^\infty_tH^5}
\le& \sum_{k\in\Lambda}\|a_{(k,q)}\|_{L^\infty_tH^5}\|\phi_{k,\bar{k},\bar{\bar{k}},\sigma}\psi_k\|_{H^5(\TTT^3)}
\lesssim\ell^{-10}_q\lambda^{5}_{q+1}.
\end{align*}
Collecting the three estimates together shows \eqref{estimate-wp}.

Applying Lemma \ref{Holder} and Proposition \ref{guji1} to $\wcq$ in \eqref{def-wcq}, we have
\begin{align*}
\|\wcq\|_{L^\infty_t L^2}
\lesssim &\sum_{k\in\Lambda}\lambda^{-1}_{q+1} \|a_{(k,q)}\phi_{k,\bar{k},\bar{\bar{k}},\sigma} \|_{L^\infty_tH^1}\|\Psi'_k \|_{L^2_t(\TTT^3)}\\
&+\sum_{k\in\Lambda}\lambda^{-2}_{q+1}\|a_{(k,q)}\phi_{k,\bar{k},\bar{\bar{k}},\sigma} \|_{L^\infty_tH^2}\|\Psi'_k\|_{L^2(\TTT^3)}\\
\lesssim&\ell^{-4}_q\lambda^{-\frac{3}{16}}_{q+1} .
\end{align*}
By  Proposition \ref{est-ak} and Proposition \ref{guji1}, we obtain
\begin{align*}
\|\wcq\|_{L^\infty_t L^1}\le &\sum_{k\in\Lambda}\lambda^{-1}_{q+1}\|a_{(k,q)}\|_{L^\infty_t W^{1,\infty}}\|\phi_{k,\bar{k},\bar{\bar{k}},\sigma} \|_{L^\infty_{t}W^{1,1}(\TTT^3)}
\lesssim\ell^{-6}_q\lambda^{-\frac{21}{16}}_{q+1},
\end{align*}
and
\begin{align*}
\|\wcq\|_{L^\infty_tH^5}\le &\sum_{k\in\Lambda}\lambda^{-1}_{q+1}\|a_{(k,q)}\|_{L^\infty W^{6,\infty}}(\|\phi_{k,\bar{k},\bar{\bar{k}},\sigma} \|_{L^\infty_{t}H^6(\TTT^3)}+\|\phi_{k,\bar{k},\bar{\bar{k}},\sigma} \|_{L^\infty_{t}H^1(\TTT^3)}\|\Psi'_k\|_{W^{5,\infty}})\\
\lesssim&\ell^{-16}_q\lambda^{\frac{13}{16}+4}_{q+1}.
\end{align*}
Hence we conclude \eqref{estimate-wc}. With the aid of  Proposition \ref{est-ak} and Proposition \ref{guji1}, one deduces from the definitions of $\wttq$ and $\wttqc$ in \eqref{wttq} and \eqref{wttqc} that
\begin{align*}
&\|\wttq\|_{L^\infty_t L^2}+\|\wttqc\|_{L^\infty_t L^2}\\
\lesssim&\sum_{k\in\Lambda} \mu^{-1}\|a^2_{(k,q)}\|_{L^\infty_{t,x}}\|\phi^2_{k,\bar{k},\bar{\bar{k}},\sigma} \|_{L^2(\TTT^3)}+\sum_{k\in\Lambda}\mu^{-1}\sigma^{-1}\|\nabla a^2_{(k,q)}\|_{L^\infty_{t,x}} \|\phi^2_{k,\bar{k},\bar{\bar{k}}} \|_{L^2(\TTT^3)}\\
\lesssim&\ell^{-8}_q\lambda^{-\frac 18}_{q+1},
\end{align*}
\begin{align*}
\|\wttq\|_{L^\infty_t L^1}+\|\wttqc\|_{L^\infty_t L^1}
\lesssim\sum_{k\in\Lambda} \mu^{-1}\| a^2_{(k,q)}\|_{L^\infty_t W^{1,\infty}}\|\phi^2_{k,\bar{k},\bar{\bar{k}},\sigma} \|_{L^1(\TTT^3)}\lesssim\ell^{-10}_q\lambda^{-\frac 54}_{q+1},
\end{align*}
and
\begin{align*}
&\|\wttq\|_{L^\infty_t H^5}+\|\wttqc\|_{L^\infty_t H^5}\\
\lesssim& \sum_{k\in\Lambda}\mu^{-1}\|a^2_{(k,q)}\|_{L^\infty_t H^5}\|\phi^2_{k,\bar{k},\bar{\bar{k}},\sigma} \|_{H^5(\TTT^3)}+\sum_{k\in\Lambda}\mu^{-1}\sigma^4\|\nabla a^2_{(k,q)}\|_{L^\infty_t H^5}\|\phi^2_{k,\bar{k},\bar{\bar{k}}} \|_{L^2(\TTT^3)}\\
\lesssim&\ell^{-14}_q\lambda^{\frac{43}{16}}_{q+1}.
\end{align*}
The two inequalities show \eqref{estimate-wttq} and \eqref{estimate-wttqc}. Then collecting \eqref{estimate-wp}--\eqref{estimate-wttqc} together yields \eqref{estimate-w}.  Thereby we finish the proof of Proposition \ref{estimate-wq+1}.
\end{proof}

\subsubsection{Estimates for the Reynolds stress $\RR_{q+1}$} In this section, we establish the estimates in order to verify that \eqref{e:RR_q-C0}--\eqref{e:RR_q-tigh} hold at $q+1$ level. Letting
\begin{align}
&\wqloc:={\wpq+\wcq+\wttq+\wttqc},\quad  w_{q+1}=\wqloc+\wtq,\label{def-wloc}\\
&u_{q+1}=u_{\ell_q}+w_{q+1}=(\ulql+\wqloc)+(\ulqnl+\wtq)=:u^{\textup{loc}}_{q+1}+u^{\textup{non-loc}}_{q+1}, \label{def-uloc}\\
&{p_{q+1}=p_{\ell_q}}+\nabla p^{\textup{(ns)}}_{q+1}-P^{(1)}_{q+1}-P^{(2)}_{q+1}-\frac{2}{3}u_{\ell_q}\cdot w_{q+1}-\frac{1}{3}w_{q+1}\cdot w_{q+1},\label{def-pq+1}
\end{align}
we have by \eqref{e:mollified-euler}
\begin{align*}
&\del_tu_{q+1}-\Delta u_{q+1} +\Div(u_{q+1}\otimes u_{q+1})+\nabla p_{q+1}\\
=&\partial_t w_{q+1}-\Delta w_{q+1}+\Div(w_{q+1}\ootimes u_{\ell_q})+\Div(u_{\ell_q}\ootimes w_{q+1})+\Div( w_{q+1}\ootimes w_{q+1})\\
&+\nabla p^{\textup{(ns)}}_{q+1}-P^{(1)}_{q+1}-P^{(2)}_{q+1}+\Div \RR_{\ell_q} +\Div \Rem\\
=& \partial_t \wtq-\Delta \wtq+\Div(\wtq\ootimes u_{\ell_q})+\Div( u_{\ell_q}\ootimes \wtq)+\Div( \wtq\ootimes \wtq)+\nabla p^{\textup{(ns)}}_{q+1}\\
&+\Div(\wpq\ootimes \wpq)+\Div \RR_{\ell_q} -\nabla P^{(1)}_{q+1}+\partial_t \wqloc -\nabla P^{(2)}_{q+1}+\Div \Rem \\
&+\Div( w_{q+1}\ootimes w_{q+1})-\Div( \wpq\ootimes \wpq)-\Div( \wtq\ootimes \wtq)-\Delta\wqloc\\
&+\Div(\wqloc\ootimes u_{\ell_q})+\Div(u_{\ell_q}\ootimes\wqloc).
\end{align*}
With the aid of Proposition \ref{def-F2}--\ref{F1}, since $\wtq$ satisfies the equations \eqref{e:wt}, we conclude from the above equality that
\begin{align}
&\del_tu_{q+1}-\Delta u_{q+1} +\Div(u_{q+1}\otimes u_{q+1})+\nabla p_{q+1}\nonumber\\
=&\,\Div\big(  w_{q+1}\ootimes w_{q+1}- \wpq\ootimes \wpq- \wtq\ootimes \wtq\nonumber\\
&+\wqloc\ootimes u_{\ell_q}+u_{\ell_q}\ootimes\wqloc+\RR^{(1)}_{q+1}+\RR^{(2)}_{q+1}-(\nabla \wqloc +(\nabla \wqloc)^{\TT})\big)\nonumber\\
=:&\Div \RR_{q+1}.\label{def-R-q+1}
\end{align}
Now we are focused on estimating $\RR_{q+1}$. According to the definition of $\RR_{q+1}$ in \eqref{def-R-q+1}, $\spt_x\RR_{q+1}$ is determined by the supports of  $\RR^{(1)}_{q+1}, \RR^{(2)}_{q+1}$ and $\wqloc$. From  $\RR^{(1)}_{q+1} $ and $\RR^{(2)}_{q+1}$ in  Proposition \ref{def-F2} and Proposition \ref{F1}, and $\spt_x\wqloc \subseteq\Omega_{q+1}$, we infer that
\begin{align}\label{suppx-R}
    \spt_x \RR_{q+1} \subseteq\Omega_{q+1}.
\end{align}
\begin{proposition}[Estimates for $\RR_{q+1}$]\label{R-q+1}Let $\RR_{q+1}$ be defined in \eqref{def-R-q+1}, it holds that
\begin{align*}
    & \|\RR_{q+1}\|_{L^\infty_t L^1}\le \delta_{q+2}\lambda^{-4\alpha}_{q+1},\qquad\|\RR_{q+1}\|_{L^\infty_t W^{5,1}}\le \lambda^7_{q+1}.
\end{align*}
\end{proposition}
\begin{proof}
With the help of \eqref{wpq+wcq} and Proposition \ref{est-ak}--\ref{guji1}, we see
\begin{align}
\|\nabla(\wpq+\wcq)\|_{L^{\infty}_t L^1}
\lesssim&\sum_{k\in\Lambda} \lambda^{-1}_{q+1}\|a_{(k,q)} \phi_{k,\bar{k},\bar{\bar{k}},\sigma}\Psi'_k\|_{L^\infty W^{2,1}}\notag\\
\lesssim&\lambda^{-1}_{q+1}\ell^{-8}_q(\|\phi_{k,\bar{k},\bar{\bar{k}},\sigma}\|_{L^\infty_t W^{2,1}(\TTT^3)}+\|\phi_{k,\bar{k},\bar{\bar{k}},\sigma}\|_{L^\infty_t L^1(\TTT^3)}\|\Psi'_k\|_{L^\infty_t W^{2,\infty}})\notag\\
\lesssim&\ell^{-8}_q\lambda^{-\frac{1}{8}}_{q+1},\notag
\end{align}
and
\begin{align}
\|\nabla(\wpq+\wcq)\|_{L^{\infty}_t W^{5,1}}
\lesssim&\sum_{k\in\Lambda} \lambda^{-1}_{q+1}\|a_{(k,q)} \phi_{k,\bar{k},\bar{\bar{k}},\sigma}\Psi'_k\|_{L^\infty_t W^{7,1}}\notag\\
\lesssim&\lambda^{-1}_{q+1}\ell^{-18}_q(\|\phi_{k,\bar{k},\bar{\bar{k}},\sigma}\|_{L^\infty_t W^{7,1}(\TTT^3)}+\|\phi_{k,\bar{k},\bar{\bar{k}},\sigma}\|_{L^\infty_t L^1(\TTT^3)}\|\Psi'_k\|_{L^\infty_t W^{7,\infty}})\notag\\
\lesssim&\lambda^{5}_{q+1}.\notag
\end{align}
Using  Proposition \ref{est-ak}--\ref{guji1}, we deduce from \eqref{wttq+wttqc} that
\begin{align*}
\|\nabla(\wttq+\wttqc)\|_{L^{\infty}_t L^1}
\lesssim& \frac{1}{\mu\sigma }\sum_{k\in\Lambda}\|a_{(k,q)}G(\sigma (x-x_k), \sigma t)\|_{L^\infty_t W^{1,1}}\\
\lesssim& \frac{1}{\mu}\sum_{k\in\Lambda}\|a_{(k,q)}\|_{L^\infty_t W^{1,\infty}}\|\phi^2_{k,\bar{k},\bar{\bar{k}}}\|_{L^\infty_t L^1(\TTT^3)}\\
\lesssim&\ell^{-6}_q\lambda^{-\frac 54}_{q+1},
\end{align*}
and
\begin{align*}
\|\nabla(\wttq+\wttqc)\|_{L^{\infty}_t W^{5,1}}
\lesssim& \frac{1}{\mu\sigma }\sum_{k\in\Lambda}\|a_{(k,q)}G(\sigma (x-x_k), \sigma t)\|_{L^\infty_t W^{6,1}}\\
\lesssim& \frac{\sigma^5}{\mu}\sum_{k\in\Lambda}\|a_{(k,q)}\|_{L^\infty_t W^{6,\infty}}\|\phi^2_{k,\bar{k},\bar{\bar{k}}}\|_{L^\infty_t L^1(\TTT^3)}\\
\lesssim&\ell^{-16}_q\lambda^{\frac{25}{16}}_{q+1}.
\end{align*}
The four estimates together show that
\begin{align}\label{nabla w-LW}
&\|\nabla \wqloc +(\nabla \wqloc)^{\TT}\|_{L^\infty_t L^1}\lesssim \ell^{-8}_q\lambda^{-\frac{1}{8}}_{q+1},\quad\|\nabla \wqloc +(\nabla \wqloc)^{\TT}\|_{L^\infty_t W^{5,1}}\lesssim\lambda^{5}_{q+1}.
\end{align}
Making use of \eqref{e:v_ell-CN+1}, \eqref{estimate-wp}--\eqref{estimate-wttqc}, we have
\begin{align}
&\|\wqloc\ootimes u_{\ell_q}+u_{\ell_q}\ootimes\wqloc\|_{L^\infty_t L^1}
\lesssim\|u_{\ell_q}\|_{L^\infty_{t,x}}\|\wqloc\|_{L^\infty_t L^1}\lesssim\lambda^{10}_q\lambda^{-\frac{1}{16}}_{q+1},
\end{align}
and
\begin{align}
&\|\wqloc\ootimes u_{\ell_q}+u_{\ell_q}\ootimes\wqloc\|_{L^\infty_t W^{5,1}}
\lesssim\|u_{\ell_q}\|_{L^\infty_t H^5 }\|\wqloc\|_{L^\infty_t H^5}\lesssim\lambda^{10}_q\lambda^6_{q+1}.
\end{align}
By virtue of \eqref{estimate-wt}, \eqref{estimate-wp}--\eqref{estimate-w}, we easily deduce that
\begin{align}
&\| w_{q+1}\ootimes w_{q+1}- \wpq\ootimes \wpq- \wtq\ootimes \wtq\|_{L^\infty_t L^1}\notag\\
=&\|w_{q+1}\ootimes \wqloc+\wqloc\ootimes \wtq- \wpq\ootimes \wpq\|_{L^\infty_t L^1}\notag\\
\lesssim&\|\wqloc\|_{L^\infty_t L^2}(\|w_{q+1}-\wpq\|_{L^\infty_t L^2}+\|\wtq\|_{L^\infty_t L^2})\notag\\
&+\|\wpq\|_{L^\infty_t L^2}\|\wqloc-\wpq\|_{L^\infty_t L^2}\notag\\
\lesssim& C_0\delta^{1/2}_{q+1}\lambda^{-20}_q,\label{wotimesw-L}
\end{align}
and
\begin{align}
&\| w_{q+1}\ootimes w_{q+1}- \wpq\ootimes \wpq- \wtq\ootimes \wtq\|_{L^\infty_t W^{5,1}}\notag\\
\lesssim&\|\wpq\|_{L^\infty L^2}\|w_{q+1}-\wpq\|_{L^\infty_t H^5}+\|\wpq\|_{L^\infty_t H^5}\|w_{q+1}-\wpq\|_{L^\infty_t L^2}\notag\\
&+\|\wcq+\wttq+\wttqc\|_{L^\infty_t L^2} \|w_{q+1}\|_{L^\infty_t H^5}+\|\wcq+\wttq+\wttqc\|_{L^\infty_t H^5} \|w_{q+1}\|_{L^\infty_t L^2}\notag\\
&+\|\wtq\|_{L^\infty_t L^2}\|w_{q+1}-\wtq\|_{L^\infty_t H^5}+\|\wtq\|_{L^\infty_t H^5}\|w_{q+1}-\wtq\|_{L^\infty_t L^2}\notag\\
\lesssim& C_0\delta_{q+1}\lambda^{6}_{q+1}.\label{wotimesw-W}
\end{align}
Collecting \eqref{est-Rq1}, \eqref{est-Rq2} and \eqref{nabla w-LW}--\eqref{wotimesw-W} together shows that
\begin{align*}
\|\RR_{q+1}\|_{L^\infty_t L^1}\lesssim\ell^{-12}_q\lambda^{-\frac{1}{16}}_{q+1}+ C_0\delta^{1/2}_{q+1}\lambda^{-20}_q\le \delta_{q+2}\lambda^{-4\alpha}_{q+1},
\end{align*}
and
\begin{align*}
\|\RR_{q+1}\|_{L^\infty_t W^{5,1}}\lesssim \ell^{-8}_q\lambda^{5}_{q+1}+\lambda^{10}_q\lambda^6_{q+1}\le \lambda^7_{q+1}.
\end{align*}
Hence, we complete the proof of Proposition \ref{R-q+1}.
\end{proof}
\subsubsection{Estimates for the energy gap} We need to verify that the energy gap satisfies \eqref{et} at $q+1$ level. In fact, we have the following estimate. \begin{proposition}[The estimate for energy gap]\label{est-E}Let $u_{q+1}=u_{\ell_q}+w_{q+1}$, we have
    \begin{align*}
\Big|e(t)-\int_{\R^3}|u_{q+1}|^2\dd x-2\delta_{q+2}\Big|\le \delta_{q+2}\lambda^{-\alpha}_q.
    \end{align*}
\end{proposition}
\begin{proof}We write
\begin{align*}
\int_{\R^3}|u_{q+1}|^2\dd x&=\int_{\R^3}|u_{\ell_q}|^2\dd x+\int_{\R^3}|w_{q+1}|^2\dd x+2\int_{\R^3}w_{q+1}\cdot u_{\ell_q}\dd x,
\end{align*}
then
\begin{align*}
&e(t)-\int_{\R^3}|u_{q+1}|^2\dd x-2\delta_{q+2}\\
=&\Big(e(t)-\int_{\R^3}|u_{\ell_q}|^2\dd x-2\delta_{q+2}-\int_{\R^3}|\wpq|^2\dd x\Big)\\
&+\Big(\int_{\R^3}|\wpq|^2\dd x-\int_{\R^3}|w_{q+1}|^2\dd x-2\int_{\R^3}w_{q+1}\cdot u_{\ell_q}\dd x\Big)\\
=:& {\rm I+ II}.
\end{align*}
Thanks to \eqref{e:v_ell-CN+1}, \eqref{estimate-wt} and \eqref{estimate-wp}--\eqref{estimate-wttqc}, one deduces that
\begin{align*}
\Big |\int_{\R^3}u_{\ell_q}\cdot w_{q+1}\dd x\Big|
\le& \|u_{\ell_q}\|_{L^\infty_t L^2}\|\wtq\|_{L^\infty_t L^2}\\
&+\|u_{\ell_q}\|_{L^\infty_{t,x}}(\|\wpq\|_{L^\infty_t L^1}+\|\wcq\|_{L^\infty_t L^1}+\|\wttq\|_{L^\infty_t L^1}+\|\wttqc\|_{L^\infty_t L^1})\\
\lesssim &\lambda^{-10}_q+\lambda^{10}_q\lambda^{-\frac{1}{16}}_{q+1} \lesssim \lambda^{-10}_q.
\end{align*}
By $w_{q+1}=\wpq+\wcq+\wtq+\wttq+\wttqc$, we rewrite
\begin{align*}
&\int_{\R^3}|\wpq|^2\dd x-\int_{\R^3}|w_{q+1}|^2\dd x\\
=&-\int_{\R^3}(\wcq+\wtq+\wttq+\wttqc)\cdot (\wpq+w_{q+1})\dd x.
\end{align*}
With the aid of  \eqref{estimate-wt} and \eqref{estimate-wp}--\eqref{estimate-wttqc}, one has
\begin{align*}
&\Big|\int_{\R^3}|\wpq|^2\dd x-\int_{\R^3}|w_{q+1}|^2\dd x\Big|\\
\le&(\|w_{q+1}\|_{L^\infty_t L^2}+\|\wpq\|_{L^\infty_t L^2})(\|\wcq\|_{L^\infty_t L^2}+\|\wtq\|_{L^\infty_t L^2}+\|\wttq\|_{L^\infty_t L^2}+\|\wttqc\|_{L^\infty_t L^2})\\
\lesssim&C_0\delta^{1/2}_{q+1}\lambda^{-20}_q.
\end{align*}
Hence, we conclude that
\begin{align}\label{est-II}
|{\rm II}| \lesssim&C_0\lambda^{-10}_q.
\end{align}
Noting that $\spt_x \wpq\subseteq [-\tfrac{1}{2}, \tfrac{1}{2}]^3=:D$, by the definition of $\wpq$, we deduce that
\begin{align}
&\int_{\R^3}|\wpq|^2\dd x=\int_{ D}\tr (\wpq\otimes\wpq)\dd x\nonumber\\
=&\int_{D}\tr \Big(\sum_{k\in\Lambda}
 a^2_{(k,q)} \phi^2_{k,\bar{k},\bar{\bar{k}},\sigma}\psi^2_k\bar{k}\otimes \bar{k}\Big)\dd x\nonumber\\
 =&
 \int_{D}\tr \Big(\sum_{k\in\Lambda}
 a^2_{(k,q)} \bar{k}\otimes \bar{k}\Big)\dd x+\int_{D}\tr \Big(\sum_{k\in\Lambda}
 a^2_{(k,q)} \mathbb{P}_{\neq 0}(\phi^2_{k,\bar{k},\bar{\bar{k}},\sigma})\bar{k}\otimes \bar{k}\Big)\dd x \nonumber\\
 &+\int_{D}\tr \Big(\sum_{k\in\Lambda}
 a^2_{(k,q)} \phi^2_{k,\bar{k},\bar{\bar{k}},\sigma}\mathbb{P}_{\neq 0}(\psi^2_k)\bar{k}\otimes \bar{k}\Big)\dd x\label{I}
\end{align}
By virtue of Lemma \ref{first S}, the first term on the right-hand side of \eqref{I} can be rewritten as
\begin{align}
\int_{D}\tr \Big(\sum_{k\in\Lambda}a^2_{(k,q)}  \bar{k}\otimes \bar{k}\Big)\dd x
 =3{\rho}_q\int_{D}\eta^2_q\chi_q\dd x
 =e(t)-\int_{\R^3}|u_{\ell_q}|^2\dd x-2\delta_{q+2},\label{Tr-kotimesk}
\end{align}
where we have used \eqref{energy2}. Via Fourier coefficient, we rewrite $\mathbb{P}_{\neq 0}(\phi^2_{k,\bar{k},\bar{\bar{k}},\sigma})$
as
\[\mathbb{P}_{\neq 0}(\phi^2_{k,\bar{k},\bar{\bar{k}},\sigma})=\sum_{|l+m|\neq 0}b_{m,r}b_{l,r}e^{\ii  {2\pi}(m+l)\sigma k\cdot x}.\]
Using integration by parts, we have
\begin{align*}
&\int_{D}\tr \Big(\sum_{k\in\Lambda}
 a^2_{(k,q)} \mathbb{P}_{\neq 0}(\phi^2_{k,\bar{k},\bar{\bar{k}},\sigma})\bar{k}\otimes \bar{k}\Big)\dd x \\
 =&\tr \Big(\sum_{k\in\Lambda} \sum_{|l+m|\neq 0} \frac{-\ii b_{m,r}b_{l,r}}{2\pi(l+m)\sigma}\int_{D}
 a^2_{(k,q)}\Div \big(e^{\ii  {2\pi}(m+l)\sigma k\cdot x} k\big)\dd x \bar{k}\otimes \bar{k}\Big)\\
 =&\tr \Big(\sum_{k\in\Lambda} \sum_{|l+m|\neq 0} \frac{-\ii b_{m,r}b_{l,r}}{2\pi(l+m)\sigma}\int_{D}
 \nabla a^2_{(k,q)}\cdot k e^{\ii  {2\pi}(m+l)\sigma k\cdot x} \dd x \bar{k}\otimes \bar{k}\Big).
\end{align*}
Hence, via  integration by parts for $L$ times with $L=60$, one infers that
\begin{align}
&\Big|\int_{D}\tr \Big(\sum_{k\in\Lambda}
 a^2_{(k,q)} \mathbb{P}_{\neq 0}(\phi^2_{k,\bar{k},\bar{\bar{k}},\sigma})\bar{k}\otimes \bar{k}\Big)\dd x\Big|\notag\\
 \lesssim& \sum_{|l+m|\neq 0} \frac{ |b_{m,r}||b_{l,r}|}{|l+m|^L\sigma^{L}} \|a^{2}_{k,q}\|_{L^\infty_t W^{L,1}}\lesssim \ell^{-200}_q\lambda^{-1}_{q+1}.\label{Tr-1}
\end{align}
Following the method on deriving the above inequality, we immediately show that
\begin{align}
&\Big|\int_{D}\tr \Big(\sum_{k\in\Lambda}
 a^2_{(k,q)} \phi^2_{k,\bar{k},\bar{\bar{k}},\sigma}\mathbb{P}_{\neq 0}(\psi^2_k)\bar{k}\otimes \bar{k}\Big)\dd x\Big|\notag\\
 \lesssim& \sum_{|l+m|\neq 0} \frac{ |c_{m}||c_{l}|}{|l+m|^L\lambda^{L}_{q+1}} \|a^{2}_{k,q}\phi^2_{k,\bar{k},\bar{\bar{k}},\sigma}\|_{L^\infty_t W^{L,1}}\lesssim \ell^{-8}_q\lambda^{-10}_{q+1}.\label{Tr-2}
\end{align}
Plugging \eqref{Tr-kotimesk}--\eqref{Tr-2} into \eqref{I} yields that
\begin{align*}
|{\rm I}|\lesssim \ell^{-200}_q\lambda^{-1}_{q+1}.
\end{align*}
This estimate together with \eqref{est-II} shows that
\begin{align*}
\Big|e(t)-\int_{\R^3}|u_{q+1}|^2\dd x-2\delta_{q+2}\Big|
\lesssim& C_0\lambda^{-10}_q \le \delta_{q+2}\lambda^{-\alpha}_q.
\end{align*}
Hence, we complete the proof of Proposition \ref{est-E}.
\end{proof}
\subsection{Iterative estimates at $q+1$ level}\label{checkq+1}
Now we collect these estimates together to show that $u_{q+1}$ and the Reynolds  stress $\RR_{q+1}$  satisfies \eqref{uq-tigh}--\eqref{e:velocity-diff}.

Firstly, we define $u_{q+1}=u_{\ell_q}+w_{q+1}=\uqql+\uqqnl$, where
\[\uqql=\ulql+\wqloc, \quad\uqqnl=\ulqnl+\wtq.\]
Using  \eqref{e:vq-C0}--\eqref{e:vq-H5}, \eqref{e:v_ell-vq} and Proposition \ref{estimate-wq+1}, we deduce that
\begin{align*}
   &\|\uqql\|_{L^{\infty}_t L^2}\le \sum_{j=0}^{q}C_0\delta^{1/2}_{j+1}+C\lambda^{10}_q\ell_q+\frac{1}{2}C_0\delta^{1/2}_{q+1}\le \sum_{j\geq0}^{q+1}C_0\delta^{1/2}_{j+1},\\
   &\|\uqqnl\|_{\widetilde{L}^{\infty}_tB^{1/2}_{2,1}}\le \sum_{j=0}^{q-1}\lambda^{-1}_j+C\lambda^{10}_q\ell_q+\lambda^{-20}_q\le \sum_{j=0}^{q}\lambda^{-1}_j,
   \end{align*}
   and
   \begin{align*}
   &\|(\uqql,\uqqnl)\|_{{L}^{\infty}_t H^5}\le C\lambda^{10}_q+\lambda^6_{q+1}\le \lambda^{10}_{q+1}.
\end{align*}
Since
$$\spt_x \wpq,\spt_x\wcq,\spt_x\wttq,\spt_x\wttqc\subseteq\Omega_{q+1}$$
and $\spt_x\ulql=\Omega_{q}+[-\lambda^{-1}_q, \lambda^{-1}_q]^3\subseteq \Omega_{q+1}$ in \eqref{supp-vlq}, we have
\[\spt_x \uqql\subseteq\Omega_{q+1}.\]
Combining with \eqref{suppx-R}, we prove that estimates \eqref{uq-tigh}--\eqref{e:vq-H5} and \eqref{e:RR_q-tigh} hold with $q$ replaced by $q+1$. Proposition \ref{R-q+1} and Proposition \ref{est-E} directly show \eqref{e:RR_q-C0} and \eqref{et} hold at $q+1$ level.

 Thanks to \eqref{e:v_ell-vq} and \eqref{estimate-w}, we obtain
\begin{align*}
    \|u_{q+1}-u_q\|_{ L^\infty_t L^2}\le &   \|u_{q+1}-u_{\ell_q}\|_{ L^\infty_t L^2}+\|w_{q+1}\|_{L^\infty_t L^2}
    \le C \lambda_q\ell_q+\frac{1}{2}C_0\delta^{1/2}_{q+1}\le C_0\delta^{1/2}_{q+1}.
\end{align*}
where the last inequality holds for large enough $a$. Therefore, we complete the proof of Proposition ~\ref{iteration}.

\section{Proof of Proposition \ref{p:main-prop2}}\label{proof-2}
In fact, by  tracking the construction process of the perturbation, it is not difficult to show Proposition \ref{p:main-prop2}. Under the assumption that
\begin{align*}
 \uql=\tuql, \quad\uqnl=\tuqnl, \quad \RR_{q } ={\tRR}_{q },\quad\text{on}\quad [0, \tfrac{T}{4}+\lambda^{-1}_q],
\end{align*}
then one immediately infers that
\begin{align*}
 \ulql=\tulql, \quad\ulqnl=\tulqnl, \quad \RR_{\ell_{q} } ={\tRR}_{\ell_{q} },\quad\text{on}\quad [0, \tfrac{T}{4}+\tfrac{1}{2}\lambda^{-1}_q],
\end{align*}
where $(\ulql, \ulqnl, \RR_{\ell_q})$ and  $(\tulql,  \tulqnl, \tRR_{\ell_{q} })$ are defined  by mollifying $(\uql, \uqnl, \RR_{q})$ and  $(\tuql,  \tuqnl, {\tRR}_{{q} })$ via spatial  mollifier $\psi_{\ell_q}$ and the time mollifier $\varphi_{\ell_q}$, as is shown in \eqref{def-ulql}--\eqref{def-Rlq}. This ensures that
\begin{align}
&\Rem=\widetilde{R}^{\textup{rem}}_q, \quad \text{on} \quad [0, \tfrac{T}{4}+\tfrac{1}{2}\lambda^{-1}_q].\label{Remq=tRemq}\\
&\chi_q=\widetilde{\chi}_q, \qquad\,\,\,\, \text{on} \quad [0, \tfrac{T}{4}+\tfrac{1}{2}\lambda^{-1}_q].\label{cq=tcq}
\end{align}
Owning to \eqref{cq=tcq} and $e(t)=\widetilde{e}(t)$ for $0\le t\le \tfrac{T}{4}+\lambda^{-1}_1$, then $\rho_q$ and $\widetilde{\rho}_q$ defined by \eqref{energy2} satisfy that
\begin{align*}
    \rho_q(t)=\widetilde{\rho}_q(t), \quad \text{on} \quad [0, \tfrac{T}{4}+\tfrac{1}{2}\lambda^{-1}_q].
\end{align*}
The definition of amplitudes $a_{(k,q)}$ in \eqref{def-akq} shows that $a_{(k,q)}$ is determined by $ \RR_{\ell_{q} }$, $\chi_q$ and $\rho_q$. Therefore, the above relations tell that
\[a_{(k,q)}=\widetilde{a}_{(k,q)}, \quad \text{on}\quad [0, \tfrac{T}{4}+\tfrac{1}{2}\lambda^{-1}_q].\]
This together with the definitions of $\wpq$, $\wcq$, $\wttq$ and $\wttqc$ in \eqref{def-wpq}--\eqref{wttqc} immediately gives
\begin{align}
    \wpq=\twpq,\,\, \wcq=\twcq,\,\, \wttq=\twttq,\,\,\wttqc=\twttqc\quad \text{on} \quad [0, \tfrac{T}{4}+\tfrac{1}{2}\lambda^{-1}_q].\label{wql=twql}
\end{align}
Hence, Proposition \ref{def-F2} and Proposition \ref{F1} implies that
\begin{align}
&\RR^{(1)}_{q+1}=\tRR^{(1)}_{q+1},\quad \RR^{(2)}_{q+1}=\tRR^{(2)}_{q+1},\quad [0, \tfrac{T}{4}+\tfrac{1}{2}\lambda^{-1}_q]\label{Rq=tRq}\\
&F^{(1)}_{q+1}=\widetilde{F}^{(1)}_{q+1},\quad\, F^{(2)}_{q+1}=\widetilde{F}^{(2)}_{q+1},\quad \,\,[0, \tfrac{T}{4}+\tfrac{1}{2}\lambda^{-1}_q].\label{Fq=tFq}
\end{align}
By the definition of $F_{q+1}$ in \eqref{def-F_{q+1}}, \eqref{Remq=tRemq}, \eqref{wql=twql} and \eqref{Fq=tFq} together show  that
\[F_{q+1}=\widetilde{F}_{q+1}, \quad \text{on} \quad [0, \tfrac{T}{4}+\tfrac{1}{2}\lambda^{-1}_q]. \]
Since $\wtq$ is the unique local-in-time solution of  the equations \eqref{e:wt} with the force term $F_{q+1}$, we  conclude that
\begin{align}\label{wtq=twtq}
\wtq=\twtq, \quad \text{on} \quad [0, \tfrac{T}{4}+\tfrac{1}{2}\lambda^{-1}_q].
\end{align}
Then, by defining $(\wqloc, \uqql, \uqqnl, \widetilde{w}^{\textup{loc}}_{q+1}, \widetilde{u}^{\textup{non-loc}}_{q+1},\widetilde{u}^{\textup{non-loc}}_{q+1})$ in \eqref{def-wloc}--\eqref{def-uloc}, one infers from \eqref{wql=twql} and \eqref{wtq=twtq} that
\begin{align}\label{uqq=tuqq}
\wqloc=\widetilde{w}^{\textup{loc}}_{q+1}, \,\,\uqql=\widetilde{u}^{\textup{loc}}_{q+1}, \,\,\uqqnl=\widetilde{u}^{\textup{non-loc}}_{q+1},\,\,\text{on} \quad [0, \tfrac{T}{4}+\tfrac{1}{2}\lambda^{-1}_q].
\end{align}
Since $\RR_{q+1}$ and $\tRR_{q+1}$ are defined in \eqref{def-R-q+1},  we obtain by \eqref{Rq=tRq}, \eqref{wtq=twtq} and \eqref{uqq=tuqq} that
\[\RR_{q+1}=\tRR_{q+1}, \quad \text{on} \quad [0, \tfrac{T}{4}+\tfrac{1}{2}\lambda^{-1}_q]. \]
Therefore, we complete the proof of Proposition \ref{p:main-prop2}.

\section{Proof of Theorem \ref{dns} and Theorem \ref{Couette}}\label{proof of application}
In this section, we employ the method in proving Proposition \ref{iteration} to show Theorem \ref{dns} and Theorem \ref{Couette}.
\subsection{Proof of  Theorem \ref{dns}}
We will prove Theorem \ref{dns} by showing the following proposition.
\begin{proposition}\label{t:dns}
Let $T>0$, $\mathcal{D}$ be a smooth  bounded domain, $e(t)$ and $\widetilde{e}(t)$ with $e(t)=\widetilde{e}(t),~t\in[0, \tfrac{T}{2}]$. Then there exist  weak solutions $ u, \widetilde{u}  \in C([0,T];L^2(\mathcal{D}))$ for the system  \eqref{DNS} with $u(0,x)=\widetilde{u}(0,x)$ and
$$\|u(t)\|^2_{L^2}=e(t), ~~~\|\widetilde{u}(t)\|^2_{L^2}=\widetilde{e}(t).$$
\end{proposition}

We prove this proposition via an iteration proposition, which can be proved by the iterative scheme developed in Section 3.

Firstly, we assume that the parameters $b, \beta, \alpha, \lambda_q,\ell_q$ are defined in \eqref{b-beta} and \eqref{def-lq}. For a given open bounded domain $\mathcal{D}$  in Proposition \ref{t:dns}, there exists a closed cube domain $\Omega$ such  that $\Omega \subsetneqq \mathcal{D}$.  For $e(t)$ and $\widetilde{e}(t)$ in Proposition \ref{t:dns}, there exist an appropriate small constant $c$ such that $c\delta_2<e(t),\widetilde{e}(t)<\frac{3}{c}\delta_2$. Without loss of generality, we set $\Omega=[-\tfrac{1}{2}, \tfrac{1}{2}]^3\subsetneqq \mathcal{D}$ and $c=1$. Let $\Omega_q$ be given by~\eqref{omega}.

Secondly, we consider the following Navier-Stokes-Reynolds system:
 \begin{equation}\label{DNS-R}
\left\{ \begin{alignedat}{-1}
&\del_t v_q-\Delta v_q+(v_q\cdot\nabla) v_q  +\nabla p_q   =  \Div \RR_q, \quad{\rm in}\quad {\mathcal{D}}\times(0,T_q],
 \\
 &\nabla \cdot v_q = 0,\quad\qquad\qquad \qquad \qquad \qquad\qquad \,  {\rm in}\quad  {\mathcal{D}}\times(0,T_q],\\
& v_q=0, \quad\qquad\qquad \qquad \qquad \qquad\qquad \quad\, \, \,\, \, {\rm on}\quad  {\partial\mathcal{D}}\times[0,T_q],
\end{alignedat}\right.
\end{equation}
where $T_q=T+\sum_{k=q}^\infty\ell_k$.

Thirdly, for $t\in [0, T_q]$, we assume that $(v_q, p_q, \RR_q)$ is a smooth solution of the system \eqref{DNS-R} and satisfies that
\begin{align}
&\vq=\vql+\vdq,   \quad\quad\quad\quad\quad\,\,\Div \vql=\Div \vdq=0, \label{v-decom}\\
    &\|\vql\|_{L^{\infty}_tL^2}\leq\sum_{j\geq0}^{q}C_0\delta^{1/2}_{j+1},\quad\|\vql\|_{L^{\infty}_{t}H^5} +\|\partial_t\vql\|_{L^\infty_t H^2}\leq \lambda^{10}_q,\label{est:vql-dns}\\
& \|\vdq\|_{L^{\infty}_tH^1_0(\mathcal{D}_q)} \le \sum_{j\geq0}^{q-1}\lambda^{-1}_{j}, \quad\|\vdq\|_{L^{\infty}_{t}H^3}\le \lambda^{10}_q,
    \\
    &\| \RR_q \|_{L^{\infty}_{t}L^1 }  \le \delta_{q+1}\lambda_q^{-4 \alpha} , \qquad \| \RR_q \|_{L^{\infty}_{t}W^{5,1} }  \le  \lambda_q^{10},
       \\
     &{\spt_{x}} \vql, {\spt_{x}} \RR_q\subseteq \Omega_q   ,\\
    &\delta_{q+1}\leq e(t)-\int_{\mathbb{R}^3}|v_q|^2dx \leq 3\delta_{q+1},\label{e-vq}
\end{align}
where $C_0$ is a sufficiently large universal constant. Then, we aim to show the  iteration proposition as follows.
\begin{proposition}\label{ite-P-1}
Let $(v_1,p_1,\RR_1 )=(0,0,0)$. Assume that $(v_q,p_q,\RR_q )$ solves
\eqref{DNS-R} and satisfies  \eqref{v-decom}--\eqref{e-vq},
then there exists a solution $ (v_{q+1},  p_{q+1}, \RR_{q+1} )$ of the system \eqref{DNS-R} on ${\mathcal{D}}\times(0,T_{q+1}]$, satisfying \eqref{v-decom}--\eqref{e-vq} with $q$ replaced by $q+1$, and such that
\begin{align}
\|v_{q+1} - v_q\|_{L^\infty_tL^2(\mathcal{D})} &\leq  C_0\delta_{q+1}^{1/2},\quad~t\in[0,T_{q+1}].
\label{e:vq-diff}
\end{align}
\end{proposition}
Furthermore, we establish a proposition similar to Proposition \ref{p:main-prop2}.
\begin{proposition}\label{ite-P-2}Let $T>0$ and $e(t)=\widetilde{e}(t)$ for $t\in[0, \tfrac{T}{4}+\lambda^{-1}_1  ]$. Suppose that  $(v_q ,p_q,\RR_q)$ solves
\eqref{DNS-R}  and satisfies \eqref{v-decom}--\eqref{e-vq},    $(\widetilde{v}_q, \widetilde{p}_q, \widetilde{\RR}_q)$  solves
\eqref{DNS-R} and satisfies \eqref{v-decom}--\eqref{e-vq} with $e(t)$ replaced by $\widetilde{e}(t)$. Then   if
\begin{align*}
 \vql=\widetilde{v}^{\textup{loc}}_q, \quad{v}^{\textup{loc}}_q=\widetilde{v}^{\textup{D}}_q, \quad \RR_{q } ={\tRR}_{q },\quad\text{on}\quad [0, \tfrac{T}{4}+\lambda^{-1}_q],
\end{align*}
then we have
\begin{align*}
 {v}^{\textup{loc}}_{q+1}=\widetilde{v}^{\textup{loc}}_{q+1}, \quad{v}^{\textup{D}}_{q+1}=\widetilde{v}^{\textup{D}}_{q+1}, \quad \RR_{q+1 } ={\tRR}_{q+1},\quad\text{on}\quad [0, \tfrac{T}{4}+\lambda^{-1}_{q+1}].
\end{align*}
\end{proposition}
By using Proposition \ref{ite-P-1} and Proposition \ref{ite-P-2}, we immediately show Proposition \ref{t:dns}. Readers can refer to Section \ref{P-to-T} for the detailed proof. The proof of  Proposition \ref{ite-P-2}, can be copied almost word by word from the proof of Proposition  \ref{p:main-prop2} and thereby we omit the details.

We are in position to prove Proposition  \ref{ite-P-1}. This strategy is similar to that of Proposition \ref{p:main-prop2}. Therefore, we only need to outline the main ideas of the proof and highlight the differences.
\vskip 2mm
\noindent{\textit{Proof of Proposition  \ref{ite-P-1}}.}\,\, As is shown in the proof of  Proposition \ref{p:main-prop2}, we construct $(v_{q+1}, p_{q+1}, R_{q+1})$ by two steps: mollification and the perturbation.
\vskip 2mm
\noindent \textbf{Mollification}: ${(\vq, p_q, \RR_q)\mapsto (v_{\ell_q},  p_{\ell_q}, \RR_{\ell_q})}$.  Due to the restrictions of Dirichlet boundary condition, here $(v_{\ell_q},  p_{\ell_q}, \RR_{\ell_q})$ is different from that given by \eqref{def-ulq}, \eqref{def-plq} and \eqref{def-Rlq}.  More precisely, using the spatial mollifiers $\psi_{\ell_q}(x)$ and temporal mollifiers $\varphi_{\ell_q}(t)$, we define $(v_{\ell_q},  p_{\ell_q}, \RR_{\ell_q})$ by
\begin{align*}
    &v_{\ell_q}=\vql\ast\psi_{\ell_q}\ast\varphi_{\ell_q}+\vdq=:\vlql+\vdq,\\
   & p_{\ell_q}=p_q,\quad
    \RR_{\ell_q}=\RR_q\ast\psi_{\ell_q}\ast\varphi_{\ell_q}.
\end{align*}
One verifies that $ (v_{\ell_q},  p_{\ell_q}, \RR_{\ell_q})$ obeys the following system
 \begin{equation}\label{DNS-Rl}
\left\{ \begin{alignedat}{-1}
&\del_t v_{\ell_q}-\Delta v_{\ell_q}+(v_{\ell_q}\cdot\nabla) v_{\ell_q}  +\nabla p_{\ell_q}   =  \Div \RR_{\ell_q}+ \Rem, \quad\quad{\rm in}\quad {\mathcal{D}}\times(0,T_{q+1}],
 \\
&  \nabla \cdot v_{\ell_q} = 0, \qquad\qquad\qquad\qquad\qquad\qquad\qquad\qquad\qquad\quad{\rm in}\quad  {\mathcal{D}}\times(0,T_{q+1}],\\
& v_{\ell_q}=0, \qquad\qquad\qquad\qquad\qquad\qquad\qquad\qquad\qquad\quad\quad\,\,\,{\rm on}\quad  {\partial\mathcal{D}}\times[0,T_{q+1}],
\end{alignedat}\right.
\end{equation}
where
\begin{align*}
\Rem=&\Div(\RR_q-\RR_{\ell_q})+\Div\big((\vlql-\vql)\otimes v_{\ell_q}\big)+\Div\big(v_q\otimes(\vlql-\vql)\big)\\
&+ (\partial_t-\Delta)(\vlql-\vql).
\end{align*}
From the definitions of $(v_{\ell_q}, \RR_{\ell_q}, \Rem)$ and the estimates \eqref{est:vql-dns}--\eqref{e-vq}, we have, for any integers $L,N\ge 0$,
\begin{align*}
&\|\partial^L_t\vlql \|_{L^\infty_tH^{N+3}} \lesssim  \lambda^{10}_{q} \ell_q^{-N-L},
\\
&{\|\RR_{\ell_q}\|_{L^\infty L^1} \le  \delta_{q+1} \lambda^{-4\alpha}_q},  \quad\|\partial^L_t\RR_{\ell_q}\|_{L^\infty_tW^{N+5,1}} \lesssim   \lambda^{10}_{q} \ell_q^{-N-L},
\\
&\|\Rem\|_{L^\infty_t L^2} \lesssim \lambda^{-40}_q, \,\,\,\qquad\|\Rem\|_{L^\infty_t H^2} \lesssim \lambda^{20}_q,\\
 & \frac{\delta_{q+1}}{2}\leq e(t) - \int_{\mathbb R^3}  |v_{\ell_q}|^2 \dd x\leq 4\delta_{q+1}.
\end{align*}
}

\vskip 2mm
\noindent \textbf{The perturbation}: ${(v_{\ell_q},  p_{\ell_q}, \RR_{\ell_q})}\mapsto (v_{q+1}, p_{q+1}, \RR_{q+1})$. We give $w_{q+1}$ by
\begin{align*}
    w_{q+1}=\wpq+\wcq+\wttq+\wttqc+\wtq,
\end{align*}
where $\wpq,\wcq,\wttq$ and $\wttqc$ are consistent with that in \eqref{def-wpq}--\eqref{wttqc}. Therefore, we immediately obtain the equality \eqref{eq-wpwp}. Following the method on deriving \eqref{decom-(a-phi)}, we perform a more refined decomposition on certain items of $F^{(1)}_{q+1}$ in \eqref{eq-wpwp} via Lemma \ref{tracefree} and infer that there exist a trace-free matrix $\mathring{\widetilde{R}}^{(1)}_{q+1}$, a pressure $\widetilde{P}^{(1)}_{q+1}$ and error term $\widetilde{F}^{(1)}_{q+1}$ such that
\begin{align}\label{wpwp-dec-1}
\Div(\wpq\ootimes \wpq)+\Div\RR_{\ell_q}=\Div \mathring{\widetilde{R}}^{(1)}_{q+1}-\partial_t\wttq+\widetilde{F}^{(1)}_{q+1}+\nabla \widetilde{P}^{(1)}_{q+1},
\end{align}
satisfying
\begin{align}\label{wpwp-R-1}
\|\mathring{\widetilde{R}}^{(1)}_{q+1}\|_{L^\infty_t L^1}\lesssim {\ell^{-8}_q}\lambda^{-\frac{1}{16}}_{q+1}, \quad \|\mathring{\widetilde{R}}^{(1)}_{q+1}\|_{L^\infty_t W^{5,1}}\lesssim {\ell^{-8}_q}\lambda^{5}_{q+1},\quad \spt_x \mathring{\widetilde{R}}^{(1)}_{q+1}\subseteq\Omega_{q+1},
\end{align}
and
\begin{align}\label{wpwp-F-1}
    \|\widetilde{F}^{(1)}_{q+1}\|_{{L}^\infty_t L^2}\lesssim\ell^{-22}_q \lambda^{-\frac{1}{32}}_{q+1},\quad \|\widetilde{F}^{(1)}_{q+1}\|_{{L}^\infty_t H^2}\lesssim \ell^{-18}_q\lambda^{\frac{35}{8}}_{q+1}, \quad \spt_x \widetilde{F}^{(1)}_{q+1}\subseteq\Omega_{q+1}.
\end{align}
Similarly, applying Lemma \ref{tracefree} to $\eqref{pt-wtc}$ and combining with Proposition~\ref{F1}, we conclude that  there exist a trace-free matrix $\RR^{(2)}_{q+1}$, a pressure $\nabla \widetilde{P}^{(2)}_{q+1}$ and error term $\widetilde{F}^{(2)}_{q+1}$ such that
\begin{align*}
    \partial_t (\wpq+\wcq+\wttqc)
    =:&\Div\mathring{\widetilde{R}}^{(2)}_{q+1}+\nabla \widetilde{P}^{(2)}_{q+1}+\widetilde{F}^{(2)}_{q+1},
\end{align*}
where $\RR^{(2)}_{q+1}$ and $F^{(2)}_{q+1}$   satisfies \eqref{est-Rq2} and \eqref{wpwp-F-1}, respectively. Then we define
\begin{align*}
    \widetilde{F}_{q+1}:=& \widetilde{F}^{(1)}_{q+1}+\widetilde{F}^{(2)}_{q+1}+ \Rem
\end{align*}
and we have
\begin{align}\label{est-tF}
\|\widetilde{F}_{q+1}\|_{L^\infty_t L^2}\lesssim\ell^{-22}_q\lambda^{-\frac{1}{32}}_{q+1}+\lambda^{-40}_q\lesssim\lambda^{-40}_q, \quad \|\widetilde{F}_{q+1}\|_{L^\infty_t H^2}\le \lambda^5_{q+1}.
\end{align}
Let $\wtq$ obey the following system
 \begin{equation}\label{wtq-NSD}
\left\{ \begin{alignedat}{-1}
&\del_t \wtq-\Delta \wtq+\Div (\wtq\ootimes \wtq)+\Div(\wtq\ootimes{v}^{\textup{D}})\\
&\qquad\quad+\Div({v}^{\textup{D}}\ootimes \wtq)  +\nabla p^{(\textup{ns})}_{q+1} + \widetilde{F}_{q+1}=0, \quad\,\,{\rm in}\quad {\mathcal{D}}\times(0,T_{q+1}],
 \\
&  \nabla \cdot \wtq = 0,\qquad\quad\qquad\quad \qquad\quad\quad\quad\quad\quad\quad\quad\,\,{\rm in}\quad  \mathcal{D}\times(0,T_{q+1}],\\
& \wtq=0,\quad\qquad\quad\qquad\quad \qquad\quad\qquad\qquad\qquad  \,\,\,\,\, {\rm on} \quad{\partial\mathcal{D}}\times[0,T_{q+1}],\\
&\wtq=0,\quad\qquad\quad\qquad\quad \qquad\quad\qquad\qquad\qquad \,\,\,\,\,  {\rm on} \quad{\mathcal{D}}\times\{t=0\}.
\end{alignedat}\right.
\end{equation}
By the local well-posedness theory in \cite{FK}, owning to \eqref{est-tF} and $\|{v}^{\textup{D}}_{{\ell_q}}\|_{L^\infty_t H^1_0(\mathcal{D})}\lesssim \lambda^{-1}_0$, for choosing $\lambda_0$ large enough, there exists a unique solution $\wtq$ of the system \eqref{wtq-NSD} on $[0, T_{q+1}]$ with
\[\|\wtq\|_{L^\infty_t H^1_0(\mathcal{D})}\le \lambda^{-20}_q, \quad \|\wtq\|_{L^\infty_t H^3(\mathcal{D})}\lesssim \lambda^{5}_{q+1}.\]
We define
\begin{align*}
&\wqloc:={\wpq+\wcq+\wttq+\wttqc},\quad  w_{q+1}=\wqloc+\wtq,\\
&v_{q+1}=v_{\ell_q}+w_{q+1}=(\vlql+\wqloc)+({v}^{\textup{D}}_q+\wtq)=:v^{\textup{loc}}_{q+1}+{v}^{\textup{D}}_{{q+1}}, \\
&{p_{q+1}=p_{\ell_q}}+\nabla p^{\textup{(ns)}}_{q+1}-\widetilde{P}^{(1)}_{q+1}-\widetilde{P}^{(2)}_{q+1}-\frac{2}{3}v_{\ell_q}\cdot w_{q+1}-\frac{1}{3}w_{q+1}\cdot w_{q+1}.\nonumber
\end{align*}
Then one infers that $(v_{q+1}, p_{q+1}, \RR_{q+1})$ satisfies the system \eqref{DNS-R} on ${\mathcal{D}}\times(0,T_{q+1}]$, where
\begin{align}
\RR_{q+1}=&\mathring{\widetilde{R}}^{(1)}_{q+1}+\mathring{\widetilde{R}}^{(2)}_{q+1}+ w_{q+1}\ootimes w_{q+1}- \wpq\ootimes \wpq- \wtq\ootimes \wtq\nonumber\\
&-\nabla \wqloc -(\nabla \wqloc)^{\TT}+\wqloc\ootimes v_{\ell_q}+v_{\ell_q}\ootimes\wqloc.\nonumber
\end{align}
Following the proof as shown in Section \ref{checkq+1}, we conclude that $(v_{q+1}, p_{q+1}, \RR_{q+1})$ satisfies  \eqref{v-decom}--\eqref{e-vq} with $q$ replaced by $q+1$ and \eqref{e:vq-diff}. Hence we complete the proof of Proposition \ref{ite-P-1}.

\subsection{Proof of Theorem \ref{Couette}}
Let $U$ be  the 3D  Couette flow, namely $U=(x_2,0,0)$ in $\R^3$. Let $u$ solve the equations \eqref{NS} and we consider the perturbation $v=u-U$, which satisfies
\begin{equation}\label{couette equation}
\left\{ \begin{alignedat}{-1}
&\del_t v-\Delta v+(v\cdot\nabla) v+x_2\partial_1v+(v^2, 0,0)  +\nabla P   =  0, &{\rm in}\quad {\mathbb{R}^3}\times [0,T],
 \\
&  \nabla \cdot v  = 0, &{\rm in}\quad  {\mathbb{R}^3}\times[0,T],
\end{alignedat}\right.
\end{equation}
where $P=p-P_{Cou}$.

Our aim is to prove the following proposition, which immediately shows Theorem \ref{Couette}.
\begin{proposition}\label{Couette1}
Given $0<\epsilon\ll 1$, there exists a weak solution $\ve$ of the system \eqref{couette equation} such that
\[\|\ve(0, \cdot)\|_{L^2(\R^3)}\leq \epsilon,\]
meanwhile
\[\|\ve(t, \cdot)\|_{L^2(\R^3)}\geq \epsilon^{-1/4},\,\,\,\,\forall t\in[3\epsilon^{1/2},5\epsilon^{1/2}].\]
\end{proposition}
We will reduce  Proposition  \ref{Couette1} to Proposition \ref{Couette-ite} by induction on $(v_q, P_q, \RR_q)$ below. Let $(v_q, P_q, \RR_q)$ obey the following system:
\begin{equation}\label{C-vq}
\left\{ \begin{alignedat}{-1}
&\del_t v_q-\Delta v_q+(v_q\cdot\nabla) v_q+x_2\partial_1v_q+(v_q^2, 0,0)  +\nabla P_q  = \Div \RR_q, &{\rm in}\quad {\mathbb{R}^3}\times [0,T],
 \\
&  \nabla \cdot v_q  = 0, &{\rm in}\quad  {\mathbb{R}^3}\times[0,T].
\end{alignedat}\right.
\end{equation}

Suppose that the constants $b, \beta, \alpha$ are consistent with these in \eqref{b-beta}. Given a small enough constant $\epsilon>0$, let $a$ be a large number  depending on $b,\beta,\alpha, \epsilon$. We define
\begin{align}
    \lambda_q :=   a^{b^q},  \quad \delta_q &:= \epsilon^{-1/2}\lambda_2^{2\beta}\lambda_q^{-2\beta}, \quad \ell_q:=\lambda^{-60}_q,\quad q\ge 0.
\end{align}
Here we choose $a$ large enough such that $\delta_q\ll\epsilon^{-\frac{1}{2}}\ll\lambda^{\alpha}_1 $ for $q\ge 3$. Let $e(t)=\frac{t}{\epsilon}$ and $\Omega_q$ be defined in \eqref{omega}.  The estimates we propagate inductively  are: for large enough constant $C_0$ and $q\ge 1$,
\begin{align}
&v_q=\vql+\vqnl,\quad\|v_q(0, \cdot)\|_{L^2(\R^3)}\leq \epsilon,\label{e:vq}\\
&\|\vql\|_{L^\infty_t L^2}\leq\sum_{j=0}^{q-1}C_0\delta^{1/2}_{j+1},\quad
\|\vqnl\|_{\widetilde L^{\infty}_tB^{1/2}_{2,1} \cap \widetilde L^{1}_tB^{5/2}_{2,1} } \le \sum_{j=0}^{q-1}\lambda^{-1}_{j} ,\label{e:vq-L2}\\
&\|(\vql,~\vqnl)\|_{L^\infty_t H^5}\leq \lambda^{10}_q  ,
    \\
    & \| \RR_q \|_{L^\infty_t L^1}\le \delta_{q+1}\lambda_q^{-4 \alpha} , \quad \| \RR_q \|_{L^{\infty}_tW^{5,1} }\le  \lambda_q^{10},
    \label{e:RR-q}
    \\
  &   {\spt_{x}}  (\vql, \RR_q)  \subseteq \Omega_q,
    \qquad {\spt_{t}}  (\vql, \RR_q)  \subseteq [\epsilon^{1/2}+\lambda^{-\alpha}_{q },6\epsilon^{1/2}-\lambda^{-\alpha}_{q}], \label{spt-RR-q}\\
    & \delta_{q+1}\leq e(t)-\int_{\mathbb{R}^3}|v_q|^2dx \leq 3\delta_{q+1},  \quad t\in [3\epsilon^{1/2}-\lambda^{-\alpha}_{q },5\epsilon^{1/2}+\lambda^{-\alpha}_{q}]  .\label{et-vq}
\end{align}

\begin{proposition}\label{Couette-ite} Let $T>0$. Assume that $(v_q,P_q,\RR_q )$ solves
\eqref{C-vq} and satisfies  \eqref{e:vq}--\eqref{et-vq},
then there exists a solution $ (v_{q+1},  P_{q+1}, \RR_{q+1} )$, satisfying \eqref{e:vq}--\eqref{et-vq} with $q$ replaced by $q+1$, and such that
\begin{align}
        \|v_{q+1} - v_q\|_{L^\infty_tL^2} &\leq  C_0\delta_{q+1}^{1/2}.
\end{align}
\end{proposition}
\begin{proof}As a matter of fact, the proof of Proposition \ref{Couette-ite} is similar to that of Proposition~\ref{iteration}. For the sake of brevity, we outline the main points of the proof.\\
\textbf{Step 1: Mollification}.\,\, Following the mollification process as shown in Section 4.1, we define $(v_{\ell_q},  P_{\ell_q}, \RR_{\ell_q})$ by mollifying  $(v_q, P_q, \RR_q)$. Then the estimates  \eqref{supp-vlq}, and \eqref{e:v_ell-vq}--\eqref{e:R_rem} in Proposition~\ref{p:estimates-for-mollified} hold. Moreover,
  \begin{align*}
      &{\spt_{t}}  (\vlql, \RR_{\ell_q})  \subseteq [\epsilon^{1/2}+\tfrac{1}{2}\lambda^{-\alpha}_{q },6\epsilon^{1/2}-\tfrac{1}{2}\lambda^{-\alpha}_{q}],
     \end{align*}
     and
     \begin{align*}
      & \frac{\delta_{q+1}}{2}\leq e(t)-\int_{\mathbb{R}^3}|v_{\ell_q}|^2dx \leq 4\delta_{q+1},  \quad t\in [3\epsilon^{1/2}-\lambda^{-\alpha}_{q },5\epsilon^{1/2}+\lambda^{-\alpha}_{q}].
  \end{align*}
\textbf{Step 2: Perturbation. }  Firstly, the amplitude function $\widetilde{a}_{(k,q)}$ is defined by
 \begin{equation*}
      \widetilde{a}_{(k,q)}(t,x)=\widetilde{\eta}_{q}(t,x)
 a_k\Big({\rm Id}-\frac{\RR_{\ell_{q}}}{\chi_{q}\widetilde{\rho}_q}\Big)
 (\chi_q \widetilde{\rho}_q)^{1/2},
 \end{equation*}
 where $a_k, \chi_q$ are consistent with that in \eqref{def-akq}, $\widetilde{\eta}_{q}(t,x)\in C^\infty_c(\R\times \R^3; [0,1])$ satisfies that
  \begin{align*}
&{\rm(i)}\,\spt_x \widetilde{\eta}_{q}\subseteq\Omega_{q+1},\quad    \spt_t \widetilde{\eta}_{q}\subseteq [\epsilon^{1/2}+\lambda^{-\alpha}_{q+1},6\epsilon^{1/2}-\lambda^{-\alpha}_{q+1}],\\
&{\rm(ii)}\,\widetilde{\eta}_{q}(t,x)\equiv 1, \,\, \text{if}\,\,t\in [\epsilon^{1/2}+\tfrac{1}{2}\lambda^{-\alpha}_{q },6\epsilon^{1/2}-\tfrac{1}{2}\lambda^{-\alpha}_{q}], x\in \Omega_q+[-\lambda^{-1}_q, \lambda^{-1}_q]^3,\\
&{\rm(iii)}\,|\partial^M_t\partial^N_x\widetilde{\eta}_q|\lesssim \lambda^{N+M}_q,\,\,\forall\,N,M\ge 0.
   \end{align*}
Choose the cutoff function $\zeta(t)\in C^\infty_c(\R; [0,1])$ such that
\begin{align*}
\,\,\spt \,\zeta=[3\epsilon^{1/2}-\lambda^{-\alpha}_{q },5\epsilon^{1/2}+\lambda^{-\alpha}_{q}], \quad \zeta(t)\equiv 1, \,\text{for}\,\, t\in [3\epsilon^{1/2}-\lambda^{-\alpha}_{q+1},5\epsilon^{1/2}+\lambda^{-\alpha}_{q+1}].
\end{align*}
We define $\widetilde{\rho}_q$ by
 \[\widetilde{\rho}_q=\delta_{q+1}(1-\zeta(t))+\zeta(t)\bar{\rho}_q,\]
where $\bar{\rho}_q$ is defined by replacing $\eta_q$ with  $\widetilde{\eta}_q$ in \eqref{energy2}.

Secondly, we define $v_{q+1}=v_{\ell_q}+w_{q+1}$, where  the perturbation $w_{q+1}$ is given by
 \begin{align*}
     w_{q+1}=\wpq+\wcq+\wttq+\wttqc+\wtq.
 \end{align*}
 Here  $\wpq$, $\wcq$, $\wttq$ and $\wttqc$ are defined by replacing $a_{(k,q)}$ with  $\widetilde{a}_{(k,q)}$ in \eqref{def-wpq}--\eqref{wttqc}, respectively. Therefore, one deduces from Proposition \ref{def-F2} and Proposition \ref{F1} that the error term $F_{q+1}$ defined by \eqref{def-F_{q+1}} satisfies that
 \begin{equation}\label{F-S}
 \|F_{q+1}\big\|_{\widetilde L^{\infty}_tB^{-3/2}_{2,1} }\lesssim \lambda^{-40}_{q}.
 \end{equation}
We define $\wtq$ on $[0, 6\epsilon^{1/2}]$ by solving the the following system
\begin{equation}
\left\{ \begin{alignedat}{-1}
&\del_t \wtq-\Delta \wtq+\Div (\wtq\ootimes \wtq)+\Div(\wtq\ootimes\vlqnl)\\
&+\Div(\vlqnl\ootimes \wtq) +x_2\partial_1\wtq+((\wtq)^2, 0,0) +\nabla p^{(\textup{ns})}_{q+1}  = F_{q+1},
\\
 & \Div \wtq = 0,
  \\
  & \wtq |_{t=0}=  0.
\end{alignedat}\right.
 \label{e:wtq-C}
\end{equation}
Under the condition \eqref{F-S}, Proposition \ref{posed-COU} implies that  there exists a $\wtq$ on  $[0, 6\epsilon^{1/2}]$ such that
\begin{equation}
    \|\wtq\|_{\widetilde L^{\infty}_tB^{1/2}_{2,1} }\le \lambda^{-20}_q.
\end{equation}

Finally, by defining $v_{q+1}=v_q+w_{q+1}$, it is not difficult to establish the estimates as shown in Proposition ~ \ref{estimate-wq+1} and  Proposition~ \ref{R-q+1}, and \eqref{et-vq} holds at $q+1$ level.  Hence we finish the proof of Proposition \ref{Couette-ite}.

 \end{proof}

\appendix

\section{}\label{App}
In this section, we compile several useful tools {including geometric Lemma, an improved H\"{o}lder inequality, the definitions of mollifiers and Lerner-Chemin
space, the inverse divergence iteration step, the local well-posedness of a perturbation system of the Navier-Stokes equations near the Couette flow. }
\begin{lemma}[Geometric Lemma \cite{2Beekie}]\label{first S}Let $B_{\sigma}({\rm Id})$ denote the ball of radius $\sigma$ centered at $\rm Id$ in the space of $3\times3$ symmetric matrices.
There exists a set $\Lambda\subset\mathbb{S}^2\cap\mathbb{Q}^3$ that consists of vectors $k$ with associated orthonormal basis $(k,\bar{k},\bar{\bar{k}}),~\epsilon>0$ and smooth function $a_{k}:B_{\epsilon}(\rm Id)\rightarrow\mathbb{R}$ such that, for every positive definite symmetric matrix $R\in B_{\epsilon}(\rm Id)$, we have the following identity:
$$R=\sum_{k\in\Lambda}a^2_{k}(R)\bar{k}\otimes\bar{k}.$$
\end{lemma}
\begin{remark}\label{Lambda}
  For instance, $\Lambda=\{\frac{5}{13}e_1\pm \frac{12}{13}e_2, \frac{12}{13}e_1\pm \frac{5}{13}e_3, \frac{5}{13}e_2\pm \frac{12}{13}e_3\}$ and $(k,\bar{k},\bar{\bar{k}})$ are as follows:
\begin{table}[ht]
\renewcommand\arraystretch{1.2}
\begin{tabular}{p{2cm}|p{2cm}|p{2cm}}
\toprule[1.2pt]
$\qquad k$& $\qquad {\bar{k}}$ & $\qquad {\bar{\bar{k}}}$\\\midrule[0.5pt]
$\frac{5}{13}e_1\pm \frac{12}{13}e_2$ &{$\frac{5}{13}e_1\mp\frac{12}{13}e_2 $}&\qquad  $e_3$ \\\hline
$\frac{12}{13}e_1\pm \frac{5}{13}e_3$&{$\frac{12}{13}e_1\mp \frac{5}{13}e_3$}& \qquad $e_2$\\\hline
$\frac{5}{13}e_2\pm \frac{12}{13}e_3$ &{$\frac{5}{13}e_2\mp \frac{12}{13}e_3$}& \qquad $e_1$\\\bottomrule[1.2pt]
\end{tabular}
\end{table}
\end{remark}

{Next, we give the temporal and spatial mollifiers  which is useful in Section 4.
\begin{definition}[Mollifiers]\label{e:defn-mollifier-t}Let nonnegative functions $\varphi(t)\in C^\infty_c(-1,0)$ and $\psi(x)\in C^\infty_c(B_1(0))$ be standard mollifying kernels such that $\int_{\R}\varphi(t)\dd t=\int_{\R^3}\psi(x)\dd x=1$.
For each $\epsilon>0$, we define  two sequences of  mollifiers as follows:
\begin{align*}
     \varphi_{\epsilon}(t)
        := \frac1{\epsilon} \varphi\left(\frac t\epsilon\right), \quad\psi_\epsilon(x)
            := \frac1{\epsilon^3} \psi\left(\frac{x}\epsilon\right).
\end{align*}
\end{definition}

\begin{lemma}[\cite{1Cheskidov, MoS}]\label{Holder}Assume that $d\ge1$, $1\le p\le \infty$ and $\lambda$ is a positive integer. Let $\Omega=\big[-\tfrac{1}{2}, \tfrac{1}{2}\big]^3\subseteq\R^3$ and smooth function $f$ support on $\Omega$. $g: \TTT^3\to\R$ is a smooth function and $\TTT^3=\R^3/(\ZZ)^3$. Then we have
\[\Big|\|fg(\lambda\cdot)\|_{L^p(\Omega)}-\|f\|_{L^p(\Omega)}\|g\|_{L^p(\Omega)}\Big|\lesssim \lambda^{-\frac{1}{p}}\|f\|_{C^1(\Omega)}\|g\|_{L^p(\TTT^d)}.\]
 \end{lemma}}
 As is shown in \cite{BMNV}, Given a vector field $G$,a zero mean periodic function $\rho^{(0)}$, then we write $G\rho^{(0)}(\lambda x)$ as the divergence of asymmetric tensor up to a pressure term and a error term.
\begin{lemma}[Inverse divergence iteration step \cite{BMNV}]\label{tracefree}
 Let $\lambda\in \NN^{+}$ and $\{\rho^{(n)}\}_{0\leq n\leq N}$ be the zero mean smooth $\TTT^3$-periodic functions such that $\rho^{(n)}=\Delta\rho^{(n+1)}$. Then for any given vector field $G$, we have
 \begin{align}\label{itra-equ}
  G^i\rho^{(0)}(\lambda x)=\partial_j \mathring{R}^{ij}_{(0)}+\partial_iP_{(0)}+E^i_{(0)} ,
 \end{align}
 where the traceless symmetric stress $\mathring{R}_{(0)}$ is given by
 $$\mathring{R}^{ij}_{(0)}={\lambda^{-1}}\big(G^i\delta_{jl}+G^j\delta_{il}-\delta_{jk}\delta_{ik}G^l\big)(\partial_{ {l}}\rho^{(1)})\circ{(\lambda x)}=:
 \lambda^{-1}A^{(0)}_{ijl}(G)\big(\frac{\partial_{ {l}}\rho^{(0)}}{\Delta}\big)\circ{(\lambda x)},$$
 where the pressure term is given by
 $$P_{(0)}=\lambda^{-1}(2G^j\delta_{jl}- \delta_{jk} \delta_{jk}G^l)(\partial_{ {l}}\rho^{(1)})\circ{(\lambda x)}=:\lambda^{-1}B^{(0)}_l(G)\big(\frac{\partial_{ {l}}\rho^{(0)}}{\Delta}\big)\circ{(\lambda x)},$$
 and the error term $E^i$ is given by
  $$E^{i }_{(0)}=\lambda^{-1}\big(\partial_j(G^l\delta_{ik}\delta_{jk}-G^j\delta_{ik}\delta_{lk})-\partial_j G^iA^j_{ l} \big)(\partial_{ {l}}\rho^{(1)})\circ{(\lambda x)}:=
 \lambda^{-1}C^{(0)}_{il}(G)\big(\frac{\partial_{ {l}}\rho^{(0)}}{\Delta}\big)\circ{(\lambda x)}.$$
Moreover, applying the relation \eqref{itra-equ} to the error term $E^i$ at each step for $N$ times, one immediately shows that
\begin{align}\label{itra-equ}
  G^i\rho^{(0)}(\lambda x)=\partial_j \mathring{R}^{ij}_{(N)}+\partial_iP_{(N)}+E^i_{(N)} ,
 \end{align}
where
\begin{align*}
 &\RR^{ij}_{(N)}=  \sum_{n=1}^{N} \lambda^{-n}A^{(n)}_{ij\alpha_n}(G)\big(\frac{\partial^{\alpha_n}}{{\Delta}^n}\rho^{(0)}\big)\circ{(\lambda x)},\\
 &P_{(N)}=\sum_{n=1}^{N}\lambda^{-n}
 B^{ (n)}_{\alpha_n}(G)\big(\frac{\partial^{\alpha_n}}{{\Delta}^n}\rho^{(0)}\big)\circ{(\lambda x)},\\
 &E^i_{(N)}=\lambda^{-N}C^{(N)}_{i\alpha_N}(G) \big(\frac{\partial^{\alpha_N}}{{\Delta}^N}\rho^{(0)}\big)\circ{(\lambda x)}
\end{align*}
Here $A^{(n)}_{ij\alpha_n}(G)$ and  $B^{ (n)}_{\alpha_n}(G)$ are composed of the $n-1$th derivative of $G$ and $C^{(N)}_{i\alpha_n}$ is formed from the $N$th derivative of $G$. Particularly, if $\spt_x G\subseteq\Omega\subseteq\big[-\tfrac{1}{2}, \tfrac{1}{2}\big]^3$, we have
\[\spt_x \RR_{(N)}=\spt_x G.\]
 \end{lemma}

We give the definition of mixed time-spatial Besov spaces, the so-called Lerner-Chemin spaces.
\begin{definition}[\cite{BCD11, MWZ}]Let $T>0$, $s\in\R$ and $1\le r,p,q\le\infty$. The mixed  time-spatial Besov spaces ${L}^r_TB^s_{p,q}$ consists of all $u\in\mathcal{S}'$ satisfying
\begin{align*}
\|u\|_{\widetilde{L}^r_TB^s_{p,q}(\R^d)}\overset{\text{def}}{=}\Big{\|}(2^{js}\|\Delta_j u\|_{L^r([0,T];L^p(\R^d))})_{j\in\ZZ}\Big{\|}_{\ell^q(\ZZ)}<\infty,
\end{align*}
where $\Delta_j$ is localization nonhomogeneous operator from the Littlewood-Paley decomposition theory. Particularly, $H^s(\R^d)\sim B^s_{2,2}(\R^d)$.
\end{definition}
We present a result describing  the smoothing effect of the heat flow  in the context of Besov spaces.
\begin{lemma}[\cite{BCD11, MWZ}]\label{heat}Let $s\in \R$ and $1\le r_1, r_2, p, q\le \infty$ with $r_2\le r_1$. Consider the heat equation
\begin{align*}
\partial_t u-\Delta u=f,\qquad
  u(0,x)=u_0(x).
  \end{align*}
Assume that $u_0\in  B^s_{p,q}(\R^d)$ and $f\in \widetilde {L}^{r_2}_t( B^{s-2+\frac{2}{r_2}}_{p,q}(\R^d))$. Then the above equation has a unique solution $u\in \widetilde{L}^{r_1}_t( \dot B^{s+\frac{2}{r_1}}_{p,q}(\R^d))$
satisfying
\[\|u\|_{\widetilde{L}^{r_1}_T(  B^{s+\frac{2}{r_1}}_{p,q}(\R^d))}\le C(1+T)\big(\|u_0\|_{ B^s_{p,q}(\R^d)}+\|f\|_{\widetilde{L}^{r_2}_T( \dot B^{s-2+\frac{2}{r_2}}_{p,q}(\R^d))}\big),\]
where $C$ is a universal constant.
\end{lemma}
\begin{proposition}\label{posed-COU}Consider the Cauchy problem of the following system
\begin{equation}
\left\{ \begin{alignedat}{-1}
&\del_t u-\Delta u+\Div (u\ootimes u)+\Div(u\ootimes v)\\
&+\Div(v\ootimes u) +x_2\partial_1u+(u_2, 0,0) +\nabla p  = f,
\\
 & \Div\, u = 0,
  \\
  & u |_{t=0}= u_0.
\end{alignedat}\right.
 \label{e:Coutee-NS}
\end{equation}
Assume that $u_0\in  B^{\frac{1}{2}}_{2,1}, v\in \widetilde L^{\infty}B^{\frac{1}{2}}_{2,1}$ and  $f\in  \widetilde L^{\infty}B^{-\frac{3}{2}}_{2,1}$. Then there  exist $\varepsilon$ and $T\le \frac{1}{4}$ such that if
\begin{align}\label{con-S}
   \|u_0\|_{ B^{\frac{1}{2}}_{2,1}}+\|v\|_{\widetilde L^{\infty}_TB^{1/2}_{2,1}}+\|f\|_{\widetilde L^{\infty}_TB^{-3/2}_{2,1}}\le \varepsilon,
\end{align}
the system \eqref{e:Coutee-NS} has a solution $u\in  \widetilde L^{\infty}([0, T); B^{1/2}_{2,1})$ with
\begin{align*}
    \|u\|_{\widetilde L^{\infty}_T B^{\frac{1}{2}}_{2,1}}\lesssim \|u_0\|_{ B^{\frac{1}{2}}_{2,1}}+\|f\|_{\widetilde L^{\infty}_TB^{-3/2}_{2,1}}.
\end{align*}
\end{proposition}
 \begin{proof}We denote $X=(x_1, x_2, x_3)$. By introducing new variables
 $$\bar{X}=(\bar{x}_1,\bar{x}_2,\bar{x}_3)=(x_1-tx_2, x_2, x_3)$$
 and setting $\bar{u}(t, \bar{X})=u(t, X)$ as shown in \cite{CWZ}, we have
\begin{equation}
\left\{ \begin{alignedat}{-1}
  &\del_t \bar{u}-\Delta_L\bar{u}+\nabla_L \cdot (\bar{u}\ootimes \bar{u})+\nabla_L \cdot(\bar{u}\ootimes\bar{v})\\
&\qquad+\nabla_L \cdot(\bar{v}\ootimes \bar{u})+ (\bar{u}_2, 0,0) +\nabla_{L}\, \bar{p} = \bar{f},  \\
&\nabla_L \cdot \bar{u}=0,\\
&\bar{u} |_{t=0}= u_0,
\end{alignedat}\right.
 \label{e:wtq-L}
\end{equation}
where $\nabla_L=(\partial_{\bar{x}_1}, \partial_{\bar{x}_2}-t\partial_{\bar{x}_1}, \partial_{\bar{x}_3})$, $\Delta_{L}=\nabla_L\cdot \nabla_L$,  $\ublqnl(t, \bar{X})=\ulqnl(t, X)$, $\bar{F}_{q+1}(t, \bar{X})={F}_{q+1}(t, {X})$. Under $\bar{X}$-coordinates, we define $ \mathbb{P}_{H_L}$  the Helmholtz  projector onto divergence-free vector fields by $\mathbb{P}_{H_L}={\rm Id}-\frac{\nabla_L \Div_L}{\Delta_L}$. Then we obtain from \eqref{e:wtq-L} that
\begin{align}
  \bar{u}(t,\bar{X})=&e^{t\Delta_L}u_0+\int_0^t e^{(t-s)\Delta_L}\mathbb{P}_{H_L}(\nabla_L \cdot (\bar{u}\ootimes \bar{u})+\nabla_L \cdot(\bar{u}\ootimes\bar{v}))\dd s\nonumber\\
   &+\int_0^t e^{(t-s)\Delta_L}\mathbb{P}_{H_L}(\nabla_L \cdot(\bar{v}\ootimes\bar{u})+(\bar{u}_2, 0,0)-\bar{f})\dd s.\label{integral}
\end{align}
Note that for $0\le t\le \frac{1}{4}$,
\begin{align*}
   \widehat{e^{t\Delta_L}}=e^{-t((1+t^2)|\bar{\xi}_1|^2+|\bar{\xi}_2|^2+|\bar{\xi}_3|^2-2t\bar{\xi}_1\bar{\xi}_2)}\le e^{-\frac{3}{4}t(|\bar{\xi}_1|^2+|\bar{\xi}_2|^2+|\bar{\xi}_3|^2)},
\end{align*}
by the smoothing effect of heat flow  \cite[Lemma 2.4]{BCD11}, one infers that there exist constants $c$ and $C$ such that if $\spt_{\bar{\xi}}\,\widehat {u}\subseteq \lambda \mathcal {C}$, where $\mathcal {C}$ is an annulus, then
\begin{align*}
    \|e^{t\Delta_L} u\|_{L^p_{\bar{X}}}\le C e^{-ct\lambda^2}\|u\|_{L^p_{\bar{X}}}.
\end{align*}
Therefore, for $0< T\le \frac{1}{4}$, we deduce from Lemma \ref{heat} that
\begin{align}
   \|\bar{u}\|_{\widetilde L^\infty_t { \bf B}^{\frac{1}{2}}_{2,1}}\lesssim& \|u_0\|_{{\bf B}^{\frac{1}{2}}_{2,1}}+\|\bar{u}\ootimes \bar{u}\|_{\widetilde L^\infty_t {\bf B}^{-\frac{1}{2}}_{2,1}}+\|\bar{u}\ootimes\bar{v}\|_{\widetilde L^\infty_t {\bf B}^{-\frac{1}{2}}_{2,1}}\nonumber\\
   &+\|\bar{v}\ootimes\bar{u}\|_{\widetilde L^\infty_t {\bf B}^{-\frac{1}{2}}_{2,1}}+\|\bar{u}\|_{\widetilde L^1_t { \bf B}^{\frac{1}{2}}_{2,1}}+\|\bar{f}\|_{\widetilde L^\infty_t { \bf B}^{-\frac{3}{2}}_{2,1}}\nonumber\\
\lesssim&  \|u_0\|_{{\bf B}^{\frac{1}{2}}_{2,1}}+\|\bar{u}\|^2_{\widetilde L^\infty_t { \bf B}^{\frac{1}{2}}_{2,1}}+\|\bar{v}\|_{\widetilde L^\infty_t { \bf B}^{\frac{1}{2}}_{2,1}}\|\bar{u}\|_{\widetilde L^\infty_t { \bf B}^{\frac{1}{2}}_{2,1}}+\|\bar{u}\|_{\widetilde L^1_t { \bf B}^{\frac{1}{2}}_{2,1}}+\|\bar{f}\|_{\widetilde L^\infty_t { \bf B}^{-\frac{3}{2}}_{2,1}},\label{bw-B}
\end{align}
here
\begin{align*}
\|\bar{u}\|_{\widetilde{L}^r_T{ \bf B}^s_{p,q}}\overset{\text{def}}{=}\Big{\|}(2^{js}\|\bar{\Delta}_j \bar{u}\|_{L^r([0,T];L^p_{\bar{X}})})_{j\in\ZZ}\Big{\|}_{\ell^q(\ZZ)}<\infty,
\end{align*}
where $\bar{\Delta}_j $ is localization nonhomogeneous operator from the Littlewood-Paley decomposition theory under $\bar{X}$--coordinates. Then by the definition of $\bar{Y}$, we have for $j\ge 0$,
\begin{align*}
    \|\bar{\Delta}_j \bar{u}\|_{L^2_{\bar{X}}}=&\Big{\|}\int h_j(\bar{X}-\bar{Y})\bar{u}(t,\bar{Y})\dd \bar{Y}\Big{\|}_{L^2_{\bar{X}}}\\
    =&\Big{\|}\int h_j(x_1-y_1-t(x_2-y_2), x_2-y_2, x_3-y_3))u(t,{Y})\dd {Y}\Big{\|}_{L^2_{{X}}}\\
    =& \Big{\|}\mathcal{F}\Big(\int h_j(x_1-y_1-t(x_2-y_2), x_2-y_2, x_3-y_3))u(t,{Y})\dd {Y}\Big)(\xi)\Big{\|}_{L^2_{{\xi}}}.
\end{align*}
By the definition of $\bar{X}$ and  $\bar{Y}$, we have
\begin{align*}
 &\mathcal{F}\Big(\int h_j(x_1-y_1-t(x_2-y_2), x_2-y_2, x_3-y_3))u(t,{Y})\dd {Y}\Big)(\xi)\\
 =&\int_{\R^3}\int_{\R^3}   h_j(x_1-y_1-t(x_2-y_2), x_2-y_2, x_3-y_3))u(t,{Y})\dd {Y} e^{-\ii 2\pi X\cdot\xi}\dd X\\
 =&\int_{\R^3}\int_{\R^3} h_j(\bar{X}-\bar{Y}) u(t,\bar{y}_1+t\bar{y}_2,\bar{y}_2,\bar{y}_3) e^{-\ii 2\pi\bar{X}\cdot(\xi_1, t\xi_1+\xi_2, \xi_3)}\dd \bar{Y}\dd \bar{X}\\
 =&\int_{\R^3}\int_{\R^3} h_j(\bar{X}-\bar{Y})e^{-\ii 2\pi (\bar{X}-\bar{Y})\cdot(\xi_1, t\xi_1+\xi_2, \xi_3)} u(t,\bar{y}_1+t\bar{y}_2,\bar{y}_2,\bar{y}_3) e^{-\ii 2\pi \bar{Y}\cdot (\xi_1, t\xi_1+\xi_2, \xi_3)}\dd \bar{X} \dd \bar{Y}\\
 =&{\varphi}_j(\xi_1, t\xi_1+\xi_2,\xi_3)\widehat{u}(t,\xi).
\end{align*}
Since $\spt\, {\varphi}_j(\xi)\subseteq 2^j\mathcal{C}(0,\frac{3}{4}, \frac{8}{3})$ with $\mathcal{C}(0,\frac{3}{4}, \frac{8}{3})=\{\xi\in\R^3| \frac{3}{4}\le |\xi|\le\frac{8}{3}\}$ and $\spt\, \widehat{\varphi}_j(\xi)\equiv 1$ for $\xi\in 2^j\mathcal{C}(0,\frac{6}{7}, \frac{12}{7})$, then for $t\le \frac{1}{4}$,
\begin{align*}
 2^j \mathcal{C}\big(0,\tfrac{9}{14}, \tfrac{9}{7}\big)\subseteq\spt\,{\varphi}_j(\xi_1, t\xi_1+\xi_2,\xi_3)\subseteq 2^j \mathcal{C}\big(0,\tfrac{9}{16}, \tfrac{7}{2}\big).
\end{align*}
Similarly,  one deduces that
\begin{align*}
  \|\bar{\Delta}_{-1} \bar{u}\|_{L^2_{\bar{X}}}=\|\chi(\xi_1, t\xi_1+\xi_2,\xi_3)\widehat{u}(t,\xi)\|_{L^2_{\xi}}, \quad B(0,\tfrac{9}{16})\subseteq\,\spt \chi(\xi_1, t\xi_1+\xi_2,\xi_3)\subseteq B(0, \tfrac{7}{4}).
\end{align*}
These facts imply that there exist another two  dyadic partition of unity $\{\widetilde{\chi}(\xi), \widetilde{\varphi}(\xi)\}$ and $\{\widetilde{\widetilde{\chi}}(\xi), \widetilde{\widetilde{\varphi}}(\xi)\}$ such that
\begin{align*}
    &\|\widetilde{\Delta}_j u\|_{L^2_X} \lesssim  \|\bar{\Delta}_j \bar{u}\|_{L^2_{\bar{X}}}\lesssim \|\widetilde{\widetilde{\Delta}}_j u\|_{L^2_X}, \quad\|\widetilde{\Delta}_{-1} u\|_{L^2_X} \lesssim  \|\bar{\Delta}_{-1} \bar{u}\|_{L^2_{\bar{X}}}\lesssim \|\widetilde{\widetilde{\Delta}}_{-1}u\|_{L^2_X}
\end{align*}
Therefore $\|\bar{u}\|_{_{\widetilde{L}^r_T{ \bf B}^s_{p,q}}}\sim\|{u}\|_{_{\widetilde{L}^r_T{ B}^s_{p,q}}}$. Hence we infers from \eqref{con-S} and \eqref{bw-B}$_{t=T}$ that
\begin{align}
   \|{u}\|_{\widetilde L^\infty_T { B}^{\frac{1}{2}}_{2,1}}
\lesssim& \|u_0\|_{{B}^{-\frac{1}{2}}_{2,1}}+ \|{u}\|^2_{\widetilde L^\infty_T { B}^{\frac{1}{2}}_{2,1}}+
\varepsilon\|{u}\|_{\widetilde L^\infty_T { B}^{\frac{1}{2}}_{2,1}}+T\|{u}\|_{\widetilde L^\infty_T { B}^{\frac{1}{2}}_{2,1}}+\|{f}\|_{\widetilde L^\infty_T { B}^{-\frac{3}{2}}_{2,1}}.\nonumber
\end{align}
Therefore, as long as $\varepsilon$ and $T$ are small enough, $u$ solves \eqref{e:Coutee-NS} on $[0, T]$ with
\begin{align*}
   \|u\|_{\widetilde L^\infty_T { B}^{\frac{1}{2}}_{2,1}}\lesssim \|u_0\|_{{ B}^{-\frac{1}{2}}_{2,1}}+ \|f\|_{\widetilde L^\infty_T {  B}^{-\frac{3}{2}}_{2,1}}.
\end{align*}
Hence, we complete this proof.
\end{proof}


\section*{Acknowledgement}
The authors are very grateful to Alexey Cheskidov for valuable suggestions. This work was supported by the National Key Research and Development Program of China (No. 2022YFA1005700).

\end{document}